\documentclass[11pt,twoside,final]{article}
\usepackage{amsmath,amsthm,amssymb,bm}
\usepackage[reset,a4paper,vmargin=3truecm,hmargin=2truecm]{geometry}

\numberwithin{equation}{section}
\allowdisplaybreaks

\DeclareMathAlphabet\matheu{U}{eur}{m}{n}
\DeclareMathAlphabet\matheuscr{U}{eus}{m}{n}
\DeclareMathAlphabet\mathscr{U}{rsfs}{m}{n}

\newcommand{\MSC}[2]{\medskip\noindent
\textbf{2010 Mathematics Subject Classification:} \textit{Primary} #1, \textit{Secondary} #2.}
\newcommand{\KeyWords}[1]{\medskip\noindent{\bfseries Keywords and phrases:} {#1}}
\newcommand{\Kakenhi}[2]{Grant-in-Aid for #1 No.\,#2}
\newcommand{\CREST}{JST CREST Grant Number JPMJCR14D6, Japan}

\newenvironment{ackn}{\subsection*{Acknowledgement}}{}

\newcommand{\eqed}{\pushQED{\qed}\qedhere\popQED}

\theoremstyle{plain}
\newtheorem{thm}{Theorem}[section]
\newtheorem{lem}[thm]{Lemma}
\newtheorem{prop}[thm]{Proposition}
\newtheorem{cor}[thm]{Corollary}
\newtheorem{conjecture}[thm]{Conjecture}
\newtheorem{lemma}[thm]{Lemma}

\theoremstyle{definition}
\newtheorem{dfn}{Definition}[section]
\newtheorem{ex}{Example}[section]
\newtheorem{prob}{Problem}[section]

\theoremstyle{remark}
\newtheorem{rem}{Remark}[section]
\newtheorem{remark}[rem]{Remark}

\newcommand{\N}{\mathbb{N}}
\newcommand{\Z}{\mathbb{Z}}
\newcommand{\Q}{\mathbb{Q}}
\newcommand{\R}{\mathbb{R}}
\newcommand{\C}{\mathbb{C}}

\newcommand{\M}{\mathbb{M}}

\newcommand{\X}{\mathfrak{X}}
\newcommand{\Holo}{\mathcal{H}}
\newcommand{\Mero}{\mathcal{M}}

\newcommand{\cZ}[1][m]{{Z}_{[#1]}}
\newcommand{\cB}[1][m]{{B}_{[#1]}}
\newcommand{\cH}[1][m]{{H}_{[#1]}}
\newcommand{\pZ}[1][m]{\widetilde{Z}_{[#1]}}
\newcommand{\pB}[1][m]{\widetilde{B}_{[#1]}}
\newcommand{\pH}[1][m]{\widetilde{H}_{[#1]}}
\newcommand{\pC}{\widetilde{C}}

\newcommand{\el}{l} 

\newcommand{\I}{\sqrt{-1}}

\newcommand{\jth}[2]{{\underset{\text{\sffamily$#1$-th}}{#2}}}

\newcommand{\kakko}[1]{\left(#1\right)}
\newcommand{\ckakko}[1]{\left\{#1\right\}}
\newcommand{\abs}[1]{\left\lvert#1\right\rvert}
\newcommand{\card}[1]{\##1}

\newcommand{\set}[3][big]
{\csname#1l\endcsname\{#2\,\csname#1\endcsname\vert\,#3\csname#1r\endcsname\}}
\newcommand{\Set}[2]{\left\{#1\,\middle\vert\,#2\right\}}

\DeclareMathOperator{\tr}{tr}
\DeclareMathOperator{\Tr}{Tr}
\DeclareMathOperator{\diag}{diag}
\DeclareMathOperator{\sgn}{sgn}
\DeclareMathOperator{\Sym}{Sym}
\DeclareMathOperator*{\Res}{Res}
\DeclareMathOperator{\Map}{Map}
\DeclareMathOperator{\im}{im}
\DeclareMathOperator{\Mat}{Mat}

\newcommand{\deq}{:=}

\newcommand{\B}{\boldsymbol{B}}
\newcommand{\E}{\boldsymbol{E}}
\newcommand{\F}{\boldsymbol{F}}

\newcommand{\va}{\boldsymbol{a}}
\newcommand{\vb}{\boldsymbol{b}}
\newcommand{\vd}{\boldsymbol{d}}
\newcommand{\ve}{\boldsymbol{e}}
\newcommand{\vi}{\boldsymbol{i}}
\newcommand{\vj}{\boldsymbol{j}}
\newcommand{\vo}{\boldsymbol{o}}
\newcommand{\vt}{\boldsymbol{t}}
\newcommand{\vu}{\boldsymbol{u}}
\newcommand{\vv}{\boldsymbol{v}}
\newcommand{\vx}{\boldsymbol{x}}

\newcommand{\symx}[1]{\operatorname{Sym}^{\times}_{#1}}

\newcommand{\esum}[1]{\left\|#1\right\|}

\newcommand{\cT}[1]{\mathcal{T}_{#1}}
\newcommand{\U}[1]{\boldsymbol{U}_{#1}}
\newcommand{\V}[1]{\boldsymbol{V}_{#1}}
\newcommand{\den}[1]{\mathcal{W}_{#1}}

\newcommand{\hA}[2]{\frac{1-u_{#1}^4u_{#2}^4}{(1-u_{#1}^4)(1-u_{#2}^4)}}
\newcommand{\hB}[1]{\frac{-u_{#1}^2}{1-u_{#1}^4}}

\newcommand{\e}{\varepsilon}
\newcommand{\conj}[1]{\overline{#1}}
\newcommand{\sym}[1]{\mathfrak{S}_{#1}}
\newcommand{\cyclic}[1]{\mathcal{C}_{#1}}

\newcommand{\hgf}[5]{{}_{#1}F_{#2}\!\left(#3;#4;#5\right)} 

\newcommand{\J}[2]{J_{#1}(#2)} 
\newcommand{\tJ}[2]{\tilde{J}_{#1}(#2)} 
\newcommand{\T}[3]{T_{#1,#2}(#3)}
\newcommand{\tempJ}[3]{J_{#2,#1-#2}(#3)}

\newcommand{\eqsp}{\phantom{{}={}}}

\newcommand{\generators}[1]{\left\langle#1\right\rangle}
\newcommand{\uhp}{\mathfrak{h}} 

\newcommand{\dG}[2][{}]{\mathrm{dG}_{#2}^{#1}} 
\newcommand{\dE}[2][{}]{\mathrm{dE}_{#2}^{#1}} 

\newcommand{\omegaul}{\underline{\omega}} 
\newcommand{\psum}{\sideset{}{'}{\sum}} 

\newcommand{\mat}[1]{\begin{pmatrix}#1\end{pmatrix}}

\newcommand{\iE}{\bm{E}}
\newcommand{\iG}{\bm{G}}
\newcommand{\iGt}{\widetilde{\iG}}

\newcommand{\der}{\partial}

\newcommand{\ds}[2]{\frac{\partial}{\partial s}#1\bigg|_{s=#2}} 

\newcommand{\q}{\kappa}
\newcommand{\cobd}{\delta} 
\newcommand{\bgamma}{\underline{\gamma}}
\newcommand{\truncate}{\mathcal{T}}
\newcommand{\contract}{\mathcal{C}}

\newcommand{\hecke}{\mathfrak{G}(2)}
\newcommand{\congsubgp}{\Gamma}
\newcommand{\suchthat}{\quad\text{s.t.}\quad}

\newcommand{\ges}[3]{G^{(#1;#2,#3)}} 
\newcommand{\Torus}{\mathbb{T}^2}
\newcommand{\data}[2]{\noindent\qquad$\displaystyle\tJ#1n\ \colon\ #2$ \bigskip}
\newcommand{\cf}[2]{A_{#1,#2}}

\newcommand{\per}{R}

\newcommand{\liesl}{\mathfrak{sl}}

\newcommand{\evenZ}{Z^{\text{\upshape even}}}
\newcommand{\oddZ}{Z^{\text{\upshape odd}}}
\newcommand{\evenY}{Y^{\text{\upshape even}}}
\newcommand{\oddY}{Y^{\text{\upshape odd}}}
\newcommand{\tA}{\widetilde A}
\newcommand{\tB}{\widetilde B}
\newcommand{\then}{\DOTSB\;\Longrightarrow\;}
\DeclareMathOperator{\ord}{ord}
\newcommand{\floor}[1]{\left\lfloor#1\right\rfloor}
\newcommand{\fracpart}[1]{\left\{#1\right\}}

\title{Ap\'ery-like numbers for non-commutative harmonic oscillators and automorphic integrals}
\author{Kazufumi Kimoto%
\thanks{Partially supported by \Kakenhi{Scientific Research (C)}{18K03248}, JSPS and by \CREST.}
\, and \,
Masato Wakayama%
\thanks{Partially supported by \Kakenhi{Scientific Research (C)}{16K05063}, JSPS and by \CREST.}
}
\begin{document}
\maketitle

\begin{abstract}
The purpose of the present paper is to study the number theoretic properties of the special values of the spectral zeta functions of the non-commutative harmonic oscillator (NcHO), especially in relation to modular forms and elliptic curves from the viewpoint of Fuchsian differential equations, and deepen the understanding of the spectrum of the NcHO. We study first the general expression of special values of the spectral zeta function $\zeta_Q(s)$ of the NcHO at $s=n$ $(n=2,3,\dots)$ and then the generating and meta-generating functions for \emph{Ap\'ery-like numbers} defined through the analysis of special values $\zeta_Q(n)$. Actually, we show that the generating function $w_{2n}$ of such Ap\'ery-like numbers appearing (as the ``first anomaly'') in $\zeta_Q(2n)$ for $n=2$ gives an example of \emph{automorphic integral with rational period functions} in the sense of Knopp,
but still a better explanation remains to be clarified explicitly for $n>2$. This is a generalization of our earlier result on showing that $w_2$ is interpreted as a $\Gamma(2)$-modular form of weight $1$. Moreover, certain congruence relations over primes for ``normalized" Ap\'ery-like numbers are also proven. In order to describe $w_{2n}$ in a similar manner as $w_2$, we introduce a \emph{differential Eisenstein series} by using analytic continuation of a classical generalized Eisenstein series due to Berndt. The differential Eisenstein series is actually a typical example of the automorphic integral of negative weight. We then have an explicit expression of $w_4$ in terms of the differential Eisenstein series. We discuss also shortly the Hecke operators acting on such automorphic integrals and relating Eichler's cohomology group.

\MSC{11M41}{11A07, 33C20}

\end{abstract}

\setcounter{tocdepth}{1} 
\tableofcontents

\section{Introduction}\label{Int}

Let $Q$ be a parity preserving matrix valued ordinary differential operator defined by
\begin{equation*}
Q=Q_{\alpha,\beta}
=\begin{pmatrix}\alpha & 0 \\ 0 & \beta\end{pmatrix}\kakko{-\frac12\frac{d^2}{dx^2}+\frac12x^2}
+\begin{pmatrix}0 & -1 \\ 1 & 0\end{pmatrix}\kakko{x\frac{d}{dx}+\frac12}.
\end{equation*}
The system defined by $Q$ is called the \emph{non-commutative harmonic oscillator} (NcHO),
which was introduced in \cite{PW2001, PW2002}
(see also \cite{P, P2014} for references therein and for recent progress).
Throughout the paper,
we always assume that $\alpha,\beta>0$ and $\alpha\beta>1$.
Under this assumption, the operator $Q$ becomes a positive self-adjoint unbounded operator on $L^2(\R;\C^2)$,
the space of $\C^2$-valued square-integrable functions on $\R$,
and $Q$ has only a discrete spectrum with uniformly bounded multiplicity:
\begin{equation*}
0<\lambda_1\le\lambda_2\le\lambda_3\le\dots(\nearrow\infty).
\end{equation*}
It was proved recently that the lowest eigenstate is multiplicity free \cite{HS2013B} and also the multiplicity of general eigenstate is less than or equal to $2$ \cite{W2013} (see \cite{ W2013-2} for the proof).

The aim of the present paper is to advance a number theoretic study of the spectrum of the NcHO through observing special values of the spectral zeta function $\zeta_Q(s)$ (\cite{IW2005a, IW2005KJM}) defined below, and further to deepen the understanding of the spectrum:
$$
\zeta_Q(s):=\sum_{n=1}^\infty \lambda_n^{-s} \quad (\Re(s)>1).
$$
It is noted that, when $\alpha=\beta$, $Q=Q_{\alpha,\alpha}$ is unitarily equivalent to the couple of quantum harmonic oscillators, whence the eigenvalues are easily calculated as
$\set{\sqrt{\alpha^2-1}\big(n+\frac12\big)}{n\in \Z_{\geq0}}$
having multiplicity 2.
Actually, when $\alpha=\beta$, behind $Q$, there exists a structure corresponding to the tensor product of the $2$-dimensional trivial representation and the oscillator representation (see e.g.\ \cite{HT1992}) of the Lie algebra $\liesl_2$. Namely, in this case, $\zeta_Q(s)$ is essentially given by the Riemann zeta function $\zeta(s)$ as
$\zeta_Q(s)=2(2^s-1)\sqrt{\alpha^2-1}\zeta(s)$. In other words, $\zeta_Q(s)$ is a $\frac{\alpha}{\beta}$-analogue of $\zeta(s)$.
The clarification of the spectrum in the general $\alpha\not=\beta$ case is, however, considered to be highly non-trivial.
Indeed, while the spectrum is described theoretically by using certain continued fractions \cite{PW2002} and also by Heun's ordinary differential equations (those have four regular singular points) \cite{Heun2008} in a certain complex domain \cite{O2001CMP, W2013-2}, almost no satisfactory information on each eigenvalue is available in reality when $\alpha\ne\beta$ (see \cite{P} and references therein).

It is nevertheless worth mentioning that, in recent years, special attention has been paid to studying the spectrum of self-adjoint operators with non-commutative coefficients, like the Jaynes-Cummings model, the quantum Rabi model and its generalized version, etc., not only in mathematics but also in theoretical/experimental physics (see e.g.\ \cite{HR2008, B2011PRL, BCBS2016JPA, Y-S2018} and references therein).
The NcHO has been expected similarly to provide one of these Hamiltonians describing such quantum interacting systems, i.e.\ a Hamiltonian describing such an interaction between photons and atoms.
Although it does not seem to be expected, it has been shown in \cite{W2013-2} that (the ``Heun picture" of) the quantum Rabi model can be obtained by the second order element of the universal enveloping algebra $U(\liesl_2)$ naturally arising from the NcHO through the oscillator representation. It is, in fact, caught by taking particular parameters and considering general \emph{confluence procedure, i.e.\ confluence of two singular points} in Heun's ordinary differential equation obtained in the action of the non-unitary principal series representation of $\liesl_2$.

Therefore, in place of hunting each eigenvalue of $Q$, it is significant to study the spectral zeta function $\zeta_Q(s)$ of the NcHOs as a sort of generating function of the eigenvalues. From the physical point of view, $\zeta_Q(s)$ is also regarded as the Mellin transform of the partition function of the system defined by the NcHO.
This paper discusses the number theoretic properties of the special values of $\zeta_Q(s)$ at integer points. We notice that special values are considered as moments of the partition functions.
We have actually studied congruence properties of the \emph{Ap\'ery-like numbers} in \cite{KW2006KJM} that have arisen naturally from the special values $\zeta_Q(2)$ at $s=2$ by the same idea guided in the studies for the Ap\'ery numbers for $\zeta(2)$ in \cite{Beu1985} (and references therein). This study of congruence properties led us further to show that the generating function $w_2$ of the Ap\'ery-like numbers for $\zeta_Q(2)$ is interpreted as a $\Gamma(2)$-modular form of weight $1$ \cite{KW2007} in the same way as in a pioneering study by Beukers \cite{Beu1983, Beu1987} for the Ap\'ery numbers. In other words, the recurrence equation of these Ap\'ery-like numbers defined in \cite{KW2006KJM} provide one of the particular examples listed in Zagier \cite{Z} (\#19)\footnote{Although the terminology ``\emph{Ap\'ery-like}'' is the identical, the usage/definition of the name in the current paper is different from the one in the title of \cite{Z}.}.
Moreover, it is known in \cite{KY2009} that the Ap\'ery-like numbers corresponding to $\zeta_Q(2)$ are described by a finite convolution of the Hurwitz zeta function and certain variation of multiple $L$-values.
Also, recently, certain nice congruence relations among these Ap\'ery-like numbers that are quite resembled to the Rodriguez-Villegas type congruence \cite{Mortenson2003JNT} and conjectured in \cite{KW2006KJM} are proved in \cite{LOS2014}. Further interesting congruence that involves Bernoulli numbers has been obtained in \cite{L2018} (see also \cite{Sun2019}). The congruence in \cite{L2018} can be considered as a one step deep congruence of the one proved in \cite{LOS2014} corrected by the remainder term.

It is hard in general to obtain the precise information of the higher special values of $\zeta_Q(n)$ $(n >2)$
as the same level of $\zeta_Q(2)$. Thus, in this paper we introduce the Ap\'ery-like numbers $J_k(n)$ $(k=0,1,2,\ldots)$ for each $n$ defined through the \emph{first anomaly} of $\zeta_Q(n)$ $(n >2)$. These Ap\'ery-like numbers share the properties of the one for $\zeta_Q(2)$, e.g.\ satisfy a similar recurrence relation as in the case of $\zeta_Q(2)$ and hence the ordinary differential equation satisfied by the generating function follows from the recurrence relation. Remarkably, each of the homogeneous part of those differential equations is identified to be a ($n$ dependent) power of the homogeneous part of the one corresponding to $\zeta_Q(2)$.
Further, we observe that the meta-generating functions of Ap\'ery-like numbers are described explicitly by
the modular Mahler measures studied by Rodriguez-Villegas in \cite{RV1999}.
Through this relation, we may expect to discuss an interesting aspect of a discrete dynamical system behind the NcHO defined by some group via (weighted) Cayley graphs (see \cite{DL2009}, also e.g.\ \cite{LSW1990}) in the future.
Moreover, we show that the generating function $w_{2n}$ of Ap\'ery-like numbers corresponding to the first anomaly in $\zeta_Q(2n)$ when $n=2$ is given by an automorphic integral with a rational period function in the sense of Knopp \cite{Knopp1978Duke}. This is obviously a generalization of our earlier result \cite{KW2007} showing that $w_2$ is interpreted as a $\Gamma(2)$-modular form of weight $1$. However, it is still unclear whether there is a similar explicit (geometric and algebraic) interpretation in general for $\zeta_Q(n)$ ($n>2$). Further, the study of the special values of the spectral zeta function for the quantum Rabi model \cite{S2016NMJ} and comparison to one for NcHO is a quite interesting future problem as NcHO is a ``covering'' of the model.

The organization of the paper is as follows:
In \S\ref{SV} we calculate (Theorem \ref{thm:specialvalues}) the special values of the spectral zeta function for the NcHO. These explicit formulas are referred already in \cite{KW2012RIMS} (see \cite{K2010CCRH}) by multiple integrals like (a generalization of) the original Ap\'ery cases for $\zeta(2)$ and $\zeta(3)$ using Legendre functions \cite{Beu1983, BP1984}.
The basic idea is on the same line as \cite{IW2005KJM} but some essentially new techniques are explored.

In \S\ref{ALN} we derive the recursion formula for the Ap\'ery-like numbers associated to the first anomalies of special values of $\zeta_Q(s)$ and the differential equations satisfied by the generating functions of such Ap\'ery-like numbers. Although our study is very much influenced by the classical (algebro-geometric) work on Ap\'ery numbers in \cite{Beu1983, Beu1987, BP1984} and its subsequent developments, since the family of generating functions for Ap\'ery-like numbers arising via the NcHO possesses a remarkable hierarchical structure, there is a decisive difference between these two.
We then define the \emph{normalized} Ap\'ery-like numbers which are shown to be rational numbers, and present a numerical data of these numbers. In the end of this section \S \ref{CongruenceProperty}, we give a certain conjecture (Conjecture \ref{GeneralCongruence}) for the congruence among those normalized Ap\'ery-like numbers which are the generalization of the results in \cite{KW2006KJM} based on numerical experiments.  We can only show in this paper a weaker/partial result in Theorem \ref{weaker version}, which may be considered as a version of the classical Kummer congruence for the special values at negative odd integer points of $\zeta(s)$.
We remark that, however, it is quite difficult to expect an exact generalization of the congruence relation (i.e. of the same shape which is relevant to the hypergeometric series) shown by employing $p$-adic analysis in \cite{LOS2014} (and \cite{L2018}) for $\zeta(2)$.

We study in \S\ref{MGF} also meta-generating functions for Ap\'ery-like numbers in relation to the study on modular Mahler measures in \cite{RV1999}.
In \S\ref{Modular}, we first recall briefly the modular form interpretation of the generating function for the Ap\'ery-like numbers for $\zeta_Q(2)$ from \cite{KW2007} and discuss the corresponding generating function $w_{2n}$ for the Ap\'ery-like numbers for (the first anomaly in) $\zeta_Q(2n)$. We may also study the Ap\'ery numbers associated with $\zeta_Q(2n+1)$ but the structure behind this is different from the one in \cite{BP1984} that is relating with $K3$ surfaces. Actually, although the homogeneous part of the differential equation satisfied by the Ap\'ery-like numbers arisen from odd special values are the same as the even case, even the $\zeta_Q(3)$ can not be interpreted as a picture of $K3$ spaces.
We recall then in \S\ref{Modular} a notion of \emph{automorphic integrals with rational period functions} in the sense of Knopp \cite{Knopp1978Duke} (that is a slightly generalized notion of the automorphic integrals \cite{G1961}). Then we study $w_{2n}$ from the viewpoint of Fuchsian differential equations. Indeed, we show that $w_{2n}$ can be expressed by the linear space spanned by higher derivatives of automorphic integrals and $w_2$. In other words, we observe that $w_{2k}$ is obtained by some linear combination of the multiple integral of the (same) modular forms. For instance, the explicit expression of $w_{6}$ by such a linear span of integrals is given in \S5.5.
In order to describe $w_{2n}$ in a similar manner as $w_2$, it is necessary to introduce a \emph{differential Eisenstein series} by using analytic continuation of a classical generalized Eisenstein series due to Berndt
\cite{B1975Cr} in \S\ref{EF}. These differential Eisenstein series provide typical examples of the automorphic integral of negative weight and we have an explicit expression of $w_4$ in terms of the differential Eisenstein series. We notice that the differential Eisenstein series is periodic, whence has a Fourier expansion at the infinity.
Further, we discuss shortly the Hecke operators acting on such automorphic integrals and compute the associated $L$-function of the differential Eisenstein series (which has an Euler product). In the final section \S\ref{Eichler}, we discuss briefly the Eichler cohomology groups relevant to the periodic automorphic integrals.
A part of ideas of the paper has been discussed in our proceedings paper \cite{KW2012RIMS}, but there is a certain misleading terminology \cite{KW2012RIMS} so that we will fix those in this paper\footnote{The general definition of ``residual modular forms'' in \cite{KW2012RIMS} is too demanded. Although the example given in \cite{KW2012RIMS} satisfies such strong condition in the definition, if the level $N$ is large, i.e.\ the number of inequivalent cusps is increasing, the definition of residual modular forms allows only the zero form. In this paper, we find actually that the notion of the automorphic integrals in the sense of \cite{Knopp1978Duke} is sufficient for our study.}.

\section{Special values of the spectral zeta function}\label{SV}

From the sequence of the eigenvalues
$0<\lambda_1\le\lambda_2\le\lambda_3\le\dots(\to\infty)$ of $Q$, we define the spectral zeta function of $Q$ by the Dirichlet series
\begin{equation*}
\zeta_Q(s)=\sum_{n=1}^\infty\frac1{\lambda_n^s}.
\end{equation*}
This series is absolutely convergent and defines a holomorphic function in $s$ in the region $\Re s>1$.
We call this function $\zeta_Q(s)$ the \emph{spectral zeta function} for the non-commutative harmonic oscillator $Q$ \cite{IW2005a}.
The zeta function $\zeta_Q(s)$ is analytically continued to the whole complex plane $\C$ as a single-valued meromorphic function which is holomorphic except for the simple pole at $s=1$.
It is notable that $\zeta_Q(s)$ has `trivial zeros' at $s=0,-2,-4,\dots$ from the presence of $\Gamma(s)^{-1}$ at the analytic continuation to the whole complex plane \cite{IW2005a}.
When the two parameters $\alpha$ and $\beta$ are equal, then $\zeta_Q(s)$ essentially gives the Riemann zeta function $\zeta(s)$ (see Remark \ref{rem:degenerate_case}).

We are interested in the special values of $\zeta_Q(s)$, that is, the values $\zeta_Q(s)$ at $s=2,3,4,\dots$.
In \cite{IW2005KJM} the first two special values are calculated as
\begin{align*}
\zeta_Q(2)&=2\kakko{\frac{\alpha+\beta}{2\sqrt{\alpha\beta(\alpha\beta-1)}}}^{\!\!2}\\
&\quad\eqsp\times\Biggl(\zeta(2,1/2)+\kakko{\frac{\alpha-\beta}{\alpha+\beta}}^{\!\!2}
\int_{[0,1]^2}\frac{4du_1du_2}{\sqrt{(1-u_1^2u_2^2)^2+(1-u_1^4)(1-u_2^4)/(\alpha\beta-1)}}\Biggr),\\
\zeta_Q(3)&=2\kakko{\frac{\alpha+\beta}{2\sqrt{\alpha\beta(\alpha\beta-1)}}}^{\!\!3}\\
&\quad\eqsp\times\Biggl(\zeta(3,1/2)+3\kakko{\frac{\alpha-\beta}{\alpha+\beta}}^{\!\!2}
\int_{[0,1]^3}\frac{8du_1du_2du_3}{\sqrt{(1-u_1^2u_2^2u_3^2)^2+(1-u_1^4)(1-u_2^4u_3^4)/(\alpha\beta-1)}}\Biggr),
\end{align*}
where $\zeta(s,x)=\sum_{n=0}^\infty(n+x)^{-s}$ is the Hurwitz zeta function.
These values are also given by the contour integral expressions using a solution of a certain Fuchsian differential equation.
Later, in \cite{O2008RJ} Ochiai gave an expression of $\zeta_Q(2)$ using the complete elliptic integral or the hypergeometric function, and the present authors \cite{KW2006KJM} gave a similar formula for $\zeta_Q(3)$.

In this section, we present an explicit calculation of the special values of the spectral zeta function $\zeta_Q(k)$ of the non-commutative harmonic oscillator $Q$ for all positive integers $k>1$, and express them in terms of integrals of certain algebraic functions (see Theorem \ref{thm:specialvalues} for the formula).

\subsection{Preliminaries for calculating special values}

Following to the method in \cite{IW2005KJM}, we first explain how to calculate the special values of $\zeta_Q(s)$.

Put
\begin{equation*}
\varepsilon\deq\frac1{\sqrt{\alpha\beta}},\quad
\q\deq\varepsilon(1-\varepsilon^2)^{-1/2}=\frac1{\sqrt{\alpha\beta-1}}
\end{equation*}
and
\begin{equation*}
A\deq\begin{pmatrix}
\alpha & 0 \\ 0 & \beta
\end{pmatrix},\quad
I\deq\begin{pmatrix}
1 & 0 \\ 0 & 1
\end{pmatrix},\quad
J\deq\begin{pmatrix}
0 & -1 \\ 1 & 0
\end{pmatrix}.
\end{equation*}
Notice that $0<\e<1$ and $\q>0$.
Since it is difficult to find the heat kernel of the NcHO
\begin{equation*}
Q=\frac12(-\partial_x^2+x^2)A+\Bigl(x\partial_x+\frac12\Bigr)J,
\end{equation*}
we look at a slightly modified one
\begin{equation*}
Q'=A^{-1/2}QA^{-1/2}=\frac12(-\partial_x^2+x^2)+\e J\Bigl(x\partial_x+\frac12\Bigr),
\end{equation*}
whose heat kernel is explicitly obtained as we see below.

The heat kernel of the usual quantum harmonic oscillator is known as the Mehler kernel and is given by
\begin{equation*}
p(t,x,y)=\pi^{-1/2}e^{-t/2}(1-e^{-2t})^{-1/2}\exp\Bigl(-\frac{x^2-y^2}2-\frac{(e^{-t}x-y)^2}{1-e^{-2t}}\Bigr).
\end{equation*}
Namely, $p(t,x,y)$ satisfies
\begin{equation*}
-\partial_tp(t,x,y)=\frac12(-\partial_x^2+x^2)p(t,x,y),\qquad
p(t,x,y)\to\delta(x-y)
\quad(t\downarrow0).
\end{equation*}

Put $q(t,x,y)=(1-\e^2)^{1/4}p((1-\e^2)^{1/2}t,(1-\e^2)^{1/4}x,(1-\e^2)^{1/4}y)$.
Then
\begin{equation*}
-\partial_tq(t,x,y)=\frac12(-\partial_x^2+(1-\e^2)x^2)q(t,x,y),\qquad
q(t,x,y)\to\delta(x-y)
\quad(t\downarrow0).
\end{equation*}
Define
\begin{equation*}
K'(t,x,y)=q(t,x,y)\exp\Bigl(\frac{\e(x^2-y^2)}2J\Bigr).
\end{equation*}
We see that
\begin{equation*}
-\partial_tK'(t,x,y)=\frac12(-\partial_x^2+x^2)K'(t,x,y)+\e J\Bigl(x\partial_x+\frac12\Bigr)K'(t,x,y),
\end{equation*}
\begin{equation*}
K'(t,x,y)\to\delta(x-y)I
\quad(t\downarrow0),
\end{equation*}
which implies that $K'(t,x,y)$ is the heat kernel of $Q'$ (see \cite{IW2005KJM} for detail\footnote{%
There is a typo in (2.11b) of \cite{IW2005KJM}.
The right equation should be
\begin{equation*}
\partial_tp_\gamma(t,x,y)=-\frac12[\partial_x^2+(1-\gamma^2)x^2]p_\gamma(t,x,y),
\end{equation*}
in which the coefficient of $x^2$ is replaced from $(1-\gamma^2)^{1/2}$ to $1-\gamma^2$.
The result itself is, however, correct.}).
Hence the integral kernel $Q^{-1}(x,y)$ of $Q^{-1}$ is
\begin{align*}
Q^{-1}(x,y)
&=\int_0^\infty A^{-1/2}K'(t,x,y)A^{-1/2}\,dt \\
&=\int_0^1 A^{1/2}K(u,(1-\e^2)^{1/4}x,(1-\e^2)^{1/4}y)A^{-1/2}\,du\qquad(u=e^{-\frac t2})
\end{align*}
since $Q^{-1}=A^{-1/2}{Q'}^{-1}A^{-1/2}$, where we put
\begin{align*}
K(u,x,y)&\deq2\pi^{-1/2}(1-\e^2)^{-1/4}(1-u^4)^{-1/2}E(u,x,y)B(x,y), \\
E(u,x,y)&\deq\exp\biggl(-\begin{pmatrix} x & y \end{pmatrix}
\begin{pmatrix}
\frac{1+u^4}{2(1-u^4)} & \frac{-u^2}{1-u^4} \\
\frac{-u^2}{1-u^4} & \frac{1+u^4}{2(1-u^4)}
\end{pmatrix}
\begin{pmatrix} x \\ y \end{pmatrix}
\biggr), \\
B(x,y)&\deq A^{-1}\exp\frac{\q(x^2-y^2)}2J.
\end{align*}
Furthermore, we introduce the following functions
\begin{align*}
\B(x_1,\dots,x_k)&\deq\tr\kakko{B(x_1,x_2)B(x_2,x_3)\dots B(x_k,x_1)},\\[0.5em]
\E(u_1,\dots,u_k;x_1,\dots,x_k)&\deq E(u_1,x_1,x_2)E(u_2,x_2,x_3)\dots E(u_k,x_k,x_1),\\
\F(u_1,\dots,u_k)&\deq\int_{\R^k}\E(u_1,\dots,u_k;x_1,\dots,x_k)\B(x_1,\dots,x_k)dx_1\dots dx_k,
\end{align*}
where the symbol $\tr$ represents the matrix trace.
Hence, for a positive integer $k$, we have
\begin{equation}
\begin{split}\label{eq:special_value_primitive}
\zeta_Q(k)&=\Tr Q^{-k}\\
&=\int_{[0,1]^k}\biggl(
\int_{\R^k}\tr\kakko{K(u_1,x_1,x_2)K(u_2,x_2,x_3)\dotsb K(u_k,x_k,x_1)}d\vx
\biggr)d\vu\\
&=\biggl(\frac2{\sqrt{\pi(1-\varepsilon^2)}}\biggr)^{\!\!k}
\int_{[0,1]^k}\F(u_1,\dots,u_k)\frac{d\vu}{\sqrt{\prod_{j=1}^k(1-u_j^4)}},
\end{split}
\end{equation}
where $d\vx=dx_1\dots dx_k$ and $d\vu=du_1\dots du_k$ for short, and the symbol $\Tr$ denotes the operator trace.
This is our basis to calculate the special values.
Thus, we have only to calculate $\F(u_1,\dots,u_k)$ to get the special values of the spectral zeta function $\zeta_Q(s)$.

\subsection{Special values $\zeta_Q(k)$}

The following lemma is crucial.
\begin{lem}\label{lem:expansion_of_B}
For any positive integer $k$, it holds that
\begin{align*}
&\eqsp\B(x_1,x_2,\dots,x_k)\\
&=2\kakko{\frac{\alpha+\beta}{2\alpha\beta}}^{\!\!k}
\ckakko{1+\sum_{0<2j\le k}\kakko{\frac{\alpha-\beta}{\alpha+\beta}}^{\!\!2j}%
\sum_{1\le i_1<i_2<\dots<i_{2j}\le k}
\cos\kakko{\q\sum_{r=1}^{2j}(-1)^{r}x_{i_r}^2}}.
\end{align*}
\end{lem}

\begin{proof}
For convenience, let us put $i=\sqrt{-1}$, $a_{1}=\alpha^{-1}$, $a_{2}=\beta^{-1}$ and $t_j=\q x_j^2/2$ $(j=1,2,\dots,k)$.
The function $\B(x_1,x_2,\dots,x_k)$ is then calculated as follows;
\begin{align*}
&\eqsp\B(x_1,x_2,\dots,x_k)\\
&=\sum_{s_1,s_2,\dots,s_k\in\{1,2\}}
a_{s_1}a_{s_2}\dots a_{s_k}
\prod_{m=1}^k \cos\!\kakko{t_m-t_{m+1}+\frac{s_{m+1}-s_m}2\pi}\\
&=\sum_{s_1,s_2,\dots,s_k\in\{1,2\}}
a_{s_1}a_{s_2}\dots a_{s_k}
\prod_{m=1}^k \frac{i^{s_{m+1}-s_m}e^{i(t_m-t_{m+1})}+i^{-(s_{m+1}-s_m)}e^{-i(t_m-t_{m+1})}}2\\
&=\frac1{2^k}\sum_{s_1,s_2,\dots,s_k\in\{1,2\}}
\sum_{l_1,l_2,\dots,l_k\in\{1,-1\}}
\prod_{m=1}^k a_{s_m}i^{l_m(s_{m+1}-s_m)}e^{il_m(t_m-t_{m+1})}\\
&=\frac1{2^k}\sum_{l_1,l_2,\dots,l_k\in\{1,-1\}}
\sum_{s_1,s_2,\dots,s_k\in\{1,2\}}
\prod_{m=1}^k a_{s_m}i^{s_m(l_{m-1}-l_m)}e^{it_m(l_m-l_{m-1})}\\
&=\frac1{2^k}\sum_{l_1,l_2,\dots,l_k\in\{1,-1\}}
\prod_{m=1}^k (a_1i^{l_{m-1}-l_m}+a_2i^{2(l_{m-1}-l_m)})
e^{i(l_m-l_{m-1})t_m},
\end{align*}
where we set $s_0=s_k$, $s_{k+1}=s_1$, $l_0=l_k$, $l_{k+1}=l_1$, $t_0=t_k$ and $t_{k+1}=t_1$.
Here we notice that
\begin{itemize}
\item[(i)] $i^{l_{m-1}-l_m}=-(-1)^{\delta_{l_m,l_{m-1}}}$,
\item[(ii)] $\card{\Set{m\in\{1,2,\dots,K\}}{l_{m-1}\ne l_m}}$ is even (remark that $l_0=l_k$),
\item[(iii)] if there exist $i_1,\dots,i_{2j}\in\{1,2,\dots,k\}$ such that $i_1<\dots<i_{2j}$,
$l_{i_r-1}\ne l_{i_r}$ for each $r=1,2,\dots,2j$
and $l_{m-1}=l_m$ for $m\in\{1,2,\dots,k\}\setminus\{i_1,\dots,i_{2j}\}$,
then $\sum_{m=1}^k (l_m-l_{m-1})t_m=2l_{i_1}\sum_{r=1}^{2j}(-1)^rt_{i_r}$.
\end{itemize}
Thus it follows that
\begin{align*}
&\eqsp\B(x_1,x_2,\dots,x_k)\\
&=\frac1{2^k}\sum_{l\in\{1,-1\}}\kakko{1+%
\sum_{0<2j\le k}(\beta^{-1}-\alpha^{-1})^{2j}(\beta^{-1}+\alpha^{-1})^{k-2j}
\sum_{1\le i_1<\dots<i_{2j}\le k}
\cos\kakko{2l\sum_{r=1}^{2j} (-1)^rt_{i_r}}}\\
&=2\kakko{\frac{\alpha+\beta}{2\alpha\beta}}^{\!k}
\kakko{1+%
\sum_{0<2j\le k}\kakko{\frac{\alpha-\beta}{\alpha+\beta}}^{\!2j}
\sum_{1\le i_1<\dots<i_{2j}\le k}
\cos\kakko{\sum_{r=1}^{2j} (-1)^r\q x^2_{i_r}}}.
\end{align*}
This is the desired conclusion.
\end{proof}

For $\vu=(u_1,u_2,\dots,u_k)$, we define the $k$ by $k$ matrix $\Delta_k(\vu)$ by
\begin{align*}
\Delta_k(\vu)&\deq\begin{pmatrix}
\hA k1 & \hB1 & 0 & 0 & \dots & \hB k \\
\hB1 & \hA12 & \hB2 & 0 & \dots & 0 \\
0 & \hB2 & \hA23 & \hB3 & \dots & 0 \\
0 & 0 & \hB3 & \ddots & \ddots & \vdots\\
\vdots & \vdots & \vdots & \ddots & \ddots & \hB{k-1} \\
\hB k & 0 & 0 & \dots & \hB{k-1} & \hA{k-1}k
\end{pmatrix}\\
&=\sum_{i=1}^k\ckakko{\kakko{E^{(k)}_{ii}+E^{(k)}_{i+1,i+1}}\kakko{\frac1{1-u_i^4}-\frac12}%
+\kakko{E^{(k)}_{i,i+1}+E^{(k)}_{i+1,i}}\frac{-u_i^2}{1-u_i^4}}.
\end{align*}
It then follows that
\begin{equation}
\E(u_1,\dots,u_k;x_1,\dots,x_k)=\exp\kakko{-\vx\Delta_k(\vu)\vx'}
\end{equation}
and
\begin{equation}\label{eq:Vn}
\det\Delta_k(\vu)=\frac{(1-u_1^2\dots u_k^2)^2}{(1-u_1^4)\dots(1-u_k^4)}
\end{equation}
(see \cite[Theorem A.2]{IW2005KJM}).
Here $\vx=(x_1,x_2,\dots,x_k)$, $E^{(k)}_{ij}$ denotes the matrix unit of size $k$.
We also assume that the indices of $E^{(k)}_{ij}$ are understood modulo $k$,
i.e.\ $E^{(k)}_{0,j}=E^{(k)}_{k,j}$, $E^{(k)}_{k+1,j}=E^{(k)}_{1,j}$, etc.
The prime $'$ indicates the matrix transpose.
Notice that $\Delta_k(\vu)$ is real symmetric and positive definite for any $\vu\in(0,1)^k$.
For $\{i_1,i_2,\dots,i_{2j}\}\subset[k]=\{1,2,\dots,k\}$, we also put
\begin{equation*}
\Xi_k(i_1,\dots,i_{2j})\deq\I\sum_{r=1}^{2j}(-1)^rE^{(k)}_{i_r,i_r}.
\end{equation*}
Since
\begin{equation*}
\sum_{r=1}^{2j}(-1)^rx_{i_r}^2=\vx\Xi_k(i_1,\dots,i_{2j})\vx'
\end{equation*}
and
\begin{equation*}
\cos\kakko{\q\sum_{r=1}^{2j}(-1)^rx_{i_r}^2}
=\frac12\ckakko{\exp\kakko{\I \q\sum_{r=1}^{2j}(-1)^rx_{i_r}^2}+\exp\kakko{-\I \q\sum_{r=1}^{2j}(-1)^rx_{i_r}^2}},
\end{equation*}
we have
\begin{equation}\label{eq:Ecos}
\begin{split}
&\eqsp\E(u_1,\dots,u_k;x_1,\dots,x_k)\cos\kakko{\q\sum_{r=1}^{2j}(-1)^rx_{i_r}^2}\\
&=\frac12\exp\kakko{-\vx\kakko{\Delta_k(\vu)+\q\Xi_k(i_1,\dots,i_{2j})}\vx'}%
+\frac12\exp\kakko{-\vx\kakko{\Delta_k(\vu)-\q\Xi_k(i_1,\dots,i_{2j})}\vx'}.
\end{split}
\end{equation}

As in \cite[Lemma A.1]{IW2005KJM}, one proves the
\begin{lem}\label{lem:reality}
The determinant
\begin{equation}\label{eq:evenness}
\det\kakko{\Delta_k(\vu)+\q\Xi_k(i_1,\dots,i_{2j})}
\end{equation}
is even in $\q$.
In particular, this determinant is real-valued for each $\vu\in(0,1)^k$ and $\q>0$.
\qed
\end{lem}
Let $\cyclic m$ denotes the cyclic subgroup of the symmetric group $\sym m$ of degree $m$
generated by the cyclic permutation $(1,2,\dots,m)\in\sym{m}$.
By Lemma \ref{lem:reality}, it follows that
\begin{equation*}
\det\kakko{\Delta_k(\vu)+\q\Xi_k(i_1,\dots,i_{2j})}
=\det\kakko{\Delta_k(\vu)+\q\Xi_k(j_{\sigma(1)},\dots,j_{\sigma(2k)})}
\end{equation*}
for any $\sigma\in \cyclic{2k}$
since $\Xi_k(j_{\sigma(1)},\dots,j_{\sigma(2k)})=\sgn(\sigma)\Xi_k(i_1,\dots,i_{2j})$.

Let $\symx k$ be the set of $k$ by $k$ \emph{complex} symmetric matrices such that all principal minors are invertible,
and $\Sym_k^+(\R)$ be the set of $k$ by $k$ \emph{positive real} symmetric matrices.
Notice that $\Delta_k(\vu)\in\Sym_k^+(\R)$ for any $\vu\in(0,1)^k$.
We need the following two lemmas for later use in the evaluation of $\F(u_1,\dots,u_k)$.

\begin{lem}[LDU decomposition]\label{lem:LDU}
Let $k$ be a positive integer.
For any $A\in\symx k$,
there exists a lower unitriangular matrix $L$ and a diagonal matrix $D$ such that $A=LDL'$.
Moreover, $D$ is given by
\begin{equation*}
D=\diag(d_1,d_2/d_1,d_3/d_2,\dots,d_k/d_{k-1}),
\end{equation*}
where $d_j$ denotes the $j$-th principal minor determinant of $A$.
\end{lem}

\begin{proof}
Let us prove by induction on $k$.
The assertion is clear if $k=1$.
Suppose that the assertion is true for $k-1$.
Take $A\in\symx k$ and write
\begin{equation*}
A=\begin{pmatrix}
A_0 & \va \\ \va' & \alpha
\end{pmatrix}
\end{equation*}
with $A_0\in\symx{k-1}$, $\va\in\C^{k-1}$ and $\alpha\in\C$.
By the induction hypothesis,
there exist lower unitriangular matrix $L_0$ and diagonal matrix $D_0$ of size $k-1$ such that $A_0=L_0D_0L_0'$.
Put
\begin{equation*}
L=\begin{pmatrix} L_0 & \vo \\ \vv' & 1 \end{pmatrix},\quad
D=\begin{pmatrix} D_0 & \vo \\ \vo' & d \end{pmatrix},
\end{equation*}
where $\vv=(L_0D_0)^{-1}\va$ and $d=\alpha-\va'A_0^{-1}\va$
(notice that $(L_0D_0)^{-1}$ and $A_0^{-1}$ exist by the induction hypothesis)
and $\vo\in\C^{k-1}$ represents the zero vector.
Then it is straightforward to check that $A=LDL'$.
This prove the first assertion of the lemma.
The second assertion is obvious by the construction of $D$ above.
\end{proof}

\begin{lem}\label{lem:positivity_lemma}
Let $T\in\Sym_k^+(\R)$ and $D$ be a real diagonal matrix of size $k$.
Denote by $d_m$ the principal $m$-minor determinant of $T+\I D$.
Then it follows that $\Re\kakko{d_{m+1}\conj{d_m}}>0$ for $m=1,2,\dots,k-1$.
\end{lem}

\begin{proof}
Clearly, it is enough to prove the positivity of $\Re\kakko{d_{m+1}\conj{d_m}}$ with $k=m+1$.
Write $T$ and $D$ as
\begin{equation*}
T=\begin{pmatrix} A & \va \\ \va' & \alpha \end{pmatrix},\quad
D=\begin{pmatrix} U & \vo \\ \vo' & u \end{pmatrix}
\end{equation*}
with $A\in\Sym_m^+(\R)$, $\va\in\R^m$, $\alpha\in\R$, $u\in\R$ and a real diagonal matrix $U$ of size $m$.
Here $\vo\in\R^m$ is the zero vector.
Since $T$ is positive, we must have $0<\va'A^{-1}\va<\alpha$.
Put $B=\sqrt A\in\Sym_m^+(\R)$, $X=B^{-1}UB^{-1}\in\Sym_m(\R)$ and $\vb=B^{-1}\va$.
Then we have
\begin{align*}
d_{m+1}\conj{d_m}
&=\det\begin{pmatrix} A+\I U & \va \\ \va' & \alpha+\I u \end{pmatrix}
\begin{pmatrix} A-\I U & \vo \\ \vo & 1 \end{pmatrix}\\
&=\det\begin{pmatrix} (A+\I U)(A-\I U) & \va \\ \va'(A-\I U) & \alpha+\I u \end{pmatrix}\\
&=\abs{\det(A+\I U)}^2\kakko{\alpha+\I u-\va'(A+\I U)^{-1}\va}\\
&=\det B^4\det(I+X^2)\kakko{\alpha+\I u-\vb'\kakko{I+\I X}^{-1}\vb}.
\end{align*}
Since
\begin{equation*}
\kakko{I+\I X}^{-1}=(I+X^2)^{-1}-\I X(I+X^2)^{-1},
\end{equation*}
it follows that
\begin{equation*}
\Re\kakko{\vb'\kakko{I+\I X}^{-1}\vb}=\vb'(I+X^2)^{-1}\vb\le\vb'\vb=\va'A^{-1}\va<\alpha
\end{equation*}
or
\begin{equation*}
\Re\kakko{\alpha+\I u-\vb'\kakko{I+\I X}^{-1}\vb}>0.
\end{equation*}
Thus we have $\Re(d_{m+1}\conj{d_m})>0$ as desired.
\end{proof}

We recall the well-known fact.
\begin{lem}[Gaussian integral]\label{lem:gaussian}
For any $a,b\in\C$ with $\Re a>0$, it follows that
\begin{equation*}
\int_{\R}\exp(-a(x-b)^2)dx=\sqrt{\frac{\pi}{a}}.
\end{equation*}
Here $\sqrt a$ is chosen as $\Re\sqrt a>0$.
\qed
\end{lem}

By Lemma \ref{lem:LDU}, $A\in\symx k$ is decomposed as $A=LDL'$
with a certain lower unitriangular matrix $L$ and a diagonal matrix $D=\diag(d_1,d_2/d_1,\dots,d_k/d_{k-1})$,
where $d_j$ is the $j$-th principal minor determinant of $A$.
If all entries of $D$ have positive real parts, then it follows from Lemma \ref{lem:gaussian} that
\begin{equation}\label{eq:gaussian_multidim}
\int_{\R^k}\exp(-\vx A\vx')d\vx=\frac{\pi^{k/2}}{\sqrt{\det A}}.
\end{equation}

Now the matrix $\Delta_k(\vu)+\q\Xi_k(\vi)$ belongs to $\symx k$ for any $\vu\in(0,1)^k$.
Denote by $d_k=d_k(k,\vu,\q,\vi)$ the $k$-th principal minor determinant of $\Delta_k(\vu)+\q\Xi_k(\vi)$, and put $d_0=1$.
It then follows from Lemma \ref{lem:positivity_lemma} that $\Re(d_j/d_{j-1})>0$ for $j=1,2,\dots,k$.
Consequently, in view of \eqref{eq:Ecos}, \eqref{eq:evenness}, \eqref{eq:gaussian_multidim} and Lemma \ref{lem:expansion_of_B},
we can calculate $\F(u_1,\dots,u_k)$ as
\begin{align*}
&\F(u_1,\dots,u_k)=2\sqrt{\pi^k}\kakko{\frac{\alpha+\beta}{2\alpha\beta}}^{\!\!k}\\
&\times\ckakko{\frac1{\sqrt{\det\Delta_k(\vu)}}%
+\sum_{0<2j\le k}\kakko{\frac{\alpha-\beta}{\alpha+\beta}}^{\!\!2j}%
\sum_{1\le i_1<i_2<\dots<i_{2j}\le k}
\frac1{\sqrt{\det\kakko{\Delta_k(\vu)+\q\Xi_k(i_1,\dots,i_{2j})}}}}.
\end{align*}
We also notice that
\begin{equation*}
\int_{[0,1]^k}\frac{2^kd\vu}{1-u_1^2u_2^2\dotsb u_k^2}=\zeta(k,1/2)
\end{equation*}
for $k\ge2$.
From these equations together with \eqref{eq:special_value_primitive} and \eqref{eq:Vn}, we now obtain the
\begin{thm}\label{thm:specialvalues}
For each positive integer $k\ge2$, it follows that
\begin{equation}\label{eq:specialvalues}
\begin{split}
\zeta_Q(k)&=2\kakko{\frac{\alpha+\beta}{2\sqrt{\alpha\beta(\alpha\beta-1)}}}^{\!\!k}
\Biggl(\zeta(k,1/2)+\sum_{0<2j\le k}\kakko{\frac{\alpha-\beta}{\alpha+\beta}}^{\!\!2j}R_{k,j}(\q)\Biggr).
\end{split}
\end{equation}
Here $R_{k,j}(\q)$ is given by a sum of integrals
\begin{equation*}
R_{k,j}(\q)=
\sum_{1\le i_1<i_2<\dots<i_{2j}\le k}\int_{[0,1]^k}
\frac{2^kdu_1\dots du_k}{\sqrt{\den k(\vu;\q;i_1,\dots,i_{2j})}},
\end{equation*}
where the function $\den k(\vu;\q;i_1,\dots,i_{2j})$ is given by
\begin{equation*}
\den k(\vu;\q;i_1,\dots,i_{2j})=\det\kakko{\Delta_k(\vu)+\q\Xi_k(i_1,\dots,i_{2j})}\prod_{r=1}^k(1-u_r^4).
\eqed
\end{equation*}
\end{thm}

\begin{rem}
The algebraic variety $\den k(\vu;\q;i_1,\dots,i_{2j})=0$ defined by the denominator of the integral $R_{k,j}(\q)$ above is worth studying, e.g.\ from the viewpoint in \cite{Beu1985, Beu1987, BP1984}.
\end{rem}

\begin{rem}\label{rem:degenerate_case}
If $\alpha=\beta$, then we have $\zeta_Q(k)=2(\alpha^2-1)^{-k/2}\zeta(k,1/2)$,
which is a special case of the fact that
$\zeta_Q(s)=2(\alpha^2-1)^{-s/2}\zeta(s,1/2)$ for $\alpha=\beta$.
In fact, when $\alpha$ and $\beta$ are equal,
we can show that $Q\cong\sqrt{\alpha^2-1}\kakko{-\frac12\frac{d^2}{dx^2}+\frac12x^2}I$ (see \cite{P}).
\end{rem}

We give an expansion of the determinants $\den k(\vu;\q;i_1,\dots,i_{2j})$ appearing in \eqref{eq:specialvalues}.
For $\vj=\{j_1,j_2,\dots,j_r\}\subset[k]$ with $r>0$ and $j_1<j_2<\dots<j_r$, define
\begin{equation*}
C_{k}(\vu;\vj)=\prod_{i=1}^r(1-u_{j_i}^4u_{j_i+1}^4\dots u_{j_{i+1}-1}^4).
\end{equation*}
We also define $C_{k}(\vu;\emptyset)=(1-u_1^2u_2^2\dots u_k^2)^2$.
Here we regard that $j_{r+1}=k+j_1$ and $u_{i+k}=u_i$.
For instance, if $k=9$ and $\vj=\{3,6,8\}$, then
\begin{equation*}
C_{9}(\vu;\vj)=(1-u_3^4u_4^4u_5^4)(1-u_6^4u_7^4)(1-u_8^4u_9^4u_1^4u_2^4).
\end{equation*}

\begin{lem}
For a given subset $\vj=\{j_1,j_2,\dots,j_r\}\subset[k]$ with $j_1<j_2<\dots<j_r$, it follows that
\begin{equation}
\den k(\vu;\q;\vj)=\sum_{d\ge0}(-\q^2)^{d}\,\den{k,d}(\vu;\vj)
\end{equation}
with
\begin{equation}
\den{k,d}(\vu;\vj)\deq\sum_{\substack{S\subset[2k]\\\card{S}=2d}}(-1)^{\esum S}C_{k}(\vu;\vj(S)).
\end{equation}
Here $\esum S\deq\sum_{s\in S}s$ is the sum of the elements in $S$ and
$\vj(S)\deq\{j_{s_1},\dots,j_{s_l}\}$ if $S=\{s_1,\dots,s_l\}$ with $s_1<\dots<s_l$.
\end{lem}

\begin{proof}
Let $\vd_i$ be the $i$-th column vector of $\Delta_k(\vu)$.
We also denote by $\{\ve_i\}_{i=1}^k$ the standard basis of $\C^k$.
By the multilinearity of a determinant, we readily get
\begin{align*}
&\eqsp\det(\Delta_k(\vu)+\q\Xi_k(\vj))\\
&=\det\Delta_k(\vu)+\sum_{r=1}^{2j}(\I \q)^r
\sum_{1\le s_1<\dots<s_r\le 2j}(-1)^{s_1+\dots+s_r}\det(\vd_1,\dots,\ve_{j_{s_1}},\dots,\ve_{j_{s_r}},\dots,\vd_k).
\end{align*}
The determinant $\det(\vd_1,\dots,\ve_{j_{s_1}},\dots,\ve_{j_{s_r}},\dots,\vd_k)$ is a product of
$r$ tridiagonal determinants
\begin{equation*}
D_i=
\begin{vmatrix}
d_{a(i)+1,a(i)+1} & d_{a(i)+1,a(i)+2} & \dots & d_{a(i)+1,a(i+1)-1} \\
d_{a(i)+2,a(i)+1} & d_{a(i)+2,a(i)+2} & \dots & d_{a(i)+2,a(i+1)-1} \\
\vdots & \vdots & \ddots & \vdots \\
d_{a(i+1)-1,a(i)+1} & d_{a(i+1)-1,a(i)+2} & \dots & d_{a(i+1)-1,a(i+1)-1}
\end{vmatrix},
\end{equation*}
where $a(i)=j_{s_i}$, $d_{ij}$ is the $(i,j)$-entry of $\Delta_k(\vu)$,
and the indices are understood modulo $k$.
If $a(i+1)=a(i)+1$, then we understand that $D_i=1$.
It is easy to see that
\begin{equation*}
D_i=\frac{1-u_{a(i)}^4u_{a(i)+1}^4\dotsb u_{a(i+1)-1}^4}{(1-u_{a(i)}^4)(1-u_{a(i)+1}^4)\dotsb(1-u_{a(i+1)-1}^4)}.
\end{equation*}
Hence we have
\begin{equation*}
\den k(\vu;\q;\vj)=\sum_{S\subset[2k]}(-1)^{\esum S}(\I \q)^{\card{S}}C_{k}(\vu;\vj(S)).
\end{equation*}
Since $\den k(\vu;\q;\vj)$ is real-valued by Lemma \ref{lem:reality},
we have the conclusion by taking the real parts.
\end{proof}

\subsection{Examples}

\subsubsection{$\den{k,d}(\vu;\vj)$ and $R_{k,j}(\q)$}

We give several examples of $\den{k,d}(\vu;\vj)$.
For convenience, we prepare some notation for abbreviation.
Let us put
\begin{equation*}
\V k(\vu)\deq(1-u_1^2\dots u_k^2)^2,\qquad
\U\vt(\vu)\deq\prod_{i=1}^m\biggl(1-\prod_{j=1}^{t_i}u^4_{j+\sum_{k<i}t_k}\biggr)
\end{equation*}
for a positive integer $k$ and a sequence $\vt=(t_1,\dots,t_m)\in\cT m(k)$, where
\begin{equation*}
\cT m(k)\deq\set[Big]{(t_1,\dots,t_m)\in[k]^4}{t_1+\dotsb+t_m=k}.
\end{equation*}
For instance, if $\vt=(2,3,2,1)\in\cT4(8)$, then
\begin{equation*}
\U\vt(u_1,\dots,u_8)=(1-u_1^4u_2^4)(1-u_3^4u_4^4u_5^4)(1-u_6^4u_7^4)(1-u_8^4).
\end{equation*}
Notice that $\den{k,0}(\vu;\vj)=\V k(\vu)$ for any $\vj$.

\begin{ex}\label{ex:Rn1}
For $\vj=\{i,j\}\subset[k]$ with $i<j$, we have
\begin{equation*}
\den{k,1}(\vu;\vj)=(-1)^{1+2}C_{k}(\vu;\vj)=-\U{(r,k-r)}(u_i,u_{i+1},\dots,u_{i-1}),
\end{equation*}
where $r=j-i$.
This fact immediately implies that $R_{k,1}(\q)$ in \eqref{eq:specialvalues} is given by
\begin{equation}\label{eq:first_anomaly}
\begin{split}
R_{k,1}(\q)
&=\frac k2\sum_{r=1}^{k-1}\int_{[0,1]^k}\frac{2^kd\vu}{\sqrt{\V k(\vu)+\q^2\U{(r,k-r)}(\vu)}}\\
&=\sum_{0<2r\le k}\frac k{1+\delta_{2r,k}}\int_{[0,1]^k}\frac{2^kd\vu}{\sqrt{\V k(\vu)+\q^2\U{(k-r,r)}(\vu)}}.
\end{split}
\end{equation}
\end{ex}

\begin{ex}\label{ex:top_term}
For $\vj\subset[k]$ with $\card\vj=2k$, it follows in general that
\begin{equation*}
\den{k,j}(\vu;\vj)=(-1)^{k}C_{k}(\vu;\vj)
\end{equation*}
since $\esum{[2k]}=k(2k-1)\equiv k\pmod2$.
\end{ex}

\begin{ex}\label{ex:Rn2}
For $\vj=\{j_1,j_2,j_3,j_4\}\subset[k]$ with $j_1<j_2<j_3<j_4$, we have
\begin{align*}
\den{k,1}(\vu;\vj)
&=(-1)^{1+2}C_{k}(\vu;j_1,j_2)+(-1)^{1+3}C_{k}(\vu;j_1,j_3)+(-1)^{1+4}C_{k}(\vu;j_1,j_4)\\
&\eqsp+(-1)^{2+3}C_{k}(\vu;j_2,j_3)+(-1)^{2+4}C_{k}(\vu;j_2,j_4)+(-1)^{3+4}C_{k}(\vu;j_3,j_4)\\
&=-(1-u_{j_1}^4\dots u_{j_2-1}^4)(1-u_{j_3}^4\dots u_{j_4-1}^4)(1-u_{j_2}^4\dots u_{j_3-1}^4)(1-u_{j_4}^4\dots u_{j_1-1}^4)\\
&\eqsp-(1-u_{j_1}^4\dots u_{j_2-1}^4u_{j_3}^4\dots u_{j_4-1}^4)(1-u_{j_2}^4\dots u_{j_3-1}^4u_{j_4}^4\dots u_{j_1-1}^4).
\end{align*}
By Example \ref{ex:top_term}, we also see that
\begin{align*}
\den{k,2}(\vu;\vj)&=C_k(\vu;\vj)\\
&=(1-u_{j_1}^4\dots u_{j_2-1}^4)(1-u_{j_2}^4\dots u_{j_3-1}^4)(1-u_{j_3}^4\dots u_{j_4-1}^4)(1-u_{j_4}^4\dots u_{j_1-1}^4).
\end{align*}
Thus we have
\begin{align*}
&\eqsp\det(\Delta_k(\vu)+\q\Xi_k(j_1,j_2,j_3,j_4))\prod_{i=1}^k(1-u_i^4)\\
&=\V k(\vu)+(\q^2+\q^4)(1-u_{j_1}^4\dots u_{j_2-1}^4)(1-u_{j_3}^4\dots u_{j_4-1}^4)(1-u_{j_2}^4\dots u_{j_3-1}^4)(1-u_{j_4}^4\dots u_{j_1-1}^4)\\
&\eqsp+\q^2(1-u_{j_1}^4\dots u_{j_2-1}^4u_{j_3}^4\dots u_{j_4-1}^4)(1-u_{j_2}^4\dots u_{j_3-1}^4u_{j_4}^4\dots u_{j_1-1}^4).
\end{align*}
If we take $(t_1,t_2,t_3,t_4)\in\cT4(k)$ such that $j_{i+1}\equiv j_i+t_i\pmod{k}$ ($i=1,2,3,4$; $j_5=j_1$),
then it follows that
\begin{multline*}
\int_{[0,1]^k}\frac{d\vu}{\sqrt{\det(\Delta_k(\vu)+\q\Xi_k(j_1,j_2,j_3,j_4))\prod_{i=1}^k(1-u_i^4)}}\\
=\int_{[0,1]^k}\frac{d\vu}{\sqrt{\V k(\vu)+\q^2\U{(t_1+t_3,t_2+t_4)}(\vu)+(\q^2+\q^4)\U{(t_1,t_2,t_3,t_4)}(\vu)}}.
\end{multline*}
The cyclic group $\cyclic 4$ of order $4$ naturally acts on $\cT4(k)$ by
\begin{equation*}
\sigma.(t_1,t_2,t_3,t_4)\deq(t_{\sigma(1)},t_{\sigma(2)},t_{\sigma(3)},t_{\sigma(4)}) \quad (\sigma\in\cyclic4).
\end{equation*}
Notice that the integral above is $\cyclic 4$-invariant.
For a given $\vt=(t_1,t_2,t_3,t_4)\in\cT4(k)$,
the number of subsets $\vj=\{j_1,j_2,j_3,j_4\}$ in $[k]$
satisfying the condition $j_{i+1}\equiv j_i+t_i\pmod{k}$ is equal to $k/\card{C_4(\vt)}$,
where $\cyclic4(\vt)$ denotes the stabilizer of $\vt$ in $\cyclic4$.
Consequently,
\begin{equation}\label{eq:second_anomaly}
\begin{split}
R_{k,2}(\q)
&=\sum_{\vt\in\cT4(k)/\cyclic4}\frac k{\card{\cyclic4(\vt)}}
\int_{[0,1]^k}\frac{2^kd\vu}{\sqrt{\V k(\vu)+\q^2\U{(t_1+t_3,t_2+t_4)}(\vu)+(\q^2+\q^4)\U{(t_1,t_2,t_3,t_4)}(\vu)}}\\
&=\frac k4\sum_{\vt\in\cT4(k)}
\int_{[0,1]^k}\frac{2^kd\vu}{\sqrt{\V k(\vu)+\q^2\U{(t_1+t_3,t_2+t_4)}(\vu)+(\q^2+\q^4)\U{(t_1,t_2,t_3,t_4)}(\vu)}},
\end{split}
\end{equation}
where $\vt=(t_1,t_2,t_3,t_4)$.
Similarly, the result in Example \ref{ex:Rn1} can be also rewritten as
\begin{align*}
R_{k,1}(\q)
&=\frac k2\sum_{\vt\in\cT2(k)}\int_{[0,1]^k}\frac{2^kd\vu}{\sqrt{\V k(\vu)+\q^2\U{(t_1,t_2)}(\vu)}}\\
&=\sum_{\vt\in\cT2(k)/\cyclic2}\frac k{\card{\cyclic2(\vt)}}\int_{[0,1]^k}\frac{2^kd\vu}{\sqrt{\V k(\vu)+\q^2\U{(t_1,t_2)}(\vu)}}.
\end{align*}
\end{ex}

\subsubsection{Several special values}

Using Theorem \ref{thm:specialvalues} and the formulas
\eqref{eq:first_anomaly} and \eqref{eq:second_anomaly} for $R_{k,1}(\q)$ and $R_{k,2}(\q)$
given in the previous examples,
we show several examples of the special values of $\zeta_Q(s)$.

\begin{ex}
The values $\zeta_Q(2)$ and $\zeta_Q(3)$ are given by
\begin{align*}
\zeta_Q(2)&=2\kakko{\frac{\alpha+\beta}{2\sqrt{\alpha\beta(\alpha\beta-1)}}}^{\!\!2}
\Biggl(\zeta(2,1/2)+\kakko{\frac{\alpha-\beta}{\alpha+\beta}}^{\!\!2}R_{2,1}(\q)\Biggr),\\
\zeta_Q(3)&=2\kakko{\frac{\alpha+\beta}{2\sqrt{\alpha\beta(\alpha\beta-1)}}}^{\!\!3}
\Biggl(\zeta(3,1/2)+\kakko{\frac{\alpha-\beta}{\alpha+\beta}}^{\!\!2}R_{3,1}(\q)\Biggr)
\end{align*}
with
\begin{align*}
R_{2,1}(\q)&=\int_{[0,1]^2}\frac{4du_1du_2}{\sqrt{\V2(\vu)+\q^2\U{(1,1)}(\vu)}}
=\int_{[0,1]^2}\frac{4du_1du_2}{\sqrt{(1-u_1^2u_2^2)^2+\q^2(1-u_1^4)(1-u_2^4)}},\\
R_{3,1}(\q)&=3\int_{[0,1]^3}\frac{8du_1du_2du_3}{\sqrt{\V3(\vu)+\q^2\U{(2,1)}(\vu)}}
=3\int_{[0,1]^3}\frac{8du_1du_2du_3}{\sqrt{(1-u_1^2u_2^2u_3^2)^2+\q^2(1-u_1^4)(1-u_2^4u_3^4)}}.
\end{align*}
This recovers the result obtained in \cite{IW2005KJM}.
\end{ex}

\begin{ex}
The values $\zeta_Q(4)$ and $\zeta_Q(5)$ are given by
\begin{align*}
\zeta_Q(4)&=2\kakko{\frac{\alpha+\beta}{2\sqrt{\alpha\beta(\alpha\beta-1)}}}^{\!\!4}
\Biggl(\zeta(4,1/2)+\kakko{\frac{\alpha-\beta}{\alpha+\beta}}^{\!\!2}R_{4,1}(\q)
+\kakko{\frac{\alpha-\beta}{\alpha+\beta}}^{\!\!4}R_{4,2}(\q)\Biggr),\\
\zeta_Q(5)&=2\kakko{\frac{\alpha+\beta}{2\sqrt{\alpha\beta(\alpha\beta-1)}}}^{\!\!5}
\Biggl(\zeta(5,1/2)+\kakko{\frac{\alpha-\beta}{\alpha+\beta}}^{\!\!2}R_{5,1}(\q)
+\kakko{\frac{\alpha-\beta}{\alpha+\beta}}^{\!\!4}R_{5,2}(\q)\Biggr)
\end{align*}
with
\begin{align*}
R_{4,1}(\q)&=4\int_{[0,1]^4}\frac{16d\vu}{\sqrt{\V4(\vu)+\q^2\U{(3,1)}(\vu)}}
+2\int_{[0,1]^4}\frac{16d\vu}{\sqrt{\V4(\vu)+\q^2\U{(2,2)}(\vu)}},\\
R_{5,1}(\q)&=5\int_{[0,1]^5}\frac{32d\vu}{\sqrt{\V5(\vu)+\q^2\U{(4,1)}(\vu)}}
+5\int_{[0,1]^5}\frac{32d\vu}{\sqrt{\V5(\vu)+\q^2\U{(3,2)}(\vu)}}
\end{align*}
and
\begin{align*}
R_{4,2}(\q)&=\int_{[0,1]^4}\frac{16d\vu}{\sqrt{\V4(\vu)+\q^2\U{(2,2)}(\vu)+(\q^2+\q^4)\U{(1,1,1,1)}(\vu)}},\\
R_{5,2}(\q)&=5\int_{[0,1]^5}\frac{32d\vu}{\sqrt{\V5(\vu)+\q^2\U{(3,2)}(\vu)+(\q^2+\q^4)\U{(2,1,1,1)}(\vu)}}.
\end{align*}
\end{ex}

\subsubsection{Ap\'ery-like numbers for $\zeta_Q(2)$ and the elliptic integral}\label{sec:Apery-Heun}

We define the numbers $J_2(m)$ ($m\ge0$) by the expansion
\begin{equation*}
R_{2,1}(\q)=\sum_{m=0}^\infty \binom{-1/2}m J_2(m)\q^{2m}.
\end{equation*}
Then they satisfy the three-term recurrence relation \cite{IW2005KJM}
\begin{equation}\label{eq:recurrence_for_J2}
4m^2J_2(m)-(8m^2-8m+3)J_2(m-1)+4(m-1)^2J_2(m-2)=0\quad(m\ge2).
\end{equation}
This implies that the generating function $w_2(z)=\sum_{m=0}^\infty J_2(m)z^m$ satisfies
\begin{equation}\label{eq:Heun}
\ckakko{z(1-z)^2\frac{d^2}{dz^2}+(1-3z)(1-z)\frac{d}{dz}+z-\frac34}w_2(z)=0.
\end{equation}
This differential equation is the Picard-Fuchs equation for the universal family of elliptic curves equipped with rational $4$-torsion. In fact, each elliptic curve in the family is
birationally equivalent to one of the curves
$(1-u^2v^2)^2+x^2(1-u^4)(1-v^4)=0$
in the $(u,v)$-plane, which are appeared
in the denominator of the integrand of $R_{2,1}(x)$.

The equation \eqref{eq:Heun} can be reduced to the Gaussian hypergeometric differential equation
by a suitable change of variable and solved as follows \cite{O2008RJ}:
\begin{equation*}
w_2(z)=\frac{3\zeta(2)}{1-z}{}_2F_1\kakko{\frac12,\frac12;1;\frac z{z-1}},
\end{equation*}
from which we obtain
\begin{equation*}
R_{2,1}(\q)=3\zeta(2){}_2F_1\kakko{\frac14,\frac34;1;-\q^2}^{\!2}.
\end{equation*}
Thus we have the following formulas for $\zeta_Q(2)$ \cite{IW2005KJM,O2008RJ}:
\begin{align*}
\zeta_Q(2)
&=\kakko{\frac{\pi(\alpha+\beta)}{2\sqrt{\alpha\beta(\alpha\beta-1)}}}^{\!\!2}
\Biggl(1+\frac1{2\pi\I}\kakko{\frac{\alpha-\beta}{\alpha+\beta}}^{\!\!2}\int_{\abs z=r}\frac{u(z)}{z(1+\q^2z)^{1/2}}dz\Biggr)\\
&=\kakko{\frac{\pi(\alpha+\beta)}{2\sqrt{\alpha\beta(\alpha\beta-1)}}}^{\!\!2}
\Biggl(1+\kakko{\frac{\alpha-\beta}{\alpha+\beta}}^{\!\!2}{}_2F_1\kakko{\frac14,\frac34;1;-\q^2}^{\!2}\Biggr),
\end{align*}
where $u(z)=w_2(z)/3\zeta(2)$ is a normalized (unique) holomorphic solution of \eqref{eq:Heun} in $\abs z<1$ and $\q^2<r<1$.
We also have similar formulas for $\zeta_Q(3)$ \cite{IW2005KJM,KW2006KJM}.

\section{Ap\'ery-like numbers}\label{ALN}

In what follows, we restrict our attention on $R_{k,1}(\q)$
appearing in the special value formula for $\zeta_Q(s)$.
We may sometimes refer to $R_{k,1}(\q)$ as the \emph{first anomaly} in $\zeta_Q(k)$ for short.
In this section, we define Ap\'ery-like numbers $J_k(n)$, and study their recurrence equation and the differential equation satisfied by the generating function of $J_k(n)$. We lastly discuss congruence properties for the \emph{normalized} Ap\'ery-like numbers $\tJ{k}n$ (\S\ref{sec:numeric_data_of_NAS}).

\subsection{Ap\'ery-like numbers associated to the first anomalies}

We expand the first anomaly $R_{k,1}(\q)$ as follows:
\begin{align*}
R_{k,1}(\q)
&=\frac{k}2\sum_{r=1}^{k-1}\int_{[0,1]^k}
\frac{2^k du_1\dotsb du_k}{\sqrt{(1-u_1^2\dotsb u_k^2)^2+\q^2(1-u_1^4\dotsb u_r^4)(1-u_{r+1}^4\dotsb u_k^4)}}\\
&=\frac{k}2\sum_{r=1}^{k-1}
\sum_{n=0}^\infty\binom{-\frac12}{n}\tempJ{k}{r}{n}\q^{2n}\\
&=\frac{k}2\sum_{n=0}^\infty\binom{-\frac12}{n}J_k(n)\q^{2n},
\end{align*}
where we put
\begin{align*}
J_k(n)&= \sum_{r=1}^{k-1}\tempJ{k}{r}{n},\\
\tempJ{k}{r}{n} &=2^k \int_{[0,1]^k}
\frac{(1-u_1^4\dotsb u_r^4)^n(1-u_{r+1}^4\dotsb u_k^4)^n}{(1-u_1^2\dotsb u_k^2)^{2n+1}}
du_1\dotsb du_k.
\end{align*}
If we change the variables of the integral by
\begin{equation*}
u_j=e^{-\frac12x_j}\quad(j=1,2,\dots,r),\qquad
u_{r+j}=e^{-\frac12y_j}\quad(j=1,2,\dots,k-r),
\end{equation*}
then the corresponding domain of integral is
\begin{equation*}
0\le x_1\le x_2\le\dotsb\le x_r,\qquad
0\le y_1\le y_2\le\dotsb\le y_{k-r},
\end{equation*}
so that we have
\begin{align*}
\tempJ{k}{r}{n}
&=\int_{\substack{0\le x_1\le x_2\le\dotsb\le x_r\\0\le y_1\le y_2\le\dotsb\le y_{k-r}}}
\frac{(1-e^{-2x_r})^n(1-e^{-2y_{k-r}})^n}{(1-e^{-x_r-y_{k-r}})^{2n+1}}e^{-\frac12(x_r+y_{k-r})}
dx_1dx_2\dotsb dx_rdy_1dy_2\dotsb dy_{k-r}\\
&=\int_0^\infty\!\!\int_0^\infty\frac{x^{r-1}}{(r-1)!}\frac{y^{k-r-1}}{(k-r-1)!}
\frac{e^{-\frac12(x+y)}}{(1-e^{-x-y})^{2n+1}}(1-e^{-2x})^n(1-e^{-2y})^ndxdy\\
&=\int_0^\infty\frac{e^{-\frac12u}}{(1-e^{-u})^{2n+1}}du
\int_0^u\frac{t^{r-1}}{(r-1)!}\frac{(u-t)^{k-r-1}}{(k-r-1)!}
(1-e^{-2t})^n(1-e^{-2(u-t)})^ndt.
\end{align*}
By the binomial theorem, we have
\begin{equation}
\begin{split}
J_k(n)= \sum_{r=1}^{k-1}\tempJ{k}{r}{n}
&=\int_0^\infty\frac{u^{k-2}}{(k-2)!}\frac{e^{-\frac12u}}{(1-e^{-u})^{2n+1}}du
\int_0^u(1-e^{-2t})^n(1-e^{-2(u-t)})^ndt\\
&=\frac1{2^{2n+1}}\int_0^\infty\frac{u^{k-2}}{(k-2)!}\frac{e^{nu}}{(\sinh\frac{u}2)^{2n+1}}du
\int_0^u(1-e^{-2t})^n(1-e^{-2(u-t)})^ndt.
\end{split}\label{eq:definition-of-aperylike-numbers}
\end{equation}
We call the numbers $\J{k}{n}$ the \emph{Ap\'ery-like numbers}
associated to the first anomaly $R_{k,1}(\q)$ of $\zeta_Q(k)$,
or \emph{$k$-th Ap\'ery-like numbers} for short.\footnote{[Differences of conventions]
$\J2{n}$ in this article is equal to $J_n$ in \cite{IW2005KJM} (and $J_2(n)$ in \cite{KW2006KJM}).
$\J3{n}$ in this article is equal to $2J^1_n$ in \cite{IW2005KJM} (and $2J_3(n)$ in \cite{KW2006KJM}),
since our $\J3{n}$ is defined to be the sum $\tempJ{3}{1}{n}+\tempJ{3}{2}{n}$,
each summand in which is equal to $J^1_n$ in \cite{IW2005KJM}.}
By the equation \eqref{eq:definition-of-aperylike-numbers} above, one has
\begin{align}
\J{k}{n}&=\frac1{2^{2n+1}}\int_0^\infty
\frac{u^{k-2}}{(k-2)!}B_n(u)du, \label{eq:J_k(n)}
\\
B_n(u)&=\frac{e^{nu}}{(\sinh\frac{u}2)^{2n+1}}\int_0^u(1-e^{-2t})^n(1-e^{-2(u-t)})^ndt \label{eq:B_n(u)}
\end{align}
for $k=2,3,4,\dots$ and $n=0,1,2,\dots$.
We notice that the function $B_n(u)$ is continuous at $u=0$ and is of exponential decay as $u\to+\infty$
(see Proposition 4.10 in \cite{IW2005KJM}).
It is convenient to introduce the numbers $\J0n$ and $\J1n$ by
\begin{equation*}
\J0n=0,\qquad \J1n=\frac{2^nn!}{(2n+1)!!}=\frac{(1)_n(1)_n}{(\frac32)_n(1)_n}\qquad(n=0,1,2,\dots),
\end{equation*}
where $(a)_n=a(a+1)\dotsb(a+n-1)$ is the Pochhammer symbol.

\begin{ex}
We see that
\begin{align*}
B_0(u)&=\frac1{\sinh\frac{u}2}\int_0^udt=\frac{u}{\sinh\frac{u}2},\\
B_1(u)&=\frac{e^u}{(\sinh\frac{u}2)^3}\int_0^u(1-e^{-2t})(1-e^{-2(u-t)})dt
=4\frac{u}{\sinh\frac{u}2}+2\frac{u-\sinh u}{(\sinh\frac{u}2)^3}\\
&=4\frac{u}{\sinh\frac{u}2}+2\frac{u}{(\sinh\frac{u}2)^3}-4\frac{\cosh\frac{u}2}{(\sinh\frac{u}2)^2}.
\end{align*}
Thus we have
\begin{align*}
\J{k}{0}&=\frac1{2\cdot(k-2)!}\int_0^\infty\frac{u^{k-1}}{\sinh\frac u2}du,\\
\J{k}{1}&=\frac1{2\cdot(k-2)!}\int_0^\infty\frac{u^{k-1}}{\sinh\frac u2}du
+\frac1{4\cdot(k-2)!}\int_0^\infty\frac{u^{k-2}(u-\sinh u)}{(\sinh\frac{u}2)^3}du\\
&=\frac1{2\cdot(k-2)!}\int_0^\infty\frac{u^{k-1}}{\sinh\frac u2}du
+\frac1{4\cdot(k-2)!}\int_0^\infty\frac{u^{k-1}}{(\sinh\frac{u}2)^3}du
-\frac1{2\cdot(k-2)!}\int_0^\infty\frac{\cosh\frac{u}2}{(\sinh\frac{u}2)^2}u^{k-2}du.
\end{align*}
Using the formulas
\begin{align*}
\zeta\Bigl(s,\frac12\Bigr)
&=\frac1{2\Gamma(s)}\int_0^\infty\frac{u^{s-1}}{\sinh\frac u2}du
=\frac1{4\Gamma(s+1)}\int_0^\infty\frac{\cosh\frac{u}2}{(\sinh\frac u2)^2}u^{s}du
\qquad(\Re(s)>1),\\
\int_0^\infty\frac{u^{s-1}}{(\sinh\frac{u}2)^3}du
&=(s-1)\int_0^\infty\frac{\cosh\frac{u}2}{(\sinh\frac u2)^2}u^{s-2}du
-\frac12\int_0^\infty\frac{u^{s-1}}{\sinh\frac u2}du
\qquad(\Re(s)>3),
\end{align*}
we get
\begin{align*}
\J{k}{0}&=(k-1)\zeta\Bigl(k,\frac12\Bigr),\\
\J{k}{1}&=(k-3)\zeta\Bigl(k-2,\frac12\Bigr)+\frac{3(k-1)}4\zeta\Bigl(k,\frac12\Bigr)
\Bigl(=\J{k-2}{0}+\frac34\J{k}{0}\Bigr)
\end{align*}
for $k\ge4$.
It is worth noting that these formulas are also valid for $k=2$ and $k=3$;
\begin{equation*}
\J20=\zeta\Bigl(2,\frac12\Bigr),\quad \J21=\frac34\zeta\Bigl(2,\frac12\Bigr);\qquad
\J30=2\zeta\Bigl(3,\frac12\Bigr),\quad \J31=1+\frac32\zeta\Bigl(3,\frac12\Bigr).
\end{equation*}
\end{ex}
Here we use the fact that
\begin{equation*}
\zeta\Bigl(0,\frac12\Bigr)=0,\qquad
\lim_{s\to1}(s-1)\zeta\Bigl(s,\frac12\Bigr)=1.
\end{equation*}

We have now the following series expansion of $\J{k}{n}$.
\begin{lem}[Series expression]\label{Series expression of ALN}
\begin{equation*}
\J{k}{n}=\sum_{r=1}^{k-1}\sum_{m=0}^\infty\binom{2n+m}{m}
\sum_{j=0}^n(-1)^j\binom{n}{j}
\frac1{(\frac12+m+2j)^r}
\sum_{j=0}^n(-1)^j\binom{n}{j}
\frac1{(\frac12+m+2j)^{k-r}}.
\end{equation*}
\end{lem}

\begin{proof}
It is elementary to see
\begin{align*}
&\int_0^\infty\!\!\int_0^\infty
\frac{x^{r-1}}{(r-1)!}
\frac{y^{k-r-1}}{(k-r-1)!}
\frac{e^{-\frac12(x+y)}(1-e^{-2x})^n(1-e^{-2y})^n}{(1-e^{-x-y})^{2n+1}}
dxdy\\
=&\sum_{m=0}^\infty(-1)^m\binom{-2n-1}{m}
\int_0^\infty
\frac{x^{r-1}}{(r-1)!}e^{-\frac12x-mx}(1-e^{-2x})^ndx
\int_0^\infty
\frac{y^{k-r-1}}{(k-r-1)!}e^{-\frac12y-my}(1-e^{-2y})^ndy.
\end{align*}
Since
\begin{align*}
\int_0^\infty
\frac{z^{a-1}}{(a-1)!}e^{-\frac12z-mz}(1-e^{-2z})^ndz
&=\frac{1}{(a-1)!}\sum_{j=0}^n(-1)^j\binom{n}{j}
\int_0^\infty
z^{a-1}e^{-(\frac12+m+2j)z}dz\\
&=\sum_{j=0}^n(-1)^j\binom{n}{j}
\frac1{(\frac12+m+2j)^a},
\end{align*}
the desired series expansion follows immediately.
\end{proof}

\subsection{Recurrence relations among Ap\'ery-like numbers}

Third Ap\'ery-like numbers $\J3n$
satisfy the following inhomogeneous recurrence formula,
which is obtained in \cite{IW2005KJM}:
\begin{equation}\label{eq:recurrence-relation-by-IW}
\begin{split}
4n^2\J{3}{n}-(8n^2-8n+3)\J{3}{n-1}+4(n-1)^2\J{3}{n-2}=4\J{1}{n-1}
\end{split}
\end{equation}
for each $n\ge2$. One should remark that the homogeneous part
of this recurrence formula is the same as the one for $\J2n$ given in \eqref{eq:recurrence_for_J2}.

We here show that the Ap\'ery-like numbers $\J{k}{n}$ for $k\ge4$ also satisfy similar
three-term recurrence formula.
Put
\begin{equation*}
\T{l}{p}{n}=\frac1{2^{2n+1}}\int_0^\infty\frac{u^l}{l!}\Bigl(\tanh\frac{u}2\Bigr)^pB_n(u)du
\end{equation*}
for $l,p,n=0,1,2,\dots$.
Notice that $\J{k}{n}=\T{k-2}{0}{n}$ for $k\ge2$.
We also note that
\begin{equation*}
\T{l}{p}{0}=\frac1{2\cdot l!}\int_0^\infty\Bigl(\tanh\frac{u}2\Bigr)^p\frac{u^{l+1}}{\sinh\frac{u}2}du.
\end{equation*}

We need the formulas (4.36) and (4.37) in \cite{IW2005KJM}:
\begin{align}
2\tanh\frac{u}2B_n'(u)&=8nB_{n-1}(u)-(2n+1)B_n(u),\label{eq:IW:4.36}\\
n\Bigl(\tanh\frac{u}2\Bigr)^2B_n(u)&=2(2n-1)B_{n-1}(u)+2(2n-1)\Bigl(\tanh\frac{u}2\Bigr)^2B_{n-1}(u)
-16(n-1)B_{n-2}(u).\label{eq:IW:4.37}
\end{align}

It follows from \eqref{eq:IW:4.36} that
\begin{equation*}
(p+1)\T{l}{p+2}{n}-2\T{l-1}{p+1}{n}=2n\T{l}{p}{n-1}-(2n-p)\T{l}{p}{n}.
\end{equation*}
Moreover, it follows from \eqref{eq:IW:4.37} that
\begin{equation*}
2n\T{l}{p+2}{n}=(2n-1)\T{l}{p}{n-1}+(2n-1)\T{l}{p+2}{n-1}-2(n-1)\T{l}{p}{n-2}.
\end{equation*}
Combining these, we get
\begin{multline*}
2n(2n-p)\T{l}{p}{n}-(8n^2-4(p+2)n+3+2p)\T{l}{p}{n-1}+2(n-1)(2n-2-p)\T{l}{p}{n-2}\\
=4n\T{l-1}{p+1}{n}-2(2n-1)\T{l-1}{p+1}{n}.
\end{multline*}

We see that
\begin{align*}
& n\tanh\frac{u}2B_n(u)-2(2n-1)\tanh\frac{u}2B_{n-1}(u)\\
&=-2\cdot\frac{8(n-1)B_{n-2}(u)-(2n-1)B_{n-1}(u)}{\tanh\frac{u}2}\qquad(\because\eqref{eq:IW:4.37})\\
&=-4B_{n-1}'(u)\qquad(\because\eqref{eq:IW:4.36}).
\end{align*}
This implies that
\begin{multline*}
\frac{4n}{2^{2n+1}}\int_0^\infty\frac{u^l}{l!}\Bigl(\tanh\frac{u}2\Bigr)^{p+1}B_n(u)du
-\frac{2(2n-1)}{2^{2n-1}}\int_0^\infty\frac{u^l}{l!}\Bigl(\tanh\frac{u}2\Bigr)^{p+1}B_{n-1}(u)du\\
=-\frac4{2^{2n-1}}\int_0^\infty\frac{u^l}{l!}\Bigl(\tanh\frac{u}2\Bigr)^{p}B_{n-1}'(u)du,
\end{multline*}
or
\begin{equation*}
4n\T{l}{p+1}{n}-2(2n-1)\T{l}{p+1}{n-1}
=4\T{l-1}{p}{n-1}+2p\T{l}{p-1}{n-1}+2p\T{l}{p+1}{n-1}.
\end{equation*}
Hence it follows that
\begin{multline}
2n(2n-p)\T{l}{p}{n}-(8n^2-4(p+2)n+3+2p)\T{l}{p}{n-1}+2(n-1)(2n-2-p)\T{l}{p}{n-2}\\[.5em]
=4\T{l-2}{p}{n-1}+2p\T{l-1}{p-1}{n-1}+2p\T{l-1}{p+1}{n-1}.
\end{multline}
In particular, if we put $p=0$ and $l=k-2$ in the equation above
and join \eqref{eq:recurrence-relation-by-IW},
we obtain the following recurrence equation for $\J{k}{n}\, (k\geq2)$, which was announced in
\cite{KW2012RIMS}.
\begin{thm}\label{thm:recurrence-relation}
\begin{equation}\label{eq:recurrence-relation}
4n^2\J{k}{n}-(8n^2-8n+3)\J{k}{n-1}+4(n-1)^2\J{k}{n-2}=4\J{k-2}{n-1}
\end{equation}
for $k\ge2$ and $n\ge2$.
\qed
\end{thm}

\begin{rem}
The generalized Ap\'ery-like numbers defined in \cite{K2016RMJ} (which was named as $J_k(n)$ in \cite{K2016RMJ}) is identical to $J_{1,k-1}(n)$ in this paper. It is quite interesting that those generalized Ap\'ery-like numbers, i.e. $J_{1,k-1}(n)$, satisfies the following recurrence relation similarly to \eqref{eq:recurrence-relation} (i.e. having the same homogeneous part of the Ap\'ery-like numbers)
\begin{equation}
4n^2\J{1,k-1}{n}-(8n^2-8n+3)\J{1,k-1}{n-1}+4(n-1)^2\J{1,k-1}{n-2}=4\J{1, k-3}{n}
\end{equation}
for $k\ge4$ and $n\ge2$. From this observation, although $J_{1,k-1}(n)$ does not describe the special values $\zeta_Q(n)$, various $J_{r,k-r}(n)\, (r=1,2,\cdots, k-1)$ are having similar nature as the Ap\'ery-like numbers possess. This may suggest that there is a certain unexpectically significant number theoretic properties behind NcHO that should be clarified.
\end{rem}

\subsection{Differential equations for the generating functions}

For $k\ge0$, we define
\begin{align}
w_k(z)&=\sum_{n=0}^\infty \J{k}{n}z^n, \label{eq:definition of w_k(z)}\\
g_k(x)&=\sum_{n=0}^\infty \binom{-\frac12}n \J{k}{n}x^n.
\end{align}
We call $w_k(z)$ the $k$-th generating function of the Ap\'ery-like numbers.
It is immediate to see that $w_0(z)=0$, $g_0(x)=0$ and
\begin{align*}
w_1(z)&=\hgf21{1,1}{\frac32}{z}=\frac1{\sqrt{1-z}}\frac{\arcsin\sqrt z}{\sqrt z}, \\
g_1(x)&=\hgf21{\frac12,1}{\frac32}{-x}=\frac{\arctan\sqrt z}{\sqrt z}.
\end{align*}
For later use, we notice two differential equations for $w_1(z)$:
\begin{align}
\left\{z(1-z)\frac{d^2}{dz^2}+\frac32(1-2z)\frac{d}{dz}-1\right\}w_1(z)&=0, \\
\left\{2z(1-z)\frac{d}{dz}+1-2z\right\}w_1(z)&=1.
\end{align}

Let us translate the formula \eqref{eq:recurrence-relation}
into the differential equations for the generating functions $w_k(z)$.
We have
\begin{multline*}
\left\{z(1-z)^2\frac{d^2}{dz^2}+(1-z)(1-3z)\frac{d}{dz}+z-\frac34\right\}\sum_{n=0}^\infty \J{k}{n}z^n\\
=\J{k}{1}-\frac34\J{k}{0}+\frac14\sum_{n=2}^\infty(4n^2\J{k}{n}-(8n^2-8n+3)\J{k}{n-1}+4(n-1)^2\J{k}{n-2})z^{n-1}.
\end{multline*}
Using \eqref{eq:recurrence-relation}
and $\J{k}{1}=\J{k-2}{0}+\frac34\J{k}{0}$, we obtain
\begin{thm}
One has\label{thm:differential-equation-for-w}
\begin{equation}\label{eq:differential-equation-for-w}
\left\{z(1-z)^2\frac{d^2}{dz^2}+(1-z)(1-3z)\frac{d}{dz}+z-\frac34\right\}w_k(z)=w_{k-2}(z)
\end{equation}
for $k\ge2$. \qed
\end{thm}

\begin{rem}
We have
\begin{equation*}
\left\{z(1-z)^2\frac{d^2}{dz^2}+(1-z)(1-3z)\frac{d}{dz}+z-\frac34\right\}^kw_{2k}(z)=0
\end{equation*}
and
\begin{equation*}
\left\{z(1-z)\frac{d^2}{dz^2}+\frac32(1-2z)\frac{d}{dz}-1\right\}
\left\{z(1-z)^2\frac{d^2}{dz^2}+(1-z)(1-3z)\frac{d}{dz}+z-\frac34\right\}^kw_{2k+1}(z)=0
\end{equation*}
for each $k\ge0$.
Namely, $w_k(z)$ is a power series solution of a linear differential equation, which is holomorphic at $z=0$.
\end{rem}

To find an explicit formula for $\J{k}{n}$, it is useful to introduce the function
\begin{equation}\label{eq:definition of v_k(t)}
v_k(t)=(1-z)w_k(z),\quad t=\frac{z}{z-1}\quad
\left(\iff w_k(z)=(1-t)v_k(t),\quad z=\frac{t}{t-1}\right).
\end{equation}
Note that
\begin{align*}
v_2(t)&=\J{2}{0}\cdot\hgf21{\frac12,\frac12}{1}{t}
=\J{2}{0}\sum_{n=0}^\infty\binom{-\frac12}{n}^{\!\!2}t^n, \\
v_{1}(t)&=\frac1{1-t}\hgf21{1,1}{\frac32}{\frac{t}{t-1}}
=\sum_{n=0}^\infty\frac{t^n}{2n+1}.
\end{align*}
The formula \eqref{eq:differential-equation-for-w} is translated as
\begin{equation}\label{eq:differential-equation-for-v}
\left\{t(1-t)\frac{d^2}{dt^2}+(1-2t)\frac{d}{dt}-\frac14\right\}v_k(t)=-v_{k-2}(t).
\end{equation}

Let us look at the (hypergeometric differential) operator
\begin{equation*}
D=t(1-t)\frac{d^2}{dt^2}+(1-2t)\frac{d}{dt}-\frac14.
\end{equation*}
It is straightforward to check that
\begin{equation*}
p_n(t)=-\frac1{(n+\frac12)^2}\binom{-\frac12}{n}^{\!\!-2}\sum_{k=0}^n\binom{-\frac12}{k}^{\!\!2}t^k,
\end{equation*}
satisfies the equation
$Dp_n(t)=t^n$
by using the fact
\begin{equation*}
D t^n=-\frac{(2n+1)^2}4t^n+n^2t^{n-1}
\end{equation*}
(see \S4 of \cite{KW2006KJM}).
Thus, if we put
\begin{equation*}
\xi_l(t)=\sum_{n=0}^\infty\binom{-\frac12}{n}^{\!\!2} \cf ln t^n\quad(l\ge0),
\end{equation*}
then
\begin{equation*}
D\left\{-\sum_{n=0}^\infty\binom{-\frac12}{n}^{\!\!2} \cf ln p_n(t)\right\}=-\xi_l(t).
\end{equation*}
On the other hand, we see that
\begin{equation*}
-\sum_{n=0}^\infty\binom{-\frac12}{n}^{\!\!2} \cf ln p_n(t)
=\sum_{n=0}^\infty \cf ln
\frac1{(n+\frac12)^2}\sum_{k=0}^n\binom{-\frac12}{k}^{\!\!2}t^k
=\sum_{k=0}^\infty\binom{-\frac12}{k}^{\!\!2}
\left\{\sum_{n=k}^\infty\frac{\cf ln}{(n+\frac12)^2}\right\}t^k.
\end{equation*}
Hence, if we \emph{assume} that the numbers $\cf lk$ satisfy the condition
\begin{equation}\label{eq:condition_among_a}
\cf{l+2}k=\sum_{n=k}^\infty\frac{\cf ln}{(n+\frac12)^2},
\end{equation}
then the functions $\xi_l(t)$ satisfy the relation
\begin{equation*}
D\xi_{l+2}(t)=-\xi_l(t)\quad(l\ge0).
\end{equation*}
Notice that we have
\begin{equation}\label{eq:extension_of_cf}
\cf{l+2m}n=\sum_{n\le s_1\le s_2\le\dots\le s_m}\frac{\cf l{s_m}}{(s_1+\frac12)^2(s_2+\frac12)^2\dotsb(s_m+\frac12)^2}
\end{equation}
under the assumption \eqref{eq:condition_among_a}.

Now we determine the numbers $\cf ln$ so that they satisfy \eqref{eq:condition_among_a}.
If we set
\begin{equation*}
\cf ln=\begin{cases}
\displaystyle
\frac12\frac1{n+\frac12}\binom{-\frac12}{n}^{\!\!-2} & l=1,\\[1em]
\displaystyle
\J20 & l=2
\end{cases}
\end{equation*}
and extend by the relation \eqref{eq:extension_of_cf},
then the relation \eqref{eq:condition_among_a} is surely satisfied.
We remark that the series \eqref{eq:extension_of_cf} indeed converges
since $A_{1,n}$ and $A_{2,n}$ are bounded so that
the positive series $A_{l+2m,n}$ is dominated by a constant multiple of the series
(\emph{multiple zeta-star value})
\begin{equation*}
\zeta^\star_m(2,2,\dots,2)=
\sum_{0<k_1\le k_2\le\dots\le k_m}(k_1k_2\dots k_m)^{-2}.
\end{equation*}
Notice that
\begin{align*}
\xi_1(t)&=\frac1{1-t}\hgf21{1,1}{\frac32}{\frac{t}{t-1}}=v_1(t),\\
\xi_2(t)&=\J20\cdot\hgf21{\frac12,\frac12}{1}{t}=v_2(t).
\end{align*}

From the discussion above we have the following (see \cite{KW2012RIMS} for the proof).

\begin{prop}
There exist constants $C_j$ $(j=1,2,\ldots)$ such that $v_l(t)$ is given by
\begin{equation}\label{v_l}
v_l(t)=\xi_l(t)+\sum_{0<j< l/2}C_{l-2j}v_{2j}(t).
\end{equation}
Moreover, the coefficients $C_{l-2}, C_{l-4}, \dots$ are determined inductively.
\end{prop}

From this proposition, we observe
\begin{equation}\label{w_l}
w_l(z)=\frac1{1-z}\xi_l\!\left(\frac{z}{z-1}\right)+\sum_{0<j< l/2}C_{l-2j}w_{2j}(z),
\end{equation}
and in particular
\begin{equation}\label{explicitJ}
\J ln=\sum_{k=0}^n(-1)^k\binom{-\frac12}{k}^{\!\!2}\binom{n}{k}A_{l,k}+ \sum_{0<j< l/2}C_{l-2j}\J {2j}n.
\end{equation}
By this equation, we can determine $C_{l-2j}$ by putting $n=0$ inductively
and obtain explicit formulas of $\J ln$ for each $l$.
We give first few examples.

\begin{ex}
For $l=2,3,4$, we have
\begin{align*}
\J2n&=\zeta\Bigl(2,\frac12\Bigr)\sum_{k=0}^n(-1)^k\binom{-\frac12}{k}^{\!\!2}\binom{n}{k},\\
\J3n&=-\frac12\sum_{k=0}^n(-1)^k\binom{-\frac12}{k}^{\!\!2}\binom{n}{k}\sum_{0\le j<k}\frac{1}{(j+\frac12)^3}\binom{-\frac12}{j}^{\!\!-2}
+2\zeta\Bigl(3,\frac12\Bigr)\sum_{k=0}^n(-1)^k\binom{-\frac12}{k}^{\!\!2}\binom{n}{k},\\
\J4n&=-\zeta\Bigl(2,\frac12\Bigr)\sum_{k=0}^n(-1)^k\binom{-\frac12}{k}^{\!\!2}\binom{n}{k}\sum_{0\le j<k}\frac{1}{(j+\frac12)^2}
+3\zeta\Bigl(4,\frac12\Bigr)\sum_{k=0}^n(-1)^k\binom{-\frac12}{k}^{\!\!2}\binom{n}{k}.
\end{align*}
\end{ex}

\subsection{Numerical data of normalized Ap\'ery-like sequences}\label{sec:numeric_data_of_NAS}

In this section, certain numerical data of the Ap\'ery-like sequences is presented.

The \emph{normalized Ap\'ery-like numbers} $\tJ{k}n$ are defined by the conditions
\begin{align*}
\J{2s}n&=\sum_{j=0}^{s-1} \J{2s-2j}0\tJ{2j+2}n, \\
\J{2s+1}n&=\sum_{j=0}^{s-1} \J{2s+1-2j}0\tJ{2j+2}n+\tJ{2s+1}n
\end{align*}
inductively.
It is equivalent to define $\tJ kn$ by the recurrence relation
\begin{equation*}
\tJ kn=\J{r}0^{-1}\Bigl(\J kn-\sum_{0<2j<k} \J{k-2j}0\tJ{2j+2}n\Bigr),\qquad
r=\begin{cases}
1 & k\equiv1\pmod2, \\[.2em]
2 & k\equiv0\pmod2.
\end{cases}
\end{equation*}
The numbers $\tJ kn$ satisfy the relation
\begin{equation}\label{eq:reccurence-for-tJ}
4n^2\tJ kn-(8n^2-8n+3)\tJ k{n-1}+4(n-1)^2\tJ k{n-2}=4\tJ{k-2}{n-1}.
\end{equation}
Notice that this is identical to the one for $J_k(n)$.
It is elementary to check that
\begin{equation*}
\tJ1n=\J1n=\frac{2^nn!}{(2n+1)!!},\qquad \tJ2n=\frac{\J2n}{\J20}=\sum_{j=0}^n(-1)^j\binom{-\frac12}j^2\binom nj
\end{equation*}
and
\begin{equation*}
\tJ k0=\begin{cases}
1 & k=1,2, \\
0 & \text{otherwise},
\end{cases}
\qquad
\tJ k1=\begin{cases}
\dfrac23 & k=1,\\[.7em]
\dfrac34 & k=2,\\[.5em]
1 & k=3,4, \\
0 & \text{otherwise}.
\end{cases}
\end{equation*}
These are all \emph{rational numbers}.
Hence, by the recurrence relation \eqref{eq:reccurence-for-tJ} for $\tJ kn$,
all the normalized Ap\'ery-like numbers $\tJ kn$ are rational.

Let us put
\begin{align*}
\evenZ_s(k)&:=(-1)^s\sum_{k>j_1>\dots>j_s\ge0}\frac1{(j_1+\frac12)^2\dots(j_s+\frac12)^2}, \\
\oddZ_s(k)&:=\frac{(-1)^s}2\sum_{k>j_1>\dots>j_s\ge0}\frac1{(j_1+\frac12)^2\dots(j_{s-1}+\frac12)^2(j_s+\frac12)^3}\binom{-\frac12}{j_s}^{\!\!-2}
\end{align*}
for $s=1,2,3,\dots$.
We also set $\evenZ_0(k)=1$ for convenience.
Then we have the
\begin{thm}
For $s=1,2,3,\dots$, we have
\begin{align*}
\tJ{2s+2}n&=\sum_{k=0}^n(-1)^k\binom{-\frac12}{k}^{\!\!2}\binom{n}{k}\evenZ_s(k), \\
\tJ{2s+1}n&=\sum_{k=0}^n(-1)^k\binom{-\frac12}{k}^{\!\!2}\binom{n}{k}\oddZ_s(k).
\end{align*}
\end{thm}

\begin{proof}
Define the numbers $\tA_{l,n}$ by the relation \eqref{eq:condition_among_a} satisfied by $A_{l,n}$
together with the normalized initial condition
\begin{equation*}
\tA_{1,n}=A_{1,n}=\frac12\frac1{n+\frac12}\binom{-\frac12}n^{\!-2},\quad
\tA_{2,n}=1.
\end{equation*}
We immediately have
\begin{align*}
\evenY_s(n):=\tA_{2s+2,n}&=\sum_{n\le j_1\le \dots\le j_s}\frac1{(j_1+\frac12)^2\dots(j_s+\frac12)^2}, \\
\oddY_s(n):=\tA_{2s+1,n}&=\frac12\sum_{n\le j_1\le \dots\le j_s}\frac1{(j_1+\frac12)^2\dots(j_s+\frac12)^3}\binom{-\frac12}{j_s}^{\!-2}.
\end{align*}
By the same discussion as in the previous section, we see that there exist certain numbers $C_j$ $(j=1,2,3,\dots)$ such that
\begin{equation}\label{eq:tJ-formula-1}
\tJ ln=\sum_{k=0}^n(-1)^k\binom{-\frac12}k^{\!2}\binom nk\tA_{l,k}
+\sum_{0<j<l/2}C_{l-2j}\tJ{2j}n.
\end{equation}
Put $n=0$ in \eqref{eq:tJ-formula-1}, we have
$0=\tA_{l,0}+C_{l-2}$ if $l\ge3$
since $\tJ20=1$ and $\tJ k0=0$ if $k>2$.
Thus we see that $\tJ ln$ are of the form
\begin{equation*}
\tJ ln=\sum_{k=0}^n(-1)^k\binom{-\frac12}k^{\!2}\binom nk\tB_{l,k}
\end{equation*}
with
\begin{equation*}
\tB_{l,k}=\tA_{l,k}-\sum_{0<j<l/2}\tA_{l-2j+2,0}\tB_{2j,k},\qquad
\tB_{2,k}=\tA_{2,k},\quad \tB_{3,k}=\tA_{3,k}.
\end{equation*}
Therefore it is enough to show that $\evenZ_s(k)$'s and $\oddZ_s(k)$'s satisfy the relations
\begin{align}
\evenZ_s(k)&=\evenY_s(k)-\sum_{j=0}^{s-1}\evenY_{s-j}(0)\evenZ_j(k), \label{eq:relation for evenZ}\\
\oddZ_s(k)&=\oddY_s(k)-\sum_{j=0}^{s-1}\oddY_{s-j}(0)\evenZ_j(k). \label{eq:relation for oddZ}
\end{align}
Assume that $s\ge2$, since these are directly proved when $s=1$.
We only prove \eqref{eq:relation for oddZ} by induction on $k$ (the proof of \eqref{eq:relation for evenZ} is parallel).
If $k=0$, then the both sides of \eqref{eq:relation for oddZ} is zero.
Suppose that \eqref{eq:relation for oddZ} is true for $k$.
Notice that
\begin{align*}
\evenZ_s(k+1)&=-\frac1{(k+\frac12)^2}\evenZ_{s-1}(k)+\evenZ_s(k), \\
\evenY_s(k+1)&=-\frac1{(k+\frac12)^2}\evenY_{s-1}(k)+\evenY_s(k), \\
\oddY_s(k+1)&=-\frac1{(k+\frac12)^2}\oddY_{s-1}(k)+\oddY_s(k).
\end{align*}
Using these relations together with the induction assumption,
it is straightforward to verify that the both sides of \eqref{eq:relation for oddZ} for $k+1$ coincide.
\end{proof}

\begin{rem}
Note that
\begin{equation*}
\evenZ_s(n)=\oddZ_s(n)=0\quad(0\le n<s),\qquad
\evenZ_s(s)=\oddZ_s(s)=\frac1{(s!)^2}\binom{-\frac12}s^{\!\!-2},
\end{equation*}
and hence
\begin{equation*}
\tJ{2s+2}n=\tJ{2s+1}n=0\quad(0\le n<s),\qquad
\tJ{2s+2}s=\tJ{2s+1}s=\frac1{(s!)^2}.
\end{equation*}
\end{rem}

We now provide several numerical data of $\tJ kn$:

\begin{flushleft}

\data1
{
1,\frac{2}{3},\frac{8}{15},\frac{16}{35},\frac{128}{315},\frac{256}{693},\frac{1024}{3003},\frac{2048}{6435},\frac{32768}{109395}
}

\data2
{
1,\frac{3}{4},\frac{41}{64},\frac{147}{256},\frac{8649}{16384},\frac{32307}{65536},\frac{487889}{1048576},\frac{1856307}{4194304},\frac{454689481}{1073741824}
}

\data3
{
0,1,\frac{65}{48},\frac{13247}{8640},\frac{704707}{430080},\frac{660278641}{387072000},\frac{357852111131}{204374016000},\frac{309349386395887}{173581664256000}
}

\data4
{
0,1,\frac{11}{8},\frac{907}{576},\frac{1739}{1024},\frac{6567221}{3686400},\frac{54281321}{29491200},\frac{7260544493}{3853516800},\frac{709180003579}{369937612800}
}

\data5
{
0,0,\frac{1}{4},\frac{109}{216},\frac{101717}{138240},\frac{4557449}{4838400},\frac{15689290781}{13934592000},\frac{131932666373}{102187008000},\frac{144010453389429161}{99983038611456000}
}

\data6
{
0,0,\frac{1}{4},\frac{73}{144},\frac{3419}{4608},\frac{29273}{30720},\frac{151587391}{132710400},\frac{232347221}{176947200},\frac{2444144299823}{1664719257600}
}

\data7
{
0,0,0,\frac{1}{36},\frac{515}{6912},\frac{76667}{576000},\frac{115560397}{580608000},\frac{1051251017}{3901685760},\frac{18813135818903}{54935735500800}
}

\data8
{
0,0,0,\frac{1}{36},\frac{43}{576},\frac{15389}{115200},\frac{1659311}{8294400},\frac{251914357}{928972800},\frac{10258433947}{29727129600}
}
\end{flushleft}

\subsection{Congruence of normalized Ap\'ery-like numbers}\label{CongruenceProperty}

The congruence properties of $\tJ k2$ (and $\tJ k3$) obtained in \cite{KW2006KJM} (see also \cite{LOS2014}) are considered to be  one of the consequences of the modular property that the generating function $w_2$ possesses (i.e. $w_2$ is an automorphic form for $\Gamma(2)(\cong \Gamma_0(4)$)). As we will show in \S\ref{Modular},
there is a ``weak modularity" for $w_{2n}$ (i.e.\ $w_{2n}$ is an automorphic integral for $\hecke$: see \S\ref{AI}). Therefore we may expect similar congruence properties among $\tJ kn \, (n\geq4)$. In fact, we provide below a certain reasonable conjecture on congruence relations among $\tJ kn$. The aim of this subsection is to show some weak and restricted version of the conjecture.

Based on a numerical experiment,
we conjecture that the following congruence relations among the normalized Ap\'ery-like numbers should hold.
\begin{conjecture}\label{GeneralCongruence}
For positive integers $m,n,s$ such that $mp\ge s$ and $mp^{n-1}\ge s$, we have
\begin{align*}
\frac{p^{(2s+1)n}\tJ{2s+1}{mp^n}}{p^{2s+1}\tJ{2s+1}{mp}}
&\equiv \frac{p^{(2s+1)(n-1)}\tJ{2s+1}{mp^{n-1}}}{p^{2s+1}\tJ{2s+1}{mp}} (\not\equiv0) \pmod{p^n}, \\
\frac{p^{2sn}\tJ{2s+2}{mp^n}}{p^{2s}\tJ{2s+2}{mp}}
&\equiv \frac{p^{2s(n-1)}\tJ{2s+2}{mp^{n-1}}}{p^{2s}\tJ{2s+2}{mp}} (\not\equiv0) \pmod{p^n}.
\end{align*}
\end{conjecture}

\begin{rem}
When $m<\frac p2$, the denominator of $p^{2s}\tJ{2s+2}{mp}$ is indivisible by $p$, that is,
$p^{2s}\tJ{2s+2}{mp}\equiv p^ta \pmod{p^n}$ for some $t\in\Z_{\ge0}$
and $a\in\Z\setminus p\Z$.
In this case, the second one in the conjecture above is equivalent to
\begin{equation*}
p^{2sn-t}\tJ{2s+2}{mp^n}\equiv p^{2s(n-1)-t}\tJ{2s+2}{mp^{n-1}} \pmod{p^n}.
\end{equation*}
\end{rem}


Here we prove slightly weaker results (Theorem \ref{weaker version}).
\emph{In what follows in this subsection, $p$ always denotes an odd prime.}
We recall the following basic congruences on binomial coefficients (see (6.7), (6.12) and (6.13) in \cite{KW2006KJM}).

\begin{lem}
For any positive integers $m,n,j$,
the following congruence relations hold:
\begin{equation*}
\binom{-\frac12}{pj}^{\!2}\binom{mp^n}{pj}\equiv\binom{-\frac12}j^{\!2}\binom{mp^{n-1}}j \pmod {p^n},
\end{equation*}
\begin{equation*}
p\nmid j \then
\binom{mp^n}j\equiv 0 \pmod{p^n}.
\end{equation*}
\end{lem}

%

We also need the following elementary facts.

\begin{lemma}\label{ord_p of binom^-2}
Let $\ord_p x$ be the exponent of $p$ in $x\in\Q$,
i.e. $x=\prod_p p^{\ord_px}$ for $x\in\Q$.
If $1\le 2j+1<p^{n+1}$, then
\begin{equation*}
\ord_p\binom{-\frac12}j\le n-\ord_p(2j+1).
\end{equation*}
\end{lemma}

\begin{proof}
Put $r=\ord_p(2j+1)$.
Then there is some odd integer $m$ such that $2j+1=mp^r<p^{n+1}$.
In general, we see that
\begin{equation*}
\ord_p \binom{-\frac12}j
=\ord_p \binom{2j}j
=\sum_{l\ge1}\Bigl(\floor{\frac{2j}{p^l}}-2\floor{\frac{2j}{p^l}}\Bigr)
=\#\set[bigg]{l\ge1}{\fracpart{\frac j{p^l}}\ge\frac12},
\end{equation*}
where $\fracpart x=x-\floor x$ is the fractional part of $x\in\R$.
Notice that $\fracpart{x+1}=\fracpart x$.
It follows then
\begin{equation*}
1\le l \le r \then
\fracpart{\frac j{p^l}}
=\fracpart{\frac{mp^r-1}{2p^l}}
=\fracpart{\frac{m}2p^{r-l}-\frac1{2p^l}}
=\fracpart{\frac12-\frac1{2p^l}}<\frac12,
\end{equation*}
and
\begin{equation*}
l\ge n+1 \then 0\le \frac j{p^l}<\frac{p^{n+1-l}}2\le\frac12 \then \fracpart{\frac j{p^l}}<\frac12.
\end{equation*}
Thus we have
\begin{equation*}
\ord_p\binom{-\frac12}j=\#\set[bigg]{l\ge1}{\fracpart{\frac j{p^l}}\ge\frac12,\; r<l\le n}\le n-r
\end{equation*}
as desired.
\end{proof}

%

\begin{lem}
For $k=0,1,\dots,\frac{p^n-1}2$,
\begin{equation}
p^{2sn}\evenZ_s(kp)\equiv p^{2s(n-1)}\evenZ_s(k) \pmod{p^n} \label{eq:congruence for evenZ}
\end{equation}
holds.
\end{lem}

Notice that the denominator of $p^{2sn}\evenZ_s(k)$ is not divisible by $p$ if $2k-1<p^{n+1}$.

\begin{proof}
First we notice that
\begin{equation*}
\set{2j+1}{0\le j<kp}\cap p\Z
=\set{p(2j+1)}{0\le j<k}.
\end{equation*}
In the sum
\begin{equation*}
p^{2sn}\evenZ_s(kp)=(-1)^s\sum_{kp>j_1>\dots>j_s\ge0}\frac{p^{2sn}}{(j_1+\frac12)^2\dots(j_s+\frac12)^2},
\end{equation*}
the summand is $\equiv0 \pmod{p^n}$ if any of $2j_1+1,\dots,2j_s+1$ is indivisible by $p$.
Hence we have
\begin{align*}
(-1)^sp^{2sn}\evenZ_s(kp)
&=\sum_{k>j_1>\dots>j_s\ge0}\frac{p^{2sn}}{(j_1+\frac12)^2\dots(j_s+\frac12)^2} \\
&\equiv\sum_{k>j_1>\dots>j_s\ge0}\frac{p^{2sn}}{(p(j_1+\frac12))^2\dots(p(j_s+\frac12))^2} \pmod{p^n} \\
&=\sum_{k>j_1>\dots>j_s\ge0}\frac{p^{2s(n-1)}}{(j_1+\frac12)^2\dots(j_s+\frac12)^2}
=(-1)^sp^{2s(n-1)}\evenZ_s(k),
\end{align*}
which implies \eqref{eq:congruence for evenZ}.
\end{proof}

%

\begin{thm}\label{weaker version}
If $1\le m<\frac p2$, then
\begin{equation}
p^{2sn}\tJ{2s+2}{mp^n}\equiv p^{2s(n-1)}\tJ{2s+2}{mp^{n-1}} \pmod{p^n}
\end{equation}
holds.
\end{thm}

\begin{proof}
Using the lemma above, we have
\begin{align*}
p^{2sn}\tJ{2s+2}{mp^n}
&=\sum_{k=0}^{mp^n}(-1)^k\binom{-\frac12}{k}^{\!\!2}\binom{mp^n}{k}p^{2sn}\evenZ_s(k) \\
&\equiv\sum_{k=0}^{mp^{n-1}}(-1)^{kp}\binom{-\frac12}{kp}^{\!\!2}\binom{mp^n}{kp}p^{2sn}\evenZ_s(kp) \pmod{p^n} \\
&\equiv\sum_{k=0}^{mp^{n-1}}(-1)^{k}\binom{-\frac12}{k}^{\!\!2}\binom{mp^{n-1}}{k}p^{2s(n-1)}\evenZ_s(k) \pmod{p^n} \\
&=p^{2s(n-1)}\tJ{2s+2}{mp^{n-1}}
\end{align*}
as desired.
\end{proof}

\begin{remark}[Odd case]
We expect that the congruence formula in Theorem \ref{weaker version} also holds for odd case.
Explicitly, we conjecture that
\begin{equation}
p^{(2s+1)n}\tJ{2s+1}{mp^n}\equiv p^{(2s+1)(n-1)}\tJ{2s+1}{mp^{n-1}} \pmod{p^n}
\end{equation}
holds for $1\le m<\frac p2$.
This is reduced to the congruence
\begin{equation}
p^{(2s+1)n}\oddZ_s(kp)\equiv p^{(2s+1)(n-1)}\oddZ_s(k) \pmod{p^n}
\end{equation}
as in the even case.
To prove this, we need the following fact, which we have not succeeded to prove:
Let $j'=\frac{p(2j+1)-1}2$, i.e. $2j'+1=p(2j+1)$. Then
\begin{equation}
\frac{p^{3(n-1)}}{(j+\frac12)^3}
\Biggl(\binom{-\frac12}{j'}^{\!\!-2}-\binom{-\frac12}{j}^{\!\!-2}\Biggr)
\equiv0\pmod{p^n}
\end{equation}
when $1\le 2j+1<p^n$.
We note that by an elementary discussion, this congruence is reduced to
\begin{equation}\label{to be proved}
\ord_p\frac1{2j+1}\Biggl(1-\frac{\binom{-\frac12}{j'}^{\!2}}{\binom{-\frac12}{j}^{\!2}}\Biggr)\ge1,
\end{equation}
where $\ord_px$ for $x\in\Q$ is the exponent of $p$ in $x$, i.e. $x=\prod_p p^{\ord_px}$.
If $2j+1=mp^r$ $(p\nmid m)$ and $s=\ord_p\binom{2j}j$, then \eqref{to be proved} is equivalent to
\begin{equation}\label{to be proved 2}
\binom{2j'}{j'}\equiv(-1)^{\frac{p-1}2}\binom{2j}{j} \pmod{p^{s+r+1}}.
\end{equation}
We note that the \emph{modulo $p^{r+1}$} version of the congruence \eqref{to be proved 2} can be proved easily
as we sketch in the following.
By a repeated use of the binomial theorem (see, e.g. \cite{Beu1987}), we get
\begin{equation*}
(1-X)^{mp^{r+1}-1}\equiv (1-X^p)^{mp^r-1}\sum_{j=0}^{p-1}X^j \pmod {p^{r+1}}.
\end{equation*}
Comparing the coefficients of $X^{ap+b}$ $(0\le b<p)$ in the both sides, we have
\begin{equation*}
\binom{mp^{r+1}-1}{ap+b}\equiv(-1)^b\binom{mp^{r}-1}a \pmod{p^{r+1}}
\end{equation*}
in general. In particular, when $a=j$ and $b=\frac{p-1}2$, we obtain
\begin{equation*}
\binom{2j'}{j'}\equiv(-1)^{\frac{p-1}2}\binom{2j}{j} \pmod{p^{r+1}}.
\end{equation*}
\end{remark}

\subsection{A remark on Euler's constant for the NcHO}

We know that the spectral zeta function $\zeta_Q(s)$ can be meromorphically continued to the whole complex plane with unique pole at $s=1$ \cite{IW2005a}. Actually, it has a simple pole at $s=1$ with residue $\frac{\alpha+\beta}{\sqrt{\alpha\beta(\alpha\beta-1)}}$. By this fact, it would be reasonable to define the Euler(-Mascheroni) constant $\gamma_Q$ for the NcHO by
\begin{equation*}
\gamma_Q:= \lim_{s\to1}\Big\{\zeta_Q(s)-\frac{\alpha+\beta}{\sqrt{\alpha\beta(\alpha\beta-1)}}
\frac1{s-1}\Big\}.
\end{equation*}
Since we can not expect neither an Euler product nor functional equation for $\zeta_Q(s)$, the analysis and results developed e.g.\ in \cite{HIKW20014} for the Dedekind and Selberg zeta function is seemingly difficult. Nevertheless, we expect that $\gamma_Q$ may possess certain arithmetic significance like Kronecker's limit formula \cite{Siegel1961}, since it can be regarded as a regularized value of ``$\zeta_Q(1)$''. Exploring this problem would be desirable to obtain new information of the spectrum.

\section{Ap\'ery-like numbers and Mahler measures}\label{MGF}

In this section, we observe certain mysterious relation between our Ap\'ery-like numbers and
the modular Mahler measures discussed in \cite{RV1999} through a generating function of the generating functions $v_k(t)$ of the Ap\'ery-like numbers.

\subsection{Meta-generating functions}

We study a generating function of $v_k(t)$ (sometimes we refer it as the meta-generating (grandmother) function of Ap\'ery-like numbers) as
\begin{align*}
V^e(t,\lambda)&:= \sum_{k=0}^\infty v_{2k+2}(t)(-1)^k\lambda^{2k},\\
V^o(t,\lambda)&:= \sum_{k=0}^\infty v_{2k+1}(t)(-1)^k\lambda^{2k}.
\end{align*}
For a time being, we will force on the even meta-generating function $V^e(t,\lambda)$.
Since $v_k(0)=w_k(0)=J_k(0)=(k-1)\zeta(k,1/2)$, we have
\begin{equation*}
V^e(0,\lambda)=\sum_{k=0}^\infty (2k+1)\zeta\Bigl(2k+2,\frac12\Bigr)(-1)^k\lambda^{2k}
=\frac{\pi^2}{2\cosh^2\pi\lambda}.
\end{equation*}

\begin{lem}\label{HolomorphyV}
For a sufficiently small $\abs{\lambda}$, the function $V^e(t,\lambda)$ $($resp. $V^o(t,\lambda)$$)$ in the variable $t$
is holomorphic around $t=0$.
\end{lem}

\begin{proof}
Recall the integral expression \eqref{eq:definition-of-aperylike-numbers} of $J_k(n)$
\begin{equation*}
J_k(n)
=\frac1{2^{2n+1}}\int_0^\infty\frac{u^{k-2}}{(k-2)!}\frac{e^{nu}}{(\sinh\frac{u}2)^{2n+1}}du
\int_0^u(1-e^{-2t})^n(1-e^{-2(u-t)})^ndt.
\end{equation*}
Since
\begin{equation*}
(1-e^{-2t})^n(1-e^{-2(u-t)})^n \leq (1-e^{-u})^{2n}
\end{equation*}
one observes that
\begin{equation*}
0< J_k(n)=\frac{1}{(k-2)!}\int_0^\infty u^{k-2} e^{-\frac{u}2} \frac{u}{1-e^{-u}}du.
\end{equation*}
Since $\frac{u}{1-e^{-u}}< \frac{1}{1-e^{-1}}<2$ for $0<u<1$ and $\frac{u}{1-e^{-u}}< 2u$ for $u>1$,
for any $\varepsilon>0$, one sees that there exists a constant $C_\varepsilon$ such that
\begin{align*}
\int_0^\infty u^{k-2} e^{-\frac{u}2} \frac{u}{1-e^{-u}}du
&{}< 2\int_0^1 u^{k-2} e^{-\frac{u}2}du + 2\int_1^\infty u^{k-1} e^{-\frac{u}2} du\\
&{}<C_\varepsilon \int_0^\infty u^{k-2} e^{-\frac{u}{2+\varepsilon}} du
= C_\varepsilon (2+\varepsilon)^{k-1}\Gamma(k-1).
\end{align*}
It follows that $0< J_k(n) < C_\varepsilon (2+\varepsilon)^{k-1}$ (independent of $n$).
Since $w_k(z)= \sum_{n=0}^\infty J_k(n)z^n$, one has
\begin{equation*}
\abs{w_k(z)}\leq C_\varepsilon (2+\varepsilon)^{k-1}(1-\abs{z})^{-1} \quad (\abs{z}<1).
\end{equation*}
Recall that $v_k(t)=(1-z)w_k(z)$ with $z=\frac{t}{t-1}$. Hence one obtains
\begin{equation*}
\abs{v_k(t)}\leq C_\varepsilon (2+\varepsilon)^{k-1}(\abs{t-1}-\abs{t})^{-1} \quad \bigl(\Re(t)<\frac12\bigr).
\end{equation*}
This immediately shows that if $\abs{\lambda}< 1/\sqrt{2+\varepsilon}$
\begin{equation*}
\abs{V^e(t,\lambda)}\leq \sum_{k=0}^\infty \abs{v_{2k+2}(t)(-1)^k\lambda^{2k}}
< C_\varepsilon (\abs{t-1}-\abs{t})^{-1}(1-(2+\varepsilon)\abs{\lambda}^2)^{-1}.
\end{equation*}
This shows the assertion for $V^e(t,\lambda)$. For the odd parity function $V^o(t,\lambda)$, the proof is the same.
\end{proof}

It follows from the equation \eqref{eq:differential-equation-for-v} that
\begin{equation}\label{Meta-Gen eigen}
D_t V^e(t,\lambda)= \lambda^2 V^e(t,\lambda).
\end{equation}
Namely, the function $V^e(t,\lambda)$ is an eigenfunction of $D_t$ with eigenvalue $\lambda^2$.
Note that $v_2(t)=V^e(t,0)$ is a modular form for $\Gamma_0(4)$.

\begin{rem}
Similarly to the even case, we have
\begin{equation*}
D_tV^o(t,\lambda)
=\sum_{k=0}^\infty (-1)^kD_tv_{2k+1}(t)\lambda^{2k}
=\sum_{k=0}^\infty (-1)^{k-1}v_{2k-1}(t)\lambda^{2k}
=-v_{-1}(t)+\lambda^2 V^o(t,\lambda),
\end{equation*}
where
\begin{align*}
v_{-1}(t)=-D_tv_1(t)
&=-\frac14\sum_{n=0}^\infty\Bigl\{4(n+1)^2\frac1{2n+3}-(2n+1)^2\frac1{2n+1}\Bigr\}t^n \\
&=-\frac14\sum_{n=0}^\infty\frac{t^n}{2n+3}=-\frac14\frac{v_1(t)-1}t.
\end{align*}
Namely one has
\begin{equation*}
(D_t-\lambda^2)V^o(t,\lambda)=\frac1{4t}\Bigg\{\frac1{1-t}\hgf21{1,1}{\frac32}{\frac{t}{t-1}}-1\Bigg\}.
\end{equation*}
\end{rem}

Rewrite \eqref{Meta-Gen eigen} as
\begin{equation*}
\Big(t(1-t)\frac{d^2}{dt^2}+(1-2t)\frac{d}{dt}-\frac14-\lambda^2\Big)V^e(t,\lambda)=0.
\end{equation*}
Since $V^e(t,\lambda)$ is holomorphic around $t=0$ and $v_k(0)=(k-1)\zeta(k,1/2)$,
one has
\begin{equation}\label{HGEofMeta}
V^e(t,\lambda)=\frac{\pi^2}{2\cosh^2\pi\lambda}\hgf21{\frac12+i\lambda, \frac12-i\lambda}{1}{t}.
\end{equation}
From this, it is immediate to see that
\begin{equation}
v_{2k+2}(t)= (2k)! \frac{d^{2k}}{d\lambda^{2k}}
\left\{\frac{\pi^2}{2\cosh^2\pi\lambda}\hgf21{\frac12+i\lambda, \frac12-i\lambda}{1}{t}\right\}
\bigg|_{\lambda=0}.
\end{equation}

In relation with the modular form interpretation of the generating function of Ap\'ery-like numbers for $\zeta_Q(2)$ developed in \cite{KW2007} (see \S\ref{Modular}) and \cite{Z}, we naturally come to the following problems.

\begin{prob}
Determine whether there are any pair of $\lambda \in \C$ and $k \in \N$ such that the function
$\frac{d^{2k}}{d\lambda^{2k}}
\frac{\pi^2}{2\cosh^2\pi\lambda}\hgf21{\frac12+i\lambda, \frac12-i\lambda}{1}{t(\tau)}$
in $\tau$ can be a modular form for some congruence subgroup $\congsubgp$ of $SL_2(\Z)$ and for
some modular function $t=t(\tau)$ for $\congsubgp$.
\end{prob}

\begin{prob}
For some $t=t(\tau)$, is there any
$\lambda$ such that $\hgf21{\frac12+i\lambda, \frac12-i\lambda}{1}{t(\tau)}$
is a modular form for some congruence subgroup $\congsubgp$ of $SL_2(\Z)$?
Moreover, how much such $\lambda$'s; are these either finite or countably infinite, etc.
locating on a certain line or algebraic curve?
Notice that, if $\lambda \in \Z\backslash \{0\}$ then
$\hgf21{\frac12+i\lambda, \frac12-i\lambda}{1}{t}$ is a polynomial so that the
function in question is trivially a modular function.
\end{prob}

\begin{prob}\label{Ve}
Can we manage directly $V^e(t,\lambda)$ as modular forms context?
\end{prob}

\subsection{Integral expression of $V^e(t,\lambda)$}

Concerning the question as is stated in Problem \ref{Ve}, we first provide the integral expression of
$V^e(t,\lambda)$.

\begin{prop}
\begin{equation*}
V^e(t,\lambda)=\frac12\int_0^\infty\!\int_0^\infty
\frac{2e^{-\frac{u+v}2}\cos(\lambda(u+v))(1-e^{-u-v})}{(1-e^{-u-v})^2-t(e^{-v}-e^{-u})^2}\,dudv.
\end{equation*}
\end{prop}

\begin{proof}
Recalling \eqref{eq:J_k(n)}, \eqref{eq:B_n(u)}, \eqref{eq:definition of w_k(z)}
and \eqref{eq:definition of v_k(t)}, we have
\begin{align*}
V^e(t,\lambda)
&=(1-z)\sum_{k=0}^\infty (i\lambda)^{2k}\sum_{n=0}^\infty \frac{z^n}{2^{2n+1}}
\int_0^\infty
\frac{u^{2k}}{(2k)!}\biggl\{
\frac{e^{nu}}{(\sinh\frac{u}2)^{2n+1}}\int_0^u(1-e^{-2s})^n(1-e^{-2(u-s)})^n\,ds
\biggr\}\,du \\
&=(1-z)\int_0^\infty\cos(\lambda u)\left(\int_0^u
\biggl\{
\sum_{n=0}^\infty
\frac{z^ne^{nu}}{(2\sinh\frac{u}2)^{2n+1}}(1-e^{-2s})^n(1-e^{-2(u-s)})^n
\biggr\}\,ds\right)\,du \\
&=(1-z)\int_0^\infty\cos(\lambda u)\left(\int_0^u
\frac1{2\sinh\frac{u}2}
\biggl\{
1-\frac{ze^u}{4\sinh^2\frac{u}2}(1-e^{-2s})(1-e^{-2(u-s)})
\biggr\}^{-1}
\,ds\right)\,du \\
&=\frac12\int_0^\infty\left(
\int_0^u\frac{2e^{-\frac u2}\cos(\lambda u)(1-e^{-u})}{(1-e^{-u})^2-t(e^{-s}-e^{s-u})^2}
\,ds\right)\,du \\
&=\frac12\int_0^\infty\!\int_0^\infty
\frac{2e^{-\frac{u+v}2}\cos(\lambda(u+v))(1-e^{-u-v})}{(1-e^{-u-v})^2-t(e^{-v}-e^{-u})^2}
\,dudv
\end{align*}
as desired.
\end{proof}

When $\lambda=\frac1i(\frac12-\frac1l)$ for some integer $l\ge2$, assuming $t=T^2$, we have
\begin{equation}
V^e\Bigl(T^2,\frac1i\Bigl(\frac12-\frac1l\Bigr)\Bigr)
=\frac{l^2}2 \int_0^1\!\int_0^1
\frac{1+(xy)^{l-2}}{1-(xy)^l-T(x^l-y^l)}dxdy.
\end{equation}
On the other hands, it follows from \eqref{HGEofMeta} that
\begin{equation}\label{HG_repre}
V^e\Bigl(T^2,\frac1i\Bigl(\frac12-\frac1l\Bigr)\Bigr)
=\frac{\pi^2}{2\sin^2\frac\pi l}\hgf21{\frac1l, 1-\frac1l}{1}{T^2}.
\end{equation}
It follows then
\begin{equation}
\pi^2\hgf21{\frac1l, 1-\frac1l}{1}{T^2}
=l^2\sin^2\frac\pi l
\int_0^1\!\int_0^1
\frac{1+(xy)^{l-2}}{1-(xy)^l-T(x^l-y^l)}dxdy.
\end{equation}

Related to this function/integral, we now recall the result in \cite{RV1999}, which is discussing the relation between
Mahler measures and the special value of $L$-functions for elliptic curves from the modular form point of view.
The (logarithmic) Mahler measure $m(P)$ of a Laurent polynomial $P\in \C[x_1^\pm,\ldots, x_n^\pm]$ is defined as
the following integral over the torus.
\begin{equation*}
m(P)=\int_0^1\cdots\int_0^1\,
\log\abs{P(e^{2\pi i \theta_1}, \cdots, e^{2\pi i \theta_n})} \, d\theta_1\cdots d\theta_n.
\end{equation*}
It is known that, for instance, there is a remarkable identity such as $m(1+x+y)=L'(\chi,-1)$, where $L(\chi,s)$ is the Dirichlet series associated to the quadratic character $\chi$ of conductor 3. Among others, the study in \cite{RV1999} shows the following result, which asserts very explicitly the relation between Mahler measures and the special value of $L$-functions for elliptic curves.

\begin{prop}
For $l=2,3,4$ and $6$, put
\begin{equation}
u_l(\lambda)=\frac1{(2\pi i)^2}\int_{\Torus}\frac1{1-\lambda P_l(x,y)}\frac{dx}x\frac{dy}y,
\end{equation}
where $\Torus=\Set{(z,w)\in\C^2}{\abs z=\abs w=1}$ and
\begin{gather*}
P_2(x,y)=x+\frac1x+y+\frac1y,\quad
P_3(x,y)=\frac{x^2}y+\frac{y^2}x+\frac1{xy},
\\
P_4(x,y)=xy^2+\frac{x}{y^2}+\frac1{x},\quad
P_6(x,y)=\frac{x^2}y-\frac{y}x-\frac1{xy}.
\end{gather*}
Then one finds
\begin{gather*}
u_l(\lambda)=\hgf21{\frac1l,1-\frac1l}{1}{C_l\lambda^l}\qquad(l=2,3,4,6),\\
C_2=2^4,\qquad
C_3=3^3,\qquad
C_4=2^6,\qquad
C_6=2^43^3.
\end{gather*}
The function $u_l(\lambda)$ is related to the Mahler measure of
the polynomial $P_l(x,y)-1/\lambda$ $($which defines an elliptic curve\/$)$ as
\begin{equation*}
m(P_l(x,y)-1/\lambda)=
\Re\Bigl\{-\log\lambda-\int_0^\lambda(u_l(t)-1)\frac{dt}t\Bigr\}.
\eqed
\end{equation*}
\end{prop}

It is worth noting that, via the hypergeometric representation \eqref{HG_repre}, this proposition shows that the following relation between Mahler measures associated with curves and our meta-generation functions of Ap\'ery-like numbers holds.
\begin{cor}
The following holds.
\begin{equation}
V^e\Bigl(C_l\lambda^l,\frac1i\Bigl(\frac12-\frac1l\Bigr)\Bigr)
=\frac{\pi^2}{2\sin^2\frac\pi l}\,u_l(\lambda) \qquad(l=2,3,4,6).
\end{equation}
\end{cor}

\begin{rem}
Since there is an intimate relation between the Mahler measure for elements in group ring of a finite group and the characteristic polynomial of the adjacency matrix of a weighted Cayley graph and characters of the group \cite{DL2009}, it is natural to expect the existence of a certain dynamical system behind the NcHO.
\end{rem}

\section{Automorphic integrals associated with Ap\'ery-like numbers}\label{Modular}

The function $w_2(t)$ becomes a modular form of weight $1$ with respect to the congruent subgroup $\Gamma(2)$
if we take $t$ as a suitable $\Gamma(2)$-modular function.
This is a reflection of the fact that the differential equation for $w_2(t)$
is the Picard-Fuchs equation for an associated family of elliptic curves.
In this section, we recall this story for $w_2(t)$ to
other generating functions $w_k(t)$ of Ap\'ery-like numbers from \cite{KW2012RIMS}
and study the Fourier expansions of certain integrals of modular forms, which are appeared naturally in the story.

\subsection{Automorphic integrals}\label{AI}

We summarize notations used in what follows,
and we briefly recall the notion of automorphic integrals due to Knopp \cite{Knopp1978Duke}. This is a slightly extended notion of automorphic integrals studied in \cite{G1961}.

Let $\congsubgp$ be a Fuchsian group and $m$ be an integer.
Let $\tau\in\uhp$, $\uhp$ being the complex upper half plane, and $q:=e^{2\pi i \tau}$.
Denote by $F(\uhp)$ the linear space of all $\C$-valued functions on the complex upper half plane.
The group $\congsubgp$ acts on $\uhp$ by $\gamma\tau:=\frac{a\tau+b}{c\tau+d}$
for $\gamma=\mat{a & b \\ c & d}\in\congsubgp$ and $\tau\in\uhp$.
The space $F(\uhp)$ becomes a (right) $\congsubgp$-module
by the map $F(\uhp)\times\congsubgp\ni(f,\gamma)\mapsto f\big|_m\gamma\in F(\uhp)$ defined by
\begin{equation}\label{eq:slash}
(f\big|_m\gamma)(\tau)=j(\gamma,\tau)^{-m}f(\gamma\tau).
\end{equation}
Here $j(\gamma,\tau):=c\tau+d$ for $\gamma=\mat{a & b \\ c & d}$ and $\tau\in\uhp$.
We denote by $\Holo(\uhp)$, $\Mero(\uhp)$, $\C(\tau)$ the subspaces of $F(\uhp)$
consisting of holomorphic functions on $\uhp$,
meromorphic functions on $\uhp$
and rational functions on $\uhp$ respectively.
We also set $\C[\tau]_{k}$
to be the space of polynomial functions on $\uhp$ which is of at most degree $k$.
Notice that these spaces are $\congsubgp$-submodules of $F(\uhp)$ under the action $f\big|_{-k}\gamma$.

The standard generators of the modular group $SL_2(\Z)$ are denoted by
\begin{equation*}
T=\mat{1 & 1 \\ 0 & 1},\qquad
S=\mat{0 & -1 \\ 1 & 0}.
\end{equation*}
Define the subgroup $\hecke$ of $SL_2(\Z)$ by $\hecke:=\generators{T^2,S}$.
Notice that $\hecke$ is a subgroup of $\Gamma(2)$, the principal congruence subgroup of level $2$;
\begin{equation*}
\hecke\supset\Gamma(2):=\Set{\gamma \in SL_2(\Z)}{\gamma \equiv I \bmod 2}
=\generators{T^2,ST^{-2}S^{-1}}.
\end{equation*}

If $f(\tau)$ is an automorphic form of even integral weight $m+2$ for $\congsubgp$,
then an $(m+1)$-fold iterated integral $F(\tau)$ of $f(\tau)$ is called an \emph{automorphic integral} of $f(\tau)$.
By the Bol formula
\begin{equation}\label{eq:Bol's formula}
\frac{d^{m+1}}{d\tau^{m+1}}\bigl(j(\gamma,\tau)^mF(\gamma\tau)\bigr)
=j(\gamma,\tau)^{-m-2}F^{(m+1)}(\gamma\tau)
\qquad
(\gamma\in SL_2(\R)),
\end{equation}
we see that $(F\big|_{-m}\gamma)(\tau)-F(\tau)$ is a \emph{polynomial} in $\tau$ of degree at most $m+1$.

In \cite{Knopp1978Duke}, Knopp introduced an extended notion of the automorphic integrals;
a meromorphic function $F$ on the upper half plane $\uhp$ is called an \emph{automorphic integral} of weight $2k$ for $\congsubgp$ with rational \emph{period functions} $\Set{R_F(\gamma)(\tau)\in\C(\tau)}{\gamma\in\congsubgp}$ if
\begin{equation*}
(F\big|_{2k}\gamma)(\tau)=F(\tau)+R_F(\gamma)(\tau)
\end{equation*}
for each $\gamma\in\congsubgp$ and $F$ is meromorphic at each cusp of $\congsubgp$.

\begin{ex}
The Eisenstein series $E_2(\tau)$ of weight $2$ satisfies
\begin{equation*}
E_2(\tau+1)=E_2(\tau),\qquad
\tau^{-2}E_2\Bigl(-\frac1\tau\Bigr)-E_2(\tau)=\frac{12}{2\pi i\tau}.
\end{equation*}
Hence $E_2(\tau)$ is an automorphic integral of weight $2$ with for $SL_2(\Z)$.
\end{ex}

Notice that an automorphic integral obtained by an $(m+1)$-fold iterated integral
of the automorphic form of weight $m+2$ is an automorphic integral of weight $-m$ with \emph{polynomial} period functions.
To emphasize the polynomiality of the period functions, in what follows, we call an automorphic integral with polynomial period functions an automorphic integral.

\subsection{Modular form interpretation of $w_2(t)$}

We first recall the result on the modularity of $w_2(t)$ in \cite{KW2007} briefly.
We recall the following standard functions;
the elliptic theta functions
\begin{equation*}
\theta_2(\tau)=\sum_{n=-\infty}^\infty q^{(n+1/2)^2/2}, \qquad
\theta_3(\tau)=\sum_{n=-\infty}^\infty q^{n^2/2}, \qquad
\theta_4(\tau)=\sum_{n=-\infty}^\infty (-1)^nq^{n^2/2},
\end{equation*}
and normalized Eisenstein series
\begin{equation*}
E_{k}(\tau)=1+\frac2{\zeta(1-k)}\sum_{n=1}^\infty\sigma_{k-1}(n)q^n
\qquad(k=2,4,6,\dots).
\end{equation*}
Put
\begin{equation}\label{eq:def_of_t}
t=t(\tau)=-\frac{\theta_2(\tau)^4}{\theta_4(\tau)^4}
=\frac{\lambda(\tau)^2}{\lambda(\tau)^2-1}= \frac{\eta(\tau)^8\eta(4\tau)^{16}}{\eta(2\tau)^{24}},
\end{equation}
which is a $\Gamma(2)$-modular function such that $t(i\infty)=0$.
Here $\eta(\tau)$ is the Dedekind eta function.
We see that
\begin{equation*}
1-t=\frac{\theta_3(\tau)^4}{\theta_4(\tau)^4},\qquad
\frac{t}{t-1}=\frac{\theta_2(\tau)^4}{\theta_3(\tau)^4},\qquad
\frac{q}{t}\frac{dt}{dq}=\frac12\theta_3(\tau)^4.
\end{equation*}
By the formula (\S22.3 in \cite{WW})
\begin{equation*}
\hgf21{\frac12,\frac12}{1}{\frac{\theta_2(\tau)^4}{\theta_3(\tau)^4}}=\theta_3(\tau)^2,
\end{equation*}
it follows from \eqref{eq:differential-equation-for-w} for $k=2$ that
\begin{equation*}
w_2(t)=\frac{\J20}{1-t}\hgf21{\frac12,\frac12}{1}{\frac{t}{t-1}}
=\J20\frac{\theta_4(\tau)^4}{\theta_3(\tau)^2}
=\J20\frac{\eta(2\tau)^{22}}{\eta(\tau)^{12}\eta(4\tau)^8},
\end{equation*}
which is a $\Gamma(2)$-modular form of weight $1$.

\subsection{Toward modular interpretation of $w_k(t)$}

The fact mentioned above on $w_2(t)$ naturally leads us to a question what the nature of $w_k(t)$ is in general.
In order to answer this question for the special case $w_4(t)$,
we recall the following general fact (Lemma \ref{lem:Yang}),
which is a slight modification of \cite[Lemma 1]{Y2008} and is proved in the same manner.
Let $\Gamma$ be a discrete subgroup of $SL_2(\R)$ commensurable with the modular group.

\begin{lem}\label{lem:Yang}
Let $A(\tau)$ be a modular form of weight $k$ and $t(\tau)$
be a non-constant modular function on $\Gamma$ such that $t(i\infty)=0$.
Let
\begin{equation*}
L:=\vartheta^{k+1}+r_k(t)\vartheta^k+\cdots+r_0(t)
\qquad
\Bigl(\vartheta=t\frac{d}{dt}\Bigr)
\end{equation*}
be the differential operator with rational coefficients $r_j(t)$.
Assume that $LA(t)=0$.
Let $g(t)=g(t(\tau))$ be another modular form.
Then a solution of the inhomogeneous differential equation $LB(t)=g(t)$ is given by
the iterated integration
\begin{equation*}
B(t)=A(t)\underbrace{\int^q\dotsb\int^q}_{k+1}\Big(\frac{q dt/dq}{t}\Big)^{k+1}\frac{g(t)}{A(t)}
\frac{dq}q\dotsb\frac{dq}q.
\eqed
\end{equation*}
\end{lem}

From Theorem \ref{thm:differential-equation-for-w}, it follows that
\begin{equation}\label{eq:diff-eq-for-w}
\Bigl(t(1-t)^2\frac{d^2}{dt^2}+(1-t)(1-3t)\frac{d}{dt}+t-\frac34\Bigr)^kw_{2k+2}(t)=w_2(t)
\end{equation}
for $k\ge1$, which can be also written in terms of the Euler operator $\vartheta$ as
\begin{equation}\label{eq:diff-eq-for-w2}
\begin{split}
L_kw_{2k+2}(t)&=\frac{t^k}{(1-t)^{2k}}w_2(t)\qquad(k\ge1), \\
L_k&=\vartheta^{2k}+r_{2k-1}(t)\vartheta^{2k-1}+\dotsb+r_0(t)
\qquad(r_0(t),\dots,r_{2k-1}(t)\in\C(t)).
\end{split}
\end{equation}

Let us consider the function
\begin{align*}
W_{k}(t)
&:=w_2(t)
\underbrace{\int_0^q\dotsb\int_0^q}_{2k}\Big(\frac{q dt/dq}{t}\Big)^{2k}
\frac{t^k}{(1-t)^{2k}}
\frac{dq}q\dotsb\frac{dq}q \\
&=\Bigl(-\frac14\Bigr)^k\J20\frac{\theta_4(\tau)^4}{\theta_3(\tau)^2}\underbrace{\int_0^q\dotsb\int_0^q}_{2k}
\Bigl(\theta_2(\tau)^{4}\theta_4(\tau)^{4}\Bigr)^{\!k}
\frac{dq}q\dotsb\frac{dq}q.
\end{align*}

Let us look at the case where $k=1$.
If we apply Lemma \ref{lem:Yang} to \eqref{eq:diff-eq-for-w2},
then we see that the integral $W_{1}(t)$ is a solution to \eqref{eq:diff-eq-for-w2},
and hence $w_{4}(t)-W_{1}(t)$ is a solution of the homogeneous equation $L_1f=0$ of degree $2$
which is holomorphic at $t=0$.
This implies that $w_4(t)-W_1(t)$ is a constant multiple of $w_2(t)$.
Thus we have
$w_4(t)=Cw_2(t)+W_{1}(t)$
for a constant $C$, which is determined to be $\pi^2(=J_4(0)/J_2(0))$ by looking at the constant terms.
Namely, we get
\begin{equation}\label{eq:w4-rep}
w_4(t)=\pi^2w_2(t)+W_{1}(t).
\end{equation}

\subsection{Automorphic integrals approach to $W_{k}(t)$}

In what follows, we consider $W_{k}(t)$ for $k\in\N$ in general.
For convenience, let us put
\begin{align}
f(\tau)&=\theta_2(\tau)^4\theta_4(\tau)^4=\frac1{15}\bigl(E_4(\tau/2)-17E_4(\tau)+16E_4(2\tau)\bigr),\\
\iE_k(\tau)&=\underbrace{\int_0^q\dotsb\int_0^q}_{2k}f(\tau)^k\frac{dq}q\dotsb\frac{dq}q,\\
\iG_k(\tau)&=\underbrace{\int_0^q\dotsb\int_0^q}_{4k-1}f(\tau)^k\frac{dq}q\dotsb\frac{dq}q
=\underbrace{\int_0^q\dotsb\int_0^q}_{2k-1} \iE_k(\tau)\frac{dq}q\dotsb\frac{dq}q.
\end{align}
Notice that
\begin{equation*}
W_k(t)=\Bigl(-\frac14\Bigr)^kw_2(t)\iE_k(\tau)
=\frac{2\pi i}{(16\pi^2)^{k}}w_2(t)\frac{d^{2k-1}}{d\tau^{2k-1}}\iG_k(\tau).
\end{equation*}
Clearly, $\iG_k(\tau)$ is a periodic function with period $2$ and $\iG_k(i\infty)=0$.
Since $f(\tau)^k$ is a modular form of weight $4k$ with respect to $\Gamma(2)$ (or $\hecke$),
the function $\iG_k(\tau)$ is an automorphic integral for $f(\tau)^k$ by definition.
Hence, by \eqref{eq:w4-rep}, we have the
\begin{thm}\label{thm:w by iG}
The fourth generating function $w_4(t)$ of Ap\'ery-like numbers is
a linear combination of $w_2(t)$ and the derivative $\iG_1'(\tau)$ of an automorphic integral for $\hecke$ of weight $-2$ as
\begin{equation}\label{eq:w_4 by iG_1}
w_4(t)=\pi^2w_2(t)+W_1(t),\qquad
W_1(t)=\frac{2\pi i}{16\pi^2}w_2(t)\iG_1'(\tau).
\end{equation}
\end{thm}

Note that the Fourier expansion of $\iG_1(\tau)$ is given by
\begin{align}
\iG_1(\tau)&=\int_0^q\int_0^q\int_0^q
\frac1{15}\bigl(E_4(\tau/2)-17E_4(\tau)+16E_4(2\tau)\bigr)
\frac{dq}q\frac{dq}q\frac{dq}q \notag \\
&=16\Biggl(8\sum_{n\ge1}\sigma_{-3}(n)q^{\frac n2}-17\sum_{n\ge1}\sigma_{-3}(n)q^n+2\sum_{n\ge1}\sigma_{-3}(n)q^{2n}\Biggr). \label{G_1Fourier}
\end{align}
In the next section we will give a formula for $\iG_1(\tau)$ and $w_4(t)$,
in which they are expressed in terms of \emph{differential Eisenstein series}
(\eqref{eq:iGt=phi} and Theorem \ref{eq:w4}).

We calculate the period function of $\iG_k(\tau)$, especially to describe $w_4(t)$ concretely via $\iG_1(\tau)$.
The $L$-function corresponding to $f(\tau)^k$ is
\begin{equation}
\Lambda_k(s)=\int_0^\infty t^{s}f(it)^k\frac{dt}t,
\end{equation}
which satisfies the functional equation $\Lambda_k(4k-s)=\Lambda_k(s)$.
By the inversion formula of Mellin's transform, one notices that
\begin{equation*}
f(iy)^k=\frac1{2\pi i}\int_{\Re s=\alpha}y^{-s}\Lambda_k(s)ds\qquad(y>0, \alpha\gg0).
\end{equation*}
Put
\begin{equation*}
\Xi_k(s)=\frac{\Lambda_k(s+2k)}{\prod_{j=1-2k}^{2k-1}(s-j)},
\qquad
\rho_{k,j}=\Res_{s=j}\Xi_k(s)\quad(j=1-2k,\dots,2k-1).
\end{equation*}
The functional equation for $\Lambda_k(s)$ implies the oddness $\Xi_k(-s)=-\Xi_k(s)$,
from which we see that $\rho_{k,-j}=\rho_{k,j}$.
Define $R^k_S(\tau)$ by
\begin{equation*}
R^k_S(\tau)=-(2\pi)^{4k-1}\sum_{j=1-2k}^{2k-1}\rho_{k,j}\Bigl(\frac{\tau}i\Bigr)^{2k-1-j}.
\end{equation*}
Notice that $R^k_S(\tau)$ is a polynomial in $\tau$ of degree $4k-2$.
We have the
\begin{lem}[{\cite[Theorem 4]{KW2012RIMS}}]\label{thm:psi_is_Eichler-modular}
One has
\begin{equation*}
\iG_k(\tau+2)=\iG_k(\tau),\qquad
\tau^{4k-2}\iG_k\Bigl(-\frac1\tau\Bigr)-\iG_k(\tau)=R^k_S(\tau).
\end{equation*}
\end{lem}

Let us consider the particular case where $k=1$.
Explicitly, we have
\begin{gather*}
\Lambda_1(s)=16\pi^{-s}\Gamma(s)\zeta(s)\zeta(s-3)(1-2^{-s})(1-2^{4-s}),\\
\rho_{1,-1}=\rho_{1,1}=\frac{7\zeta(3)}{\pi^3},\quad
\rho_{1,0}=-\frac12,\quad
R^1_S(\tau)=56\zeta(3)(\tau^2-1)+\frac{4\pi^3}i\tau.
\end{gather*}
Lemma \ref{thm:psi_is_Eichler-modular} then reads
\begin{lem}
The function
\begin{equation}\label{eq:Psi1}
\iGt_1(\tau):=\iG_1(\tau)-56\zeta(3)
\end{equation}
satisfies
\begin{equation}\label{eq:transformation rule of iGt_1}
\iGt_1(\tau+2)=\iGt_1(\tau),\qquad
\tau^2\iGt_1\Bigl(-\frac1\tau\Bigr)-\iGt_1(\tau)=\frac{4\pi^3}i\tau.
\eqed
\end{equation}
\end{lem}

\subsection{Experimental calculation to determine the coefficients}
In this subsection, we observe that the generating function $w_{2k}(t)$ of the Ap\'ery numbers for $\zeta_Q(2k)$ is  expressed by a certain linear combination of the multiple integral of the (same) modular forms. Namely,
we try to determine the coefficients $c'_{k,j}$ in the equation
\begin{align}\label{eq:w_{2k+2} by W_k's}
w_{2k+2}(t)
&=\frac{\J{2k+2}0}{\J20}w_2(t)+\sum_{j=1}^{k-1} c'_{k,j}W_j(t)+W_k(t) \\
&=w_2(t)\biggl\{
\frac{\J{2k+2}0}{\J20}+\sum_{j=1}^{k} \Bigl(-\frac14\Bigr)^jc'_{k,j}\iE_j(\tau)
\biggr\}\qquad(c'_{k,k}=1).
\end{align}
(Recall that $W_j(t)=(-\frac14)^jw_2(t)\iE_j(\tau)$.)
Let $\alpha^{(m)}_l$, $\beta^{(m)}_l$ and $\gamma^{(m)}_l$ be the $l$-th Fourier coefficients of $w_m(t)$, $\iE_m(\tau)$ and $t^m$ respectively, that is,
\begin{gather*}
w_m(t)=\sum_{l=0}^\infty \alpha^{(m)}_lq^{\frac l2},\qquad
\iE_m(\tau)
=\sum_{l=0}^\infty \beta^{(m)}_lq^{\frac l2},\qquad
t^m=\sum_{l=0}^\infty \gamma^{(m)}_lq^{\frac l2}.
\end{gather*}
Trivially we have $\gamma^{(0)}_l=\delta_{l,0}$.
We also note that $\gamma^{(m)}_l=\beta^{(m)}_l=0$ if $l<m$ since $t$ and $f(\tau)$ vanish at $i\infty$.
Thus we have
\begin{equation*}
w_m(t)=\sum_{n=0}^\infty \J mnt^n
=\sum_{n=0}^\infty \J mn\sum_{l=0}^\infty \gamma^{(n)}_lq^{\frac l2}
=\sum_{l=0}^\infty \Bigl(\sum_{n=0}^l \J mn \gamma^{(n)}_l\Bigr)q^{\frac l2},
\end{equation*}
or
\begin{equation*}
\alpha^{(m)}_l=\sum_{n=0}^l \J mn \gamma^{(n)}_l.
\end{equation*}

Recall that $W_j(t)=(-\frac14)^jw_2(t)\iE_j(\tau)$. Now \eqref{eq:w_{2k+2} by W_k's} reads
\begin{align*}
\sum_{l=0}^\infty \alpha^{(2k+2)}_l q^{\frac l2}
&=\frac{\J{2k+2}0}{\J20}\sum_{l=0}^\infty \alpha^{(2)}_l q^{\frac l2}
+\sum_{j=1}^{k}\Bigl(-\frac14\Bigr)^jc'_{k,j}
\sum_{l,m\ge0} \alpha^{(2)}_l\beta^{(j)}_m q^{\frac{l+m}2}
\end{align*}
Comparing the coefficients of $q^{l/2}$ of the both sides, we have
\begin{equation*}
\sum_{j=1}^k\Bigl(-\frac14\Bigr)^j
\biggl\{\sum_{m=0}^{l}\alpha^{(2)}_{l-m} \beta^{(j)}_{m}\biggr\}c_{k,j}'
=\alpha^{(2k+2)}_l-\frac{\J{2k+2}0}{\J20}\alpha^{(2)}_l
\end{equation*}
for $l=1,\dots,k$.

\begin{ex}[$l=1$]
We have
\begin{equation*}
-\frac14\J20\gamma^{(0)}_0\beta^{(1)}_{1}c_{k,1}'
=\frac{\J{2k+2}1\J20-\J{2k+2}0\J21}{\J20}\gamma^{(1)}_1
=\J{2k}0\gamma^{(1)}_1
\end{equation*}
since
\begin{gather*}
\J{2k+2}1\J20-\J{2k+2}0\J21
=\Bigl(\J{2k}{0}+\frac34\J{2k+2}{0}\Bigr)\J20-\frac34\J20\J{2k+2}0
=\J{2k}0\J20.
\end{gather*}
We see that
\begin{equation*}
\gamma^{(1)}_1=-16,\qquad
\beta^{(1)}_1=64
\end{equation*}
Thus we have
\begin{equation*}
c_{k,1}'
=\frac{\J{2k}0}{\J20}.
\end{equation*}
For instance, we have
\begin{align*}
w_6(t)
&=w_2(t)\biggl\{
\frac{\J60}{\J20}-\frac14\frac{\J40}{\J20}\iE_1(\tau)+\frac1{16}\iE_2(\tau)
\biggr\} \\
&=w_2(t)\biggl\{
\frac{\J60}{\J20}+\frac{\J40}{\J20}\frac{2\pi i}{16\pi^2}\frac{d}{d\tau}\iG_1(\tau)+\frac{2\pi i}{256\pi^4}\frac{d^3}{d\tau^3}\iG_2(\tau)
\biggr\}.
\end{align*}
\end{ex}

\begin{ex}[$l=2$]
We see that
\begin{equation*}
\gamma^{(1)}_2=-128,\quad
\gamma^{(2)}_2=256,\qquad
\beta^{(1)}_2=-128,\quad
\beta^{(2)}_2=256.
\end{equation*}
For $k\ge2$, we have
\begin{gather*}
-\frac14
\biggl\{ \J20 \beta^{(1)}_{2}
+\J20\gamma^{(0)}_1\beta^{(1)}_{1}+\J21\gamma^{(1)}_1\beta^{(1)}_{1}
\biggr\}c_{k,1}'
+\frac1{16}\J20 \beta^{(2)}_{2}c_{k,2}' \\
=\frac{\J{2k+2}1\J20-\J{2k+2}0\J21}{\J20}\gamma^{(1)}_2
+\frac{\J{2k+2}2\J20-\J{2k+2}0\J22}{\J20}\gamma^{(2)}_2,
\end{gather*}
which is reduced to
\begin{gather*}
c_{k,2}'=-8-14c_{k,1}'+16\frac{\J{2k+2}2\J20-\J{2k+2}0\J22}{\J20^2}
=-8+8\frac{\J{2k}0}{\J20}+4\frac{\J{2k-2}0}{\J20}.
\end{gather*}
\end{ex}

\begin{ex}
We have
\begin{gather*}
c'_{k,1}=\frac{\J{2k}0}{\J20},\quad
c'_{k,2}=\frac{4\J{2k-2}0}{\J20},\quad
c'_{k,3}=\frac{117\J{2k-2}0+162\J{2k-4}0}{8\J20},\\
c'_{k,4}=\frac{-695\J{2k-2}0+2794\J{2k-4}0+1024\J{2k-6}0}{9\J20}.
\end{gather*}
\end{ex}

A systematic study of the generating functions $w_{2n}$ for higher special values is desirable.

\section{Differential Eisenstein series}\label{EF}

We have shown in Theorem \ref{thm:w by iG} that the function $w_{4}(t)$ is a linear combination of $w_2(t)$ and the derivative $\iG_{1}'$ of an automorphic integral.
To understand the integrals $\iG_j(\tau)$ more concretely, we introduce a family of functions called differential Eisenstein series which play a role analogous to the ordinary Eisenstein series.

\subsection{Periodic automorphic integrals}

Let $\congsubgp$ be a congruence subgroup of level $N$ and $m$ be an integer.
We take a $\congsubgp$-submodule $\X$ of $F(\uhp)$.
We focus our attention on automorphic integrals of special types defined as follows.

\begin{dfn}[Periodic automorphic integrals]
Let $\chi$ be a (multiplicative) character of $\congsubgp$ such that $\chi(T^N)=1$.
A holomorphic function $f\in\Holo(\uhp)$ is called a \emph{periodic automorphic integral} for $\congsubgp$ of weight $m$ with character $\chi$ and \emph{period functions} $\{\per_{f,\chi}(\gamma)\}_{\gamma\in\congsubgp}\subset\X$ if
\begin{gather}
f(\tau+N)=f(\tau),\label{eq:EF-1} \\
(f\big|_m\gamma)(\tau)-\chi(\gamma)f(\tau)=\per_{f,\chi}(\gamma)(\tau)
\quad(\forall\gamma\in \congsubgp),\label{eq:EF-2} \\
\forall\gamma\in SL_2(\Z),\ \exists\{a_n\}_{n\in\Z} \suchthat
(f\big|_m\gamma)(\tau)-\sum_{n\in\Z}a_nq^{\frac nN}\in\X,\quad a_n=0\quad(n\ll 0).
\label{eq:EF-3}
\end{gather}
We denote by $M^\chi_m(\congsubgp,\X)$ the set consisting of such periodic automorphic integral.
When $\chi$ is the trivial character, we omit the symbol $\chi$ and simply write $M_m(\congsubgp,\X)$.
We call $f$ an \emph{Eichler cusp forms} if it is a periodic automorphic integral such that the Fourier expansion part of $f\big|_m\gamma$ in \eqref{eq:EF-3} has no constant term for every $\gamma\in SL_2(\Z)$.
The space of automorphic cusp forms is denoted by $C^\chi_m(\congsubgp,\X)$.
\end{dfn}

When $m>0$, $M_m(\congsubgp):=M_m(\congsubgp,\{0\})$ and $C_m(\congsubgp):=C_m(\congsubgp,\{0\})$ are nothing but the spaces of classical modular forms and cusp forms of weight $m$ respectively.
Indeed, $f\in M_m(\congsubgp)$ is holomorphic at every cusp of $\congsubgp$ in this case.

\begin{rem}
If $1\in\X$, that is, $\X$ contains constant functions, then any constant shift $f(\tau)+c$ ($c\in\C$) of $f\in M_m(\congsubgp,\X)$ also belongs to $M_{m}(\congsubgp,\X)$.
In this case, it is natural to study the quotient space $M_{m}(\congsubgp,\X)/\text{(constants)}$.
\end{rem}

\begin{ex}\label{V_{2k, m}}
We give a non-trivial example of $\congsubgp$-submodule $\X$ of $F(\uhp)$ as follows.
Let $V_{2k, m}$ be a subspace of $F(\uhp)$ generated by $\tau^j\, (j=0,1,\ldots,2k)$ and $(\tau-\alpha)^{-j}\, (\alpha \in \C \backslash \Set{\gamma\infty}{\gamma\in \Gamma(2)}$, $j=1,2, \ldots, m)$. Notice that $0\not\in \Set{\gamma\infty}{\gamma\in \Gamma(2)}$.
Then the space $V_{2k, m}$ is $\Gamma(2)$-stable subspace of $\C(\tau)$ under the action $(f\big|_{-2k}\gamma)(\tau)=j(\gamma,\tau)^{2k}f(\gamma\tau)\, (\gamma \in \Gamma(2))$. In fact, for $\gamma=\mat{a & b \\ c & d}$, if we put
$f(\tau)=(\tau-\alpha)^{-j}\, (1\leq j\leq m)$, we observe
\begin{align*}
(f\big|_{-2k}\gamma)(\tau)& = (c\tau+d)^{2k} \Big(\frac{a\tau+b}{c\tau+d}-\alpha \Big)^{-j}
= \frac{(c\tau+d)^{2k+j} }{((a-c\alpha)\tau+(b-d\alpha))^j}\\
&= \text{a polynomial of degree $2k$} +
\frac{\text{a polynomial of degree $j-1$}}{(\tau-\gamma\alpha)^j}.
\end{align*}
This clearly shows that $(f\big|_{-2k}\gamma)(\tau)\in V_{2k, m}$.
\end{ex}

The period functions $\{\per_{f,\chi}(\gamma)\}$ for $f\in M_m(\congsubgp,\X)$ obeys the relation
\begin{align}
R_{f,\chi}(T^N)&=0, \label{eq:periodicity}\\
R_{f,\chi}(\gamma_1\gamma_2)&=\chi(\gamma_1)R_{f,\chi}(\gamma_2)+R_{f,\chi}(\gamma_1)\big|_m\gamma_2\qquad(\gamma_1,\gamma_2\in \congsubgp).
\label{eq:1-cocycle-from-RMF}
\end{align}
The latter identity is readily checked as follows.
\begin{equation}\label{eq:proof-of-1-cocycle-condition}
\begin{split}
R_{f,\chi}(\gamma_1\gamma_2)
&=f\big|_m\gamma_1\gamma_2-\chi(\gamma_1\gamma_2)f \\
&=\chi(\gamma_1)\bigl(f\big|_m\gamma_2-\chi(\gamma_2)f\bigr)
+\bigl(f\big|_m\gamma_1-\chi(\gamma_1)f\bigr)\big|_m\gamma_2 \\
&=\chi(\gamma_1)R_{f,\chi}(\gamma_2)+R_{f,\chi}(\gamma_1)\big|_m\gamma_2.
\end{split}
\end{equation}
Hence, by \eqref{eq:1-cocycle-from-RMF}, the condition \eqref{eq:EF-2} can be replaced by the one only for generators of $\congsubgp$.

For convenience, we give the definitions of the space of negative weight holomorphic automorphic integrals (with characters) in terms of the generators for the specific groups $\hecke$ and $\Gamma(2)$.
\begin{dfn}[Periodic automorphic integrals for $\hecke$ and $\Gamma(2)$]
\begin{align*}
M_{m}(\hecke,\X)&:=
\Set{f\in\Holo(\uhp)}
{\begin{minipage}{11em}
$f(\tau+2)=f(\tau)$, \\[.3em]
$\tau^{-m}f\Bigl(-\dfrac1\tau\Bigr)-f(\tau)\in\X$, \\[.3em]
$f$ is holomorphic at $i\infty$
\end{minipage}},\\
M_{m}(\Gamma(2),\X)&:=
\Set{f\in\Holo(\uhp)}
{\begin{minipage}{15.3em}
$f(\tau+2)=f(\tau)$, \\[.3em]
$(2\tau+1)^{-m}f\Bigl(\dfrac\tau{2\tau+1}\Bigr)-f(\tau)\in\X$, \\[.3em]
$f$ satisfies \eqref{eq:EF-3}
\end{minipage}}.
\end{align*}
\end{dfn}

\begin{rem}
If $f\in\Holo(\uhp)$ is holomorphic at $i\infty$ and satisfies the conditions
\begin{equation*}
f(\tau+2)=f(\tau),\qquad
\tau^{-m}f\Bigl(-\dfrac1\tau\Bigr)-f(\tau)=\per_{f,\chi}(S)(\tau)\in\X,
\end{equation*}
then we see that
\begin{equation*}
(f\big|_mT)(\tau)=\sum_{n\ge0}(-1)^na_nq^{\frac n2}
\end{equation*}
when the Fourier expansion of $f$ is given by $f(\tau)=\sum_{n\ge0}a_nq^{\frac n2}$.
Namely, $f$ satisfies the condition \eqref{eq:EF-3} in the definition of periodic automorphic integrals for $\hecke$.
\end{rem}

\begin{ex}
By Lemma \ref{thm:psi_is_Eichler-modular}, we have
$\iG_k(\tau)\in M_{2-4k}(\hecke,\C[\tau])$
for each positive integer $k$.
\end{ex}

\begin{rem}
When $f(\tau)\in M_{m}(\hecke,\X)$ with period functions $\{\per_f(\gamma)\}$, by virtue of \eqref{eq:1-cocycle-from-RMF},
we have $f(\tau)\in M_{m}(\Gamma(2),\X)$.
Indeed, we have
\begin{equation*}
\per_{f}(ST^{-2}S^{-1})
=\per_{f}(S)\big|_mT^{-2}S^{-1}+\per_{f}(T^{-2})\big|_mS^{-1}+\per_{f}(S^{-1})\in\X.
\end{equation*}
\end{rem}

\subsection{Differential Eisenstein series}

We always assume that $-\pi\le\arg z<\pi$ for $z\in\C$ to determine the branch of complex powers.
Define
\begin{align*}
G(s,x,\tau)&:=\psum_{m,n\in\Z}(m\tau+n+x)^{-s},\\
G(s,\tau)&:=G(s,0,\tau),\\
\ges Nab(s,\tau)
&:=\psum_{\substack{m,n\in\Z \\ m\equiv a\,(\mathrm{mod}\,N) \\ n\equiv b\,(\mathrm{mod}\,N)}}(m\tau+n)^{-s}
\qquad(a,b\in\{0,1,\dots,N-1\})
\end{align*}
for $s\in\C$ such that $\Re(s)>2$.
Here $\psum_{m,n\in\Z}$ means the sum over all pairs $(m,n)$ of integers such that the summand is defined.
We sometimes refer to these series as \emph{generalized Eisenstein series} (e.g.\ \cite{B1975Cr}).
Remark that
\begin{equation*}
\ges Nab(s,\tau)=N^{-s}G\Bigl(s,\frac{a\tau+b}N,\tau\Bigr),
\end{equation*}
in particular that $\ges N00(s,\tau)=N^{-s}G(s,\tau)$.

It is known that $G(s,x,\tau)$ is analytically continued to the whole $s$-plane,
and $G(s,x,\tau)$ can be written in the form
\begin{equation*}
G(s,x,\tau)=\sum_{n>-x}\frac1{(n+x)^s}+\frac1{\Gamma(s)}A(s,x,\tau),
\end{equation*}
when $x\in\R$, where $A(s,x,\tau)$ is holomorphic in $s$ and $\tau$.
In particular, we see that
\begin{equation*}
G(-2k,\tau)=\ges211(-2k,\tau)=0
\end{equation*}
for any positive integer $k$ (see \cite[Theorem 1]{L1972TAMS}; see also \cite{B1975Cr}).
We now introduce the notion of \emph{differential Eisenstein series}.
\begin{dfn}[Differential Eisenstein series\footnote{
We have used the notation $dE_{m}(\tau)$ in \cite{KW2012RIMS} in place of $\dG{m}(\tau)$.
In this paper, however, we use the notation $\dE{m}(\tau)$
for representing the \emph{normalized} differential Eisenstein series in \S\ref{Hecke},
which follows the standard use of Eisenstein series in the classical theory of modular forms.}%
]\label{dfn:DES}
For $m\in\Z$, define
\begin{align*}
\dG{m}(\tau)&:=\frac{\der}{\der s}G(s,\tau)\bigg|_{s=m}, \\
\dG[(N;a,b)]{m}(\tau)&:=\frac{\der}{\der s}\ges Nab(s,\tau)\bigg|_{s=m}\qquad(a,b\in\{0,1,\dots,N-1\}).
\end{align*}
\end{dfn}

It is immediate to see that $\dG{m}(\tau+1)=\dG{m}(\tau)$
and $\dG[(N;a,b)]{m}(\tau+N)=\dG[(N;a,b)]{m}(\tau)$.
In the case where $N=2$,
it is convenient to introduce an abbreviation $\dG[a,b]m(\tau)$ for $\dG[(2;a,b)]m(\tau)$,
which will appear frequently below.

For later use, we recall the definitions and several results on the double zeta functions and
double Bernoulli numbers \cite{Barnes1904}.
Let $\omegaul=(\omega_1,\omega_2)$ be a pair of complex parameters.
\emph{Barnes' double zeta function} is defined by
\begin{equation*}
\zeta_2(s,z\,|\,\omegaul):=\sum_{m,n\ge0}(m\omega_1+n\omega_2+z)^{-s}, \quad (\Re s >2)
\end{equation*}
and the \emph{double Bernoulli polynomials} $B_{2,k}(z\,|\,\omegaul)$ are defined by the generating function
\begin{equation*}
\frac{t^2e^{zt}}{(e^{\omega_1t}-1)(e^{\omega_2t}-1)}=\sum_{k=0}^\infty B_{2,k}(z\,|\,\omegaul)\frac{t^k}{k!}.
\end{equation*}

It is well known that the Barnes double zeta function is extended meromorphically to the whole complex plane and the special values at the non-positive integer points are given by (see, e.g.\ \cite{Barnes1904})
\begin{lem}\label{lem:special values of double  zeta}
For each $m\in \N$, one has
\begin{equation*}
\zeta_2(1-m,z\,|\,\omegaul)=\frac{B_{2,m+1}(z\,|\,\omegaul)}{m(m+1)}.
\end{equation*}
\end{lem}

\begin{ex}\label{SV of Barnes}
\begin{align*}
\zeta_2\Big({-2k},\frac{\tau-1}2\,\Big|\,(-1,\tau)\Bigr)
&=\frac{B_{2,2k+2}(\frac{\tau-1}2\,|\,(-1,\tau))}{(2k+1)(2k+2)} \in \frac1\tau\C[\tau],\\
\zeta_2(-2k,\tau\,|\,(-1,\tau))&=\frac{B_{2,2k+2}(\tau\,|\,(-1,\tau))}{(2k+1)(2k+2)}
\in \frac1\tau\C[\tau].
\end{align*}
\end{ex}

\subsection{$\dG{-2k}$ is an automorphic integral}

We notice the following elementary fact.
\begin{lem}\label{lem:arg-exceeds-pi}
If $\tau\in\uhp$ and $(a,b)\in\R^2-\{(0,0)\}$, then
\begin{equation*}
\arg\Bigl(-\frac1\tau\Bigr)+\arg(a\tau+b)\ge\pi
\iff
a>0,\, b\le0.
\eqed
\end{equation*}
\end{lem}

\begin{lem}\label{lem:-1/tau in dE_{-2k}}
For each $k\in\N$, one has
\begin{equation*}
\dG{-2k}\Bigl(-\frac1\tau\Bigr)
=\Bigl(-\frac1\tau\Bigr)^{2k}
\Bigl\{
\dG{-2k}(\tau)-4k\pi i\zeta_2\bigl(-2k,\tau\,|\,(-1,\tau)\bigr)
\Bigr\}.
\end{equation*}
\end{lem}

\begin{proof}
It follows from Lemma \ref{lem:arg-exceeds-pi} that
\begin{align*}
G\Bigl(s,-\frac1\tau\Bigr)
&=\psum_{m,n\in\Z}\Bigl(-m\frac1\tau+n\Bigr)^{-s}
=\psum_{m,n\in\Z}\Bigl(\Bigl(-\frac1\tau\Bigr)(m\tau+n)\Bigr)^{-s}\\
&=\Bigl(-\frac1\tau\Bigr)^{-s}\Bigl\{
\psum_{m,n\in\Z}(m\tau+n)^{-s}+(e^{2\pi is}-1)\sum_{\substack{m>0,\\ n\le0}}(m\tau+n)^{-s}
\Bigr\} \\
&=\Bigl(-\frac1\tau\Bigr)^{-s}\Bigl\{
G(s,\tau)+(e^{2\pi is}-1)\zeta_2\bigl(s,\tau\,|\,(-1,\tau)\bigr)
\Bigr\}.
\end{align*}
This yields that
\begin{align*}
\frac{\partial}{\partial s}G\Bigl(s,-\frac1\tau\Bigr)\bigg|_{s=-2k}
&=\frac{\partial}{\partial s}\Bigl(-\frac1\tau\Bigr)^{-s}\bigg|_{s=-2k}
\Bigl\{
G(-2k,\tau)+(e^{-4k\pi}-1)\zeta_2\bigl(-2k,\tau\,|\,(-1,\tau)\bigr)
\Bigr\} \\
&\quad{}+\Bigl(-\frac1\tau\Bigr)^{2k}
\frac{\partial}{\partial s}
\Bigl\{
G(s,\tau)+(e^{2\pi is}-1)\zeta_2\bigl(s,\tau\,|\,(-1,\tau)\bigr)
\Bigr\}\bigg|_{s=-2k}\\
&=\Bigl(-\frac1\tau\Bigr)^{2k}
\Bigl\{
\frac{\partial}{\partial s}G(s,\tau)\bigg|_{s=-2k}
-4k\pi i\zeta_2\bigl(-2k,\tau\,|\,(-1,\tau)\bigr)
\Bigr\}.
\end{align*}
Thus we have
\begin{equation*}
\dG{-2k}\Bigl(-\frac1\tau\Bigr)
=\Bigl(-\frac1\tau\Bigr)^{2k}
\Bigl\{
\dG{-2k}(\tau)-4k\pi i\zeta_2\bigl(-2k,\tau\,|\,(-1,\tau)\bigr)
\Bigr\}.
\qedhere
\end{equation*}
\end{proof}

By a similar calculation, we also have the
\begin{lem}\label{lem:-1/tau in dE^{1,1}_{-2k}}
For each $k\in\N$, one has
\begin{equation*}
\dG[1,1]{-2k}\Bigl(-\frac1\tau\Bigr)
=\tau^{-2k}\Bigl(\dG[1,1]{-2k}(\tau)-4k\pi i\zeta_2(-2k,{\tau-1}\,\big|\,(-2,2\tau))\Bigr).
\end{equation*}
\end{lem}

By the lemmas above, we obtain the
\begin{cor}
One has
\begin{equation*}
\dG{-2k}(\tau)\in M_{-2k}(SL_2(\Z),\C(\tau)),
\quad
\dG[0,0]{-2k}(\tau), \dG[1,1]{-2k}(\tau)\in M_{-2k}(\hecke,\C(\tau))
\end{equation*}
for each $k\in\N$.
\end{cor}

\begin{rem}
We observe that
$\dG[0,0]{-2k}(\tau), \dG[1,1]{-2}(\tau)\in M_{-2}(\Gamma(2),V_{2,1})$, $V_{2,1}$ being the space defined in Example \ref{V_{2k, m}}.
\end{rem}

\begin{rem}
A recent calculation due to Shibukawa \cite{Shib2013} on the same analysis of the lemmas above shows that $\dG{-2k}(\tau)\in M_{-2k}(SL_2(\Z),\Mero(\uhp))$ but $\notin M_{-2k}(SL_2(\Z),\C(\tau))$ for $k>0$.
\end{rem}

\begin{rem}
Although we have given the proof of lemmas above directly, we may extend these relations to the general case by a similar analysis in \cite{B1975Cr}.
\end{rem}

\begin{rem}
The function $\dG[1,1]m(\tau)$ can be written as
\begin{equation*}
\dG[1,1]m(\tau)=(1+2^{-m})\dG m(\tau)-2^{-m}\dG m(\tau/2)-\dG m(2\tau).
\end{equation*}
\end{rem}

\subsection{An expression of $w_4(t)$ in terms of differential Eisenstein series}

By Lemmas \ref{lem:-1/tau in dE_{-2k}}, \ref{lem:-1/tau in dE^{1,1}_{-2k}}
and \ref{lem:special values of double zeta}, we have
\begin{align*}
\dG[1,1]{-2}\Bigl(-\frac1\tau\Bigr)
&=\tau^{-2}\Bigl(\dG[1,1]{-2}(\tau)-\frac{\pi i}3{B_{2,4}({\tau-1}\,|\,(-2,2\tau))}\Bigr), \\
\dG{-2}\Bigl(-\frac1\tau\Bigr)
&=\tau^{-2}\Bigl(\dG{-2}(\tau)-\frac{\pi i}3{B_{2,4}(\tau\,|\,(-1,\tau))}\Bigr).
\end{align*}
A straightforward calculation using these formulas shows that
\begin{equation*}
7B_{2,4}\bigl(\tau\,\big|\,(-1,\tau)\bigr)+2B_{2,4}\bigl({\tau-1}\,\big|\,(-2,2\tau)\bigr)=-\frac32\tau.
\end{equation*}
Therefore, if we put
\begin{equation}
\phi_1(\tau):=-8\pi^2\Bigl\{7\dG{-2}(\tau)+2\dG[1,1]{-2}(\tau)\Bigr\},
\end{equation}
then we have
\begin{equation*}
\phi_1(\tau+2)=\phi_1(\tau),\qquad
\tau^2\phi_1\Bigl(-\frac1\tau\Bigr)-\phi_1(\tau)=4\pi^3\frac\tau i.
\end{equation*}
These relations are exactly the same with
the ones \eqref{eq:transformation rule of iGt_1} for $\iGt_1(\tau)=\iG_1(\tau)-56\zeta(3)$.
Therefore, the difference
$\iG_1(\tau)-56\zeta(3)-\phi_1(\tau)$
is a classical holomorphic modular forms of weight $-2$ for $\Gamma(2)$.
Since $M_{-2}(\Gamma(2))=\{0\}$, we have
\begin{equation}\label{eq:iGt=phi}
\iG_1(\tau)=\phi_1(\tau)+56\zeta(3).
\end{equation}

Putting this expression into \eqref{eq:w_4 by iG_1}, we obtain the following
\begin{thm}
The generating function $w_4(t)$ of Ap\'ery-like numbers $\J4n$ is given by
\begin{equation*}\label{eq:w4}
w_4(t)=
\frac{\pi^4}2\frac{\theta_4(\tau)^4}{\theta_3(\tau)^2}
\biggl[1+\frac1{\pi i}\frac{d}{d\tau}\Bigl\{
7\dG{-2}(\tau)+2\dG[1,1]{-2}(\tau)
\Bigr\}\biggr],
\end{equation*}
where $t=t(\tau)= -\theta_2(\tau)^4\theta_4(\tau)^{-4}$.
\qed
\end{thm}

\subsection{Fourier expansion of $\dG{-2k}(\tau)$}

We now compute the Fourier expansion of the differential Eisenstein series $\dG{-2k}(\tau)$ using the result in \cite{Shib2013}.
Similarly to the classical Eisenstein series, we will find that the Fourier expansion of $\dG{-2k}(\tau)$ is given by the Lambert series.
In particular, we notice that $\dG{-2k}(\tau)$ is not a cusp form.

We first recall the result for the bilateral zeta function in \cite{Shib2013}.
For $(\omega_1, \omega_2)$, the bilateral zeta function $\xi_2(s, z \,|\, \omega_1, \omega_2)$ is defined by
\begin{equation}
\xi_2(s, z \,|\, \omega_1, \omega_2)
:=\zeta_2(s,z+\omega_1 \,|\,\omega_1, \omega_2)+\zeta_2(s, z \,|\, {-\omega_1}, \omega_2).
\end{equation}
We take $\omega_1$ as $0<\arg(\omega_1)\leq\pi$ in the subsequent discussion.
Then the following Fourier expansion of $\xi_2(s, z \,|\, {-1}, \omega)$ is known (Theorem 4.7 and Corollary 4.8 in \cite{Shib2013}).
\begin{prop}
Suppose $z, \omega \in \uhp$. Then we have
\begin{equation}\label{Fourier_xi}
\xi_2(s, z \,|\, {-1}, \omega) = \frac{e^{-(\pi/2)is}(2\pi)^s}{\Gamma(s)}
\sum_{n=1}^\infty \frac{n^{s-1}e^{2\pi i nz}}{1-e^{2\pi i n\omega}}.
\end{equation}
Moreover, one notices that the bilateral zeta function $\xi_2(s, z \,|\, \omega_1, \omega)$ is an entire function in $s \in \C$ and for $m\in \N$
\begin{equation}\label{SV_xi}
\xi_2(1-m, z \,|\, \omega_1, \omega) = 0.
\end{equation}
\qed
\end{prop}

Using this proposition, we prove the following
\begin{thm} \label{Lambert series expansion of DEisen}
The Fourier expansion of the differential Eisenstein series $\dG{-2k}(\tau)\, (k\in \N)$ is expressed by the Lambert series as
\begin{align*}
\dG{-2k}(\tau)
&= \frac{(-1)^k(2k)!}{(2\pi)^{2k}}
\Big\{2\sum_{n=1}^\infty \frac{n^{-2k-1}}{1-e^{2\pi i n\tau}} - \zeta(2k+1)\Big\} \\
&= \frac{(-1)^k(2k)!}{(2\pi)^{2k}}
\Big\{\zeta(2k+1)+ 2\sum_{n=1}^\infty \sigma_{-2k-1}(n) e^{2\pi i n\tau} \Big\}.
\end{align*}
In particular, the constant term is given by the multiple of $\zeta(2k+1)$ as
\begin{equation*}
\dG{-2k}(i\infty)=\frac{(-1)^k(2k)!}{(2\pi)^{2k}}\zeta(2k+1) \ne0.
\end{equation*}
\end{thm}
\begin{proof}
We observe that
\begin{align*}
G(s,\tau)&= \psum_{m_0, \,m_1\in \Z} (m_0+m_1\tau)^{-s}\\
&= \psum_{m_0\in \Z, \,m_1\in \Z_{\geq0}} (m_0+m_1\tau)^{-s}
+\psum_{m_0\in \Z, \,m_1\in \Z_{\geq0}} (m_0-m_1\tau)^{-s} -\sum_{m_0\not=0}m_0^{-s}\\
&= (1+ e^{\pi i s})\Bigl\{ \psum_{m_0\in \Z, \,m_1\in \Z_{\geq0}}(m_0+m_1\tau)^{-s} -\zeta(s)\Bigr\}.
\end{align*}
By \eqref{SV_xi}, since
\begin{equation*}
0=\xi_2(1-2k, z\,|\, {-1}, \tau)
=\psum_{m_0\in \Z, \,m_1\in \Z_{\geq0}} (z+m_0+m_1\tau)^{-s}\Big|_{s=-2k} + z^{2k},
\end{equation*}
for $k\in \N$, we have
\begin{equation*}
\psum_{m_0\in \Z, \,m_1\in \Z_{\geq0}} (m_0+m_1\tau)^{-s}\Big|_{s=-2k} =0.
\end{equation*}
It follows that
\begin{equation}\label{Differential of G at s=-2k}
\frac{\partial}{\partial s}G(s,\tau)\Big|_{s=-2k}
=2 \frac{\partial}{\partial s}\Big\{ \psum_{m_0\in \Z, \,m_1\in \Z_{\geq0}}(m_0+m_1\tau)^{-s}\Big\} \Big|_{s=-2k}
- 2\zeta'(-2k).
\end{equation}
On the other hand, we observe
\begin{align*}
\xi_2(s, z \,|\,-1, -\tau) &= \zeta_2(s, z-1 \,|\,-1, -\tau) + \zeta_2(s, z \,|\,1, -\tau)\\
&= \sum_{m_0, m_1=0}^\infty (z-1-m_0+m_1\tau)^{-s}
+\sum_{m_0, m_1=0}^\infty (z+m_0+m_1\tau)^{-s}\\
&= \sum_{m_0\in \Z}\sum_{m_1=0}^\infty (z+m_0+m_1\tau)^{-s}
= \psum_{m_0\in \Z, \,m_1\in \Z_{\geq0}} (z+m_0+m_1\tau)^{-s} +z^{-s}.
\end{align*}

By the Fourier expansion \eqref{Fourier_xi}, using the fact
$\frac{d}{ds}\frac1{\Gamma(s)}\Big|_{s=-2k}=(2k)!$, we have
\begin{equation}
\frac{\partial}{\partial s}\xi_2(s, z \,|\, -1, \tau)\Big|_{s=-2k}
= \frac{(-1)^k(2k)!}{(2\pi)^{2k}}
\sum_{n=1}^\infty \frac{n^{-2k-1}e^{2\pi i nz}}{1-e^{2\pi i n\tau}}.
\end{equation}
Therefore we have
\begin{align*}
\frac{\partial}{\partial s}\Big\{ \psum_{m_0\in \Z, \,m_1\in \Z_{\geq0}}(m_0+m_1\tau)^{-s}\Big\} \Big|_{s=-2k}
&=\Big\{\frac{\partial}{\partial s}\xi_2(s, z \,|\, {-1}, \tau)\Big|_{s=-2k}
+ (\log z) z^{2k} \Big\}\Big|_{z=0}\\
&=\frac{(-1)^k(2k)!}{(2\pi)^{2k}}
\sum_{n=1}^\infty \frac{n^{-2k-1}}{1-e^{2\pi i n\tau}}.
\end{align*}
It follows from \eqref{Differential of G at s=-2k} that
\begin{equation}
\dG{-2k}(\tau)= \frac{\partial}{\partial s}G(s,\tau)\Big|_{s=-2k}
= \frac{2(-1)^k(2k)!}{(2\pi)^{2k}}
\sum_{n=1}^\infty \frac{n^{-2k-1}}{1-e^{2\pi i n\tau}} -2\zeta'(-2k).
\end{equation}
Using the functional equation $\frac{2^{s-1}\pi^s}{\Gamma(s)}\zeta(1-s)=\cos\frac{\pi s}{2}\zeta(s)$ and
$\frac{d}{ds}\frac1{\Gamma(s)}\Big|_{s=-2k}=(2k)!$, we have
$\zeta'(-2k)= \frac{(-1)^k(2k)!}{2(2\pi)^{2k}}\zeta(2k+1)$.
Hence we complete the proof of the theorem.
\end{proof}

\begin{rem}
We note that the function $\phi_1(\tau)$ is expressible only by $\dG{-2}$ as
\begin{equation*}
\phi_1(\tau)=8\pi^2\Bigl(8\dG{-2}(\tau/2)-17\dG{-2}(\tau)+2\dG{-2}(2\tau)\Bigr).
\end{equation*}
By Theorem \ref{Lambert series expansion of DEisen}, we have
\begin{equation*}
\dG{-2}(\tau)=-\frac1{2\pi^2}
\Big\{\zeta(3)+2\sum_{n=1}^\infty \sigma_{-3}(n)q^n \Big\}.
\end{equation*}
Comparing with \eqref{G_1Fourier}, we obtain \eqref{eq:iGt=phi} again.
\end{rem}

\begin{rem}
Like Ramanujan did, we may evaluate values of the Lambert series at $\tau=i$ if
$k\in\N$ is odd as follows.
\begin{equation}
\sum_{n=1}^\infty \frac1{n^{2k+1}(1-e^{-2\pi n})}
= \frac{ki (2\pi)^{2k+1}}{2(-1)^k(2k+2)!} B_{2, 2k+2}(i\,|\, (-1,i)) + \frac12\zeta(2k+1).
\end{equation}
In fact, by Lemma \ref{lem:-1/tau in dE_{-2k}} together with Example \ref{SV of Barnes}, we have
\begin{equation*}
\dG{-2k}(i) = 2k\pi i \frac{B_{2, 2k+2}(i\,|\, (-1,i))}{(2k+1)(2k+2)},
\end{equation*}
whenever $k$ is odd. Hence the formula follows immediately from Theorem \ref{Lambert series expansion of DEisen}.
\end{rem}

\subsection{Hecke operators acting on automorphic forms of negative weight}\label{Hecke}

We give a short remark on the Hecke operators acting on the negative weight automorphic forms.

Let $n\in \N$ and set $\M_n:= \Set{g \in \Mat_2(\Z)}{\det g=n}$.
Since the group $SL_2(\Z)$ acts on $\M_n$ on the left, one may decompose $\M_n$ into orbits.
We now consider the automorphic forms of weight $-k$ ($k\in \N$).
For $f\in M_{-k}(SL_2(\Z),\X)$, we set
\begin{equation}\label{Hecke operator}
(T(n)f)(\tau)= n^{-k/2-1}\sum_{\mu\in SL_2(\Z)\backslash \M_n} (f|_{-k} \mu)(\tau).
\end{equation}
Here we notice that the sum $\sum_{\mu\in SL_2(\Z)\backslash \M_n} f|_{-k} \mu$ depends on the choice of a system of representatives $\{\mu\}$ for the orbits $SL_2(\Z)\backslash \M_n$.
Actually, if we take another representatives $\{\gamma\mu\}$ ($\gamma \in SL_2(\Z)$) we observe that
\begin{align*}
\sum_{\mu\in SL_2(\Z)\backslash \M_n} f|_{-k} (\gamma\mu)
=\sum_{\mu\in SL_2(\Z)\backslash \M_n} f|_{-k} \gamma |_{-k}\mu
&=\sum_{\mu\in SL_2(\Z)\backslash \M_n} (f+ R_f^{-k}(\gamma))|_{-k}\mu \\
\in \sum_{\mu\in SL_2(\Z)\backslash \M_n} f|_{-k} \mu + \X.
\end{align*}
This fact shows that $T(n)f$ is determined modulo the space $\X$ for another choice $\{\gamma\mu\}$
of the representatives for $SL_2(\Z)\backslash \M_n$.
This observation, however, proves also that

\begin{lem}
Let $f\in M_{-k}(SL_2(\Z),\X)$. Then for $\gamma\in SL_2(\Z)$ we have
\begin{equation}
T(n)f|_{-k}\gamma \equiv T(n)f \mod \X
\end{equation}
for any choice of a system of representatives $\{\mu\}$ for $SL_2(\Z)\backslash \M_n$.
\qed
\end{lem}

This lemma shows that for $n\in\N$ the operator $T(n)$ defines a well-defined linear endomorphism of $M_{-k}(SL_2(\Z),\X)$. We call $T(n)$ the \emph{Hecke operator} of index $n$ (acting on the automorphic integrals of negative weight). Similarly to the classical case, we have the following

\begin{prop}
The Hecke operator $T(n)\, (n=1,2,\ldots)$ on the space $M_{-k}(SL_2(\Z),\X)$ possesses the following properties.
\begin{itemize}
\item[{\upshape (i)}]
The operator $T(n)$ has the following expression.
\begin{equation}
(T(n)f)(\tau)
:=n^{-k-1}\sum_{\substack{a\ge1\\ ad=n\\ 0\le b<d}}d^{k}f\Bigl(\frac{a\tau+b}d\Bigr)
=\frac1n\sum_{\substack{a\ge1\\ ad=n\\ 0\le b<d}}a^{-k}f\Bigl(\frac{a\tau+b}d\Bigr).
\end{equation}
\item[{\upshape (ii)}]
Let $f(\tau)= \sum_{\el=0}^\infty \lambda(\el) q^\el\, (q=e^{2\pi i\tau})$. Then
\begin{equation}
(T(n)f)(\tau)=\sum_{\el=0}^\infty \Bigg(\sum_{d\mid(n,\el)} d^{-k-1}\lambda\Big(\frac{n\el}{d^2}\Big)\Bigg)
q^{\el}.
\end{equation}
In particular, the space of cusp forms $C_{-k}(SL_2(\Z),\X)$ is stable under $T(n)$.
\item[{\upshape (iii)}]
Let $m, n\in \N$. Then
\begin{equation}
T(n)T(m)=\sum_{d\mid(n,m)} d^{-k-1}T(nm/d^2)=T(m)T(n).
\end{equation}
In particular, $T(n)T(m)=T(nm)$ whenever $(n,m)=1$.
\end{itemize}
\end{prop}
\begin{proof}
The proof can be done in the same way for the classical case.
Actually, since every matrix $\mu\in \M_n$ can be made upper triangular by multiplying it on the left
by $\gamma \in SL_2(\Z)$, we have a system of representatives for $SL_2(\Z)\backslash \M_n$ as $\pm \mat{a & b+dr \\ 0 & d}\, (ad=n)$ with $a>0$ and $0\leq b <d$. With this choice of representatives, by the definition \eqref{Hecke operator}, we have the expression (i). Using the elementary relation
\begin{equation*}
\sum_{b=0}^{d-1}f\Bigl(\frac{a\tau+b}d\Bigr)= d \sum_{m=0}^\infty\lambda(md)q^{md},
\end{equation*}
we have the formula (ii) from (i). We notice that the constant term in $q$ equals $\sigma_{-k-1}(d)\lambda(0)$, whence the $C_{-k}(SL_2(\Z),\X)$ is stable under $T(n)$. The last assertion (iii)
can be deduced from the formula (ii) by computation. This completes the proof.
\end{proof}

We now show the differential Eisenstein series $\dG{-2k} \in M_{-2k}(SL_2(\Z),\C(\tau))$ is a joint eigenfunction of $T(n)$ for all $n\in\N$.
\begin{lem}
We have
\begin{equation*}
(T(n)\dG{-2k})(\tau)=\sigma_{-2k-1}(n)\dG{-2k}(\tau)
\end{equation*}
for each $n,k\in\N$.
\end{lem}

\begin{proof}
Consider the function
\begin{equation*}
F_n(s,\tau)
=n^{s-1}\sum_{\substack{a\ge1\\ ad=n\\ 0\le b<d}}d^{-s}G\Bigl(s,\frac{a\tau+b}d\Bigr)
=\frac1n\sum_{\substack{a\ge1\\ ad=n\\ 0\le b<d}}a^{s}G\Bigl(s,\frac{a\tau+b}d\Bigr).
\end{equation*}
Then we observe
\begin{align*}
\ds{F_n(s,\tau)}{-2k}
&=\frac1n\sum_{\substack{a\ge1\\ ad=n\\ 0\le b<d}}
\biggl\{
\ds{a^{s}}{-2k}G\Bigl(-2k,\frac{a\tau+b}d\Bigr)
+a^{-2k}\ds{G\Bigl(s,\frac{a\tau+b}d\Bigr)}{-2k}
\biggr\} \\
&=\frac1n\sum_{\substack{a\ge1\\ ad=n\\ 0\le b<d}}
a^{-2k}\dG{-2k}\Bigl(\frac{a\tau+b}d\Bigr)
=(T(n)\dG{-2k})(\tau).
\end{align*}
On the other hand, we have
\begin{equation*}
F_n(s,\tau)
=n^{s-1}\sum_{\substack{a\ge1\\ ad=n\\ 0\le b<d}}
\psum_{m,l\in\Z}\bigl(dm+(a\tau+b)l\bigr)^{-s}
=\sigma_{s-1}(n)G(s,\tau),
\end{equation*}
and hence
\begin{align*}
\ds{F_n(s,\tau)}{-2k}
&=\ds{\sigma_{s-1}(n)}{-2k}G(-2k,\tau)+\sigma_{-2k-1}(n)\ds{G(s,\tau)}{-2k} \\[1em]
&=\sigma_{-2k-1}(n)\dG{-2k}(\tau).
\end{align*}
Thus we have the lemma.
\end{proof}

This lemma implies again that $\dG{-2k}$ can have the Fourier series expansion as
\begin{equation*}
\dG{-2k}(\tau)=\dG{-2k}(i\infty) + C_k \sum_{n=1}^\infty\sigma_{-2k-1}(n)q^n, \quad(q=e^{2\pi i \tau}),
\end{equation*}
for some constant $C_k$. From Theorem \ref{Lambert series expansion of DEisen} one finds that $C_k=2\frac{(-1)^k(2k)!}{(2\pi)^{2k}}$.
We define the \emph{normalized} differential Eisenstein series $\dE{-2k}$ of weight $-2k$ as
\begin{align*}
\dE{-2k}(\tau)&:=\frac2{C_k\zeta(2k+1)}\dG{-2k}(\tau)\\
&=1 + \frac{2}{\zeta(2k+1)}\sum_{n=1}^\infty\sigma_{-2k-1}(n)q^n
\end{align*}
Then the associated $L$-function of $\dE{-2k}$ is given by
\begin{equation*}
L(\dE{-2k}, s)=\sum_{n=1}^\infty\frac{\sigma_{-2k-1}(n)}{n^s}=\zeta(s)\zeta(s+2k+1).
\end{equation*}
Namely $\dE{-2k}$ is a Hecke form (see \cite{Helleg2002}).
We observe in particular that $L(\dE{-2k}, s)$ has a unique pole at $s=1$, while there is no pole at $s=-2k$.
Notice that since $L(E_{2k}, s)=\zeta(s)\zeta(s-2k+1)$,
$E_{2k}$ being the classical Eisenstein series of weight $2k$ for $SL_2(\Z)$,
$L(E_{2k}, s)$ has a unique pole at $s=2k$
but not at $s=1$ for $k>1$.

Further, the completed $L$-function
\begin{equation*}
\Xi_{-2k}(s)=\xi(s)\xi(s+2k+1),\qquad
\xi(s)=\pi^{-s/2}\Gamma\Bigl(\frac s2\Bigr)\zeta(s)
\end{equation*}
satisfies the functional equation
\begin{equation*}
\Xi_{-2k}(-2k-s)=\Xi_{-2k}(s).
\end{equation*}

\begin{rem}
Note that the function $\Xi_{-2k}(s)$ is meromorphic but not entire. It would be interesting to study a Hecke-Weil type theorem about the correspondence between negative weight automorphic integrals and their $L$-functions (Euler products).
\end{rem}

\section{Periodic Eichler cohomology for automorphic integrals}\label{Eichler}

We construct a cochain complex from the period functions of negative weight periodic automorphic integrals.
Let us fix an integer $m$.
Denote by $\congsubgp$ a congruence subgroup of level $N$,
and $\chi$ a multiplicative character of $\congsubgp$ such that $\chi(T^N)=1$.
Suppose that $\X$ is a $\congsubgp$-submodule of the space $F(\uhp)$ of functions on $\uhp$
via the action $f\big|_m\gamma$ ($\gamma\in \congsubgp$).

\subsection{First cohomology}

Let $C^1(\congsubgp,\X)$ be the space of all maps from $\congsubgp$ to $\X$.
We call $R_\chi\in C^1(\congsubgp,\X)$ a \emph{$($twisted\/$)$ $1$-cocycle} with weight $\chi$ if it satisfies
\begin{equation*}
R_\chi(\gamma_1\gamma_2)=\chi(\gamma_1)R_\chi(\gamma_2)+R_\chi(\gamma_1)\big|_m\gamma_2.
\end{equation*}
Notice that $R_\chi(I)=0$ if $R_\chi$ is a $1$-cocycle.
We denote by $\cZ^1(\congsubgp,\X)$ the set of all (twisted) $1$-cocycles
(Here and after, to avoid complication, we do not specify the character $\chi$ in notation).
Obviously $\cZ^1(\congsubgp,\X)$ is a subspace of $C^1(\congsubgp,\X)$.

Define the element $\cobd f \in C^1(\congsubgp,\X)$ for $f\in\X$ by
\begin{equation*}
(\cobd f)(\gamma)=f\big|_m\gamma-\chi(\gamma)f\quad(\gamma\in \congsubgp).
\end{equation*}

By a similar calculation as in \eqref{eq:proof-of-1-cocycle-condition} shows the
\begin{lem}
$\cobd f\in \cZ^1(\congsubgp,\X)$ for each $f\in\X$.
\qed
\end{lem}

Define the subgroup $\cB^1(\congsubgp,\X)$ of $\cZ^1(\congsubgp,\X)$ by
\begin{equation*}
\cB^1(\congsubgp,\X)=\im\cobd=\Set{\cobd f}{f\in\X}.
\end{equation*}
We call an element of $\cB^1(\congsubgp,\X)$ by a \emph{$($twisted\/$)$ $1$-coboundary}.
The quotient group defined by
\begin{equation*}
\cH^1(\congsubgp,\X):=\cZ^1(\congsubgp,\X)/\cB^1(\congsubgp,\X)
\end{equation*}
is called the first \emph{Eichler cohomology group} of weight $m$ for the $\congsubgp$-module $\X$.

\subsubsection*{Periodic cohomology}

Assume that $\congsubgp$ is a congruence subgroup of level $N$.
For $f\in M_m(\congsubgp,\X)$, put
\begin{equation*}
R^m_f(\gamma):=f\big|_m\gamma-\chi(\gamma)f\quad(\gamma\in \congsubgp).
\end{equation*}
It is easy to check that $R^m_f$ gives an element in $\cZ^1(\congsubgp,\X)$ (see \eqref{eq:1-cocycle-from-RMF}).
We notice that $R^m_f(T^N)=f(\tau+N)-f(\tau)=0$ by definition.

\begin{dfn}
Define
\begin{align*}
\pZ^1(\congsubgp,\X)&:=\set{R\in\cZ^1(\congsubgp,\X)}{R(T^N)=0}, \\
\pB^1(\congsubgp,\X)&:=\set{\cobd f\in\cB^1(\congsubgp,\X)}{f\in\X,\ f(\tau+N)=f(\tau)}\bigl(\subset\pZ^1(\congsubgp,\X)\bigr), \\
\pH^1(\congsubgp,\X)&:=\pZ^1(\congsubgp,\X)/\pB^1(\congsubgp,\X).
\end{align*}
We call $\pH^1(\congsubgp,\X)$ the first \emph{periodic} Eichler cohomology group (of weight $m$).
\end{dfn}

Let $f\in M_m(\congsubgp,\X)$.
We see that $R^m_f\in\pZ^1(\congsubgp,\X)$.
If $R^m_f\in\pB^1(\congsubgp,\X)$, then there exists some $g\in\X$ such that $R^m_f=\cobd g$ and $g(\tau+N)=g(\tau)$.
It follows that $(f-g)\big|_m\gamma=f-g$, which implies that $f-g\in M_m(\congsubgp)$
and hence $g\in M_m(\congsubgp,\X)+M_m(\congsubgp)=M_m(\congsubgp,\X)$.
Thus we have an injection
\begin{equation*}
M_m^*(\congsubgp,\X)\hookrightarrow\pH^1(\congsubgp,\X),
\end{equation*}
where we put
\begin{equation}
M_m^*(\congsubgp,\X):=M_m(\congsubgp,\X)/(\X\cap M_m(\congsubgp,\X)+M_m(\congsubgp)).
\end{equation}
If $m<0$ and $\X\subset\C(\tau)$, then we have
\begin{equation*}
M_m^*(\congsubgp,\X)=
\begin{cases}
M_m(\congsubgp,\X)/(\text{constants}) & 1\in\X, \\
M_m(\congsubgp,\X) & 1\notin\X.
\end{cases}
\end{equation*}
In particular, we have the inequality
\begin{equation}\label{eq:dim ineq 1}
\dim M_m^*(\congsubgp,\X)\le\dim\pH^1(\congsubgp,\X).
\end{equation}
We also have
\begin{equation}\label{eq:dim ineq 2}
\dim\pH^1(\congsubgp,\X)\le\dim\cH^1(\congsubgp,\X)-1
\end{equation}
when $\X=\C(\tau)$ or $\X=\C[\tau]_{-m}$ (see Lemma 17 in \cite{KW2012RIMS}).

Notice that $1=j(\gamma,\tau)^{-2}\{1-(1-j(\gamma,\tau)^2)\}\in M_{-2}(\hecke, \C[\tau]_2)$.
By \eqref{eq:dim ineq 1} and \eqref{eq:dim ineq 2}, we have
\begin{equation*}
1\leq \dim_{\C} M_{-2}^*(\hecke, \C[\tau]_2)
\leq \dim_{\C} H_{[-2]}^1(\Gamma(2), \C[\tau]_2)-1
\end{equation*}
since $\iGt_1 \in M_{-2}(\hecke, \C[\tau]_2)$.
It is known in \cite{G1961} that
\begin{equation*}
H_{[-2k]}^1(\Gamma(2), \C[\tau]_{2k})\cong
M_{2k+2}(\Gamma(2))\oplus C_{2k+2}(\Gamma(2)),
\end{equation*}
$C_{2k+2}(\Gamma(2))$ being the space of cusp forms of weight $2k+2$ for $\Gamma(2)$.
Since $\dim_\C M_{4}(\Gamma(2))=2$ and $\dim_\C C_{4}(\Gamma(2)) =0$ (see, e.g.\ \cite{S1971}),
one concludes that $\dim_{\C} M_{-2}^*(\Gamma(2), \C[\tau]_2) =1$.
Thus we have the

\begin{cor}\label{RMF2}
$M_{-2}(\hecke, \C[\tau]_2) = M_{-2}(\Gamma(2), \C[\tau]_2)= \C\cdot \iGt_1 \oplus \C\cdot 1$.
\qed
\end{cor}

This shows that the $M_{-2}(\hecke, \C[\tau]_2)$ is essentially given by $w_4$, i.e. the special value $\zeta_Q(4)$.

The following lemma is obvious.
\begin{lem}\label{GenRel}
Assume that a congruence subgroup $\congsubgp$ of level $N$ contains $S$.
If $f\in M_{-k}(\congsubgp,\X)$ we have
\begin{equation*}
R_f^{-k}(T^N)(\tau)=0,
\quad
R_f^{-k}(S)(S\tau)= -\tau^{-k} R_f^{-k}(S)(\tau).
\end{equation*}
In particular, $R_f^{-k}(\gamma)\in\pZ^1(\congsubgp,\X)$.
From the cocycle condition, one knows that $R \in\pZ^1(\congsubgp,\X)$ is determined by
the double coset of $\congsubgp_{\infty}=\generators{T^N}$:
\begin{equation*}
R(T^N\gamma)(\tau)=R(\gamma)(\tau), \quad
R(\gamma T^N)(\tau)=R(\gamma)(T^N \tau).
\end{equation*}
\end{lem}

\subsection{Cochain complex}

Let us put
\begin{equation*}
C^n=C^n(\congsubgp,\X):=\Map(\congsubgp^n,\X),
\end{equation*}
for $n=1,2,3,\dots$ and $C^0=C^0(\congsubgp,\X):=\X$.
For an $n$-tuple $\bgamma=(\gamma_1,\dots,\gamma_n)\in \congsubgp^n$, we define
\begin{equation*}
\truncate_j\bgamma:=(\gamma_1,\dots,\gamma_{j-1},\gamma_{j+1},\dots,\gamma_n),\quad
\contract_j\bgamma:=(\gamma_1,\dots,\jth{j}{\gamma_{j+1}\gamma_j},\dots,\gamma_{n})
\quad(j=1,2,\dots,n)
\end{equation*}
for convenience.
Define the linear operator $\cobd^n\colon C^n\to C^{n+1}$ by
\begin{equation}\label{eq:def_of_der}
(\cobd^n f)(\bgamma)
:=f(\truncate_1\bgamma)\big|_m\gamma_1
+(-1)^{n+1}\chi(\gamma_{n+1})f(\truncate_{n+1}\bgamma)
+\sum_{j=1}^n(-1)^{j}f(\contract_j\bgamma)
\end{equation}
for $f\in C^n$ and $\bgamma=(\gamma_1,\dots,\gamma_{n+1})\in \congsubgp^{n+1}$.

Although we have given the proof of the following fact in \cite{KW2012RIMS}, we give here a shorter one.
\begin{lem}
$\cobd^{n+1}\circ\cobd^n=0$.
\end{lem}

\begin{proof}
Take arbitrary $f\in C^n$.
Let $\bgamma=(\gamma_1,\dots,\gamma_{n+2})\in \congsubgp^{n+2}$.
One has
\begin{align*}
(\cobd^nf)(\truncate_k\bgamma)
&=f(\truncate_1\truncate_k\bgamma)\big|_m\truncate_k\bgamma_1
+(-1)^{n+1}\chi(\truncate_k\bgamma_{n+1})f(\truncate_{n+1}\truncate_k\bgamma)
+\sum_{j=1}^{n}(-1)^jf(\contract_j\truncate_k\bgamma),
\allowdisplaybreaks\\
(\cobd^nf)(\contract_j\bgamma)
&=f(\truncate_1\contract_j\bgamma)\big|_m\contract_j\bgamma_1
+(-1)^{n+1}\chi(\contract_j\bgamma_{n+1})f(\truncate_{n+1}\contract_j\bgamma)
+\sum_{l=1}^n(-1)^lf(\contract_l\contract_j\bgamma)
\end{align*}
for $1\le k\le n+2$ and $1\le j\le n+1$,
where $\truncate_k\bgamma_r$, $\contract_j\bgamma_r$ are the $r$-th entry of
$\truncate_k\bgamma$, $\contract_j\bgamma$.
We have
\begin{equation*}
\truncate_k\bgamma_1=\begin{cases}
\gamma_2 & k=1, \\
\gamma_1 & k>1,
\end{cases}
\qquad
\truncate_k\bgamma_{n+1}=\begin{cases}
\gamma_{n+1} & k=n+2, \\
\gamma_{n+2} & k<n+2,
\end{cases}
\end{equation*}
and
\begin{equation*}
\contract_j\bgamma_1=\begin{cases}
\gamma_2\gamma_1 & j=1, \\
\gamma_1 & j>1,
\end{cases}
\qquad
\contract_j\bgamma_{n+1}=\begin{cases}
\gamma_{n+2}\gamma_{n+1} & j=n+1, \\
\gamma_{n+2} & j<n+1.
\end{cases}
\end{equation*}
Using these, we have
\begin{equation}\label{eq:derder}
\begin{split}
((\cobd^{n+1}\circ\cobd^n)f)(\bgamma)
&=(\cobd^nf)(\truncate_1\bgamma)\big|_m\gamma_1
+(-1)^n\chi(\gamma_{n+2})(\cobd^nf)(\truncate_{n+2}\bgamma)
+\sum_{j=1}^{n+1}(-1)^j(\cobd^nf)(\contract_j\bgamma) \\
&=\sum_{j=1}^{n+1}(-1)^j\sum_{l=1}^n(-1)^lf(\contract_l\contract_j\bgamma),
\end{split}
\end{equation}
which surely vanishes.
\end{proof}

Thus we can now define cocycles and coboundaries
\begin{equation*}
\cZ^n(\congsubgp,\X):=\ker\cobd^n,\qquad
\cB^n(\congsubgp,\X):=\im\cobd^{n-1}
\end{equation*}
in $C^n(\congsubgp,\X)$ and the cohomology group
\begin{equation*}
\cH^n(\congsubgp,\X):=\cZ^n(\congsubgp,\X)/\cB^n(\congsubgp,\X)
\end{equation*}
for each $n=0,1,2,\dots$.

The following is a special case of the result by Gunning \cite{G1961}.
\begin{prop}
$\cH^n(\congsubgp,\C[\tau]_{-m})=0$ if $n>1$ and $m<0$.
\qed
\end{prop}

\subsubsection*{Periodic cohomology}

We define the groups $\pZ^n(\congsubgp,\X), \pB^n(\congsubgp,\X)$ and $\pH^n(\congsubgp,\X)$ as follows:
\begin{align*}
\pC^n(\congsubgp,\X)&:=\set{f\in C^n(\congsubgp,\X)}{f(T^{k_1N},\dots,T^{k_nN})=0,\ k_1,\dots,k_n\in\Z}, \\
\pZ^n(\congsubgp,\X)&:=\cZ^n(\congsubgp,\X)\cap\pC^n(\congsubgp,\X), \\
\pB^n(\congsubgp,\X)&:=\cB^n(\congsubgp,\X)\cap\pC^n(\congsubgp,\X), \\
\pH^n(\congsubgp,\X)&:=\pZ^n(\congsubgp,\X)/\pB^n(\congsubgp,\X).
\end{align*}

\begin{prop}
If $\cH^n(\congsubgp,\X)=0$, then $\pH^n(\congsubgp,\X)=0$.
In particular, $\pH^n(\congsubgp,\C[\tau]_{-m})=0$ if $n>1$ and $m<0$.
\end{prop}

\begin{proof}
This is obvious because $\cZ^n(\congsubgp,\X)=\cB^n(\congsubgp,\X)$ readily implies $\pZ^n(\congsubgp,\X)=\pB^n(\congsubgp,\X)$ by definition.
\end{proof}

\begin{prob}
When $\cH^n(\congsubgp,\X)$ vanishes?
How about the periodic case $\pH^n(\congsubgp,\X)$?
\end{prob}

\subsubsection*{Zero-dimensional cohomology}

The group $\cH^0(\congsubgp,\X)$ is easily described.
Indeed, since $\cB^0(\congsubgp,\X)=\im\cobd^{-1}=0$, we have
\begin{equation}
\cH^0(\congsubgp,\X)=\cZ^0(\congsubgp,\X)
=\set{f\in\X}{f(\gamma\tau)=\chi(\gamma)j(\gamma,\tau)^mf(\tau)~(\forall\gamma\in \congsubgp)}.
\end{equation}
If $m<0$, $\X=\C[\tau]_{-m}$, $\chi$ is trivial and $\congsubgp$ is a congruent subgroup of level $N$,
then we have
\begin{gather*}
\cH^0(\congsubgp,\C[\tau]_{-m})
\subset
\Set{f\in\C[\tau]_{-m}}{f(\tau+N)=f(\tau),\ f\Bigl(\frac1{N\tau+1}\Bigr)=(N\tau+1)^mf(\tau)}=0,
\end{gather*}
from which it also follows that $\pH^0(\congsubgp,\C[\tau]_{-m})=0$.

\begin{ackn}
The second author thanks Evgeny Verbitskiy for making him aware of the paper by
F.~Rodriguez Villegas \cite{RV1999}.
\end{ackn}

\begin{flushleft}
\bigskip

Kazufumi Kimoto \par
Department of Mathematical Sciences, \par
University of the Ryukyus \par
1 Senbaru, Nishihara, Okinawa 903-0213 JAPAN \par
\texttt{kimoto@math.u-ryukyu.ac.jp} \par

\bigskip

Masato Wakayama \par
Institute of Mathematics for Industry,\par
Kyushu University \par
744 Motooka, Nishi-ku, Fukuoka 819-0395 JAPAN \par
\texttt{wakayama@imi.kyushu-u.ac.jp}
\end{flushleft}

\end{document}